\documentclass[12pt, a4paper]{amsart}
\usepackage{amsmath,amssymb,amscd,amsfonts}

\usepackage{ifthen}
\usepackage[T1]{fontenc}
\usepackage[utf8]{inputenc}
\usepackage[all]{xy}
\usepackage{graphicx}
\usepackage{enumerate}
\usepackage{xspace}
\usepackage{pdfsync}
\usepackage{epic}
\usepackage{dsfont}

\newtheorem{theorem}{Theorem}[section]
\newtheorem{remark}[theorem]{Remark}
\newtheorem{definition}[theorem]{Definition}
\newtheorem{lemma}[theorem]{Lemma}

\newtheorem{proposition}[theorem]{Proposition}
\newtheorem{conjecture}[theorem]{Conjecture}
\newtheorem{corollary}[theorem]{Corollary}
\newtheorem{example}[theorem]{Example}
\newtheorem{question}[theorem]{Question}

\newcommand\Q{{\mathbb{Q}}}

\def\Supp{\mathop{\rm Supp}\nolimits}

\def\cO{{\mathcal O}}

\def\cE{{\mathcal E}}

\def\cR{{\mathcal R}}
\def\cG{{\mathcal G}}
\def\cH{{\mathcal H}}
\def\cF{{\mathcal F}}

\def\cQ{{\mathcal Q}}

\def\cM{{\mathcal M}}

\let\ol=\overline

\let\wh=\widehat
\def\bQ{{\mathbb Q}}
\def\bC{{\mathbb C}}
\def\bR{{\mathbb R}}

\def\bP{{\mathbb P}}

\begin{document}
\title[]{ORbifold Slope Rational-connectedness}

\author{Fr\'ed\'eric Campana}
\address{Universit\'e Lorraine \\
 Institut Elie Cartan\\
Nancy \\ Institut Universitaire de France\\ and KIAS scholar,\ KIAS\\
85 Hoegiro, Dongdaemun-gu\\
Seoul 130-722, South Korea,}

\email{frederic.campana@univ-lorraine.fr}

\date{\today}

\maketitle

\tableofcontents


\begin{abstract} We define, for smooth projective orbifold pairs $(X,D)$ notions of `slope Rational connectedness', and of orbifold `slope Rational quotient' . These notions extend to this larger context the classical notions of rationally connected manifold and `rational quotient' (sometimes called `MRC fibration'). Our notions and proofs work entirely in characteristic zero, and are based on the consideration of foliations with minimal positive slope with respect to some suitable movable class. The existence of covering or connecting families of `orbifold rational curves' is indeed presently unknown in the orbifold context, in situations analogous to the classical case $D=0$. By contrast, the notions we introduce here, are checkable in practice and can certainly be used to show general properties expected from the existence of connecting families of `orbifold rational curves'. The proofs given here in the orbifold context provide new proofs in the classical case where $D=0$, since the classical proofs did not seem to adapt, with the presently existing techniques, to this broader context.
\end{abstract}

\section{Introduction}

Let $X$ be any connected complex projective manifold. Its \emph{rational quotient map} $\rho:X\dasharrow R$ splits $X$ into its antithetical parts: rationally connected (the fibres), and non-uniruled (the base $R$). This splitting can aternatively be defined according to the slope-positivity of the cotangent bundle relatively to movable classes, without referring to rational curves. This will be the point of view adopted here.

Indeed, rational-connectedness is also characterised by the existence of a movable class $\alpha$ on $X$ such that $\mu_\alpha(\cQ)> 0$, for any quotient $\cQ$ of the tangent bundle $T_X$ (cf. Proposition \ref{prc} below). Replacing $T_X$ by the orbifold tangent bundle of a smooth orbifold pair $(X,D)$ leads to define the notion of \emph{slope Rationally Connected orbifolds}. 

\noindent The main objective of the present text is indeed to introduce this notion by several definitions shown to be equivalent, and to construct the rational quotient splitting in the category of smooth orbifold pairs $(X,D)$. This permits to formulate and extend to arbitrary smooth orbifold pairs results previously restricted to either manifolds without orbifold structure, or to orbifold pairs with pseudo-effective canonical bundles. 

One of the main problems of birational geometry is to `decompose' functorially (quasi)-projective manifolds, by suitable fibrations, into parts having a `signed' canonical bundle, the rational quotient map being the first step of such a decomposition. This decomposition is (conditionally in an orbifold version of the $C_{n,m}$ conjecture) achieved in \cite{Ca07}. The `parts' in the decomposition are, however, not manifolds, but orbifold pairs with a 'signed' canonical bundle. Understanding the structure of general (quasi)-projective manifolds thus requires the consideration of the larger category of orbifold pairs. They also appear in questions seemingly independent from these structural considerations such as the solution of the Shafarevich-Viehweg `hyperbolicity conjecture' in \cite{CP15}.

As said above, our considerations do not refer to rational curves. A crucial advantage is that only characteristic zero arguments are used. Moreover, many basic properties of rationally connected manifolds can be obtained using negativity properties of their cotangent or tensor bundles, bypassing the consideration of rational curves on them (\cite{BC}, \cite{CR}).

\medskip

Let us now state more precisely the main results of the text. We need to recall first some of the key notions. 
\medskip

\smallskip

Let $(X,D)$ be a `smooth projective orbifold pair. In our previous text \cite{CP15} (see also \cite{Ca04} and \cite{Ca07}), we introduced orbifold (co)tangent sheaves for general smooth orbifold pairs as above, by lifting them to some (or any) Kawamata cover $\pi:X'\to (X,D)$ adapted to $(X,D)$. We thus obtained dual locally-free sheaves $\pi^*(T(X,D))\subset \pi^*(TX)$ and $\pi^*(\Omega^1(X,D))\supset \pi^*(\Omega^1_X)$. Their determinants are the $\pi$-liftings of the usual anticanonical and canonical $\Bbb Q$-bundles $\pm(K_X+D)$.

Let now $\alpha$ be a movable class on $X$ (see \cite{CPe} for this notion, and the ones to follow, which were introduced there). 
The notions of slope, (semi-)stability, Harder-Narasimhan filtrations of torsion-free sheaves $\cE$ on $X$, with respect to any movable class $\alpha$, are defined exactly as in the classical case of polarisations, with the same properties\footnote{With the only exception that the restriction theorem of Mehta-Ramanathan remains an open problem in this broader context.}. In particular, the property $\mu_{\alpha,min}(\cE)>0$ means that any quotient $\cQ$ of $\cE$ has a strictly positive $\alpha$-slope. 
\medskip

\noindent Our first main result is:

\begin{theorem}\label{torc} Let $(X,D)$ be a smooth orbifold pair. The following properties are equivalent:
\begin{enumerate}

\item[{\rm 1.}] For any dominant rational map $f:X\dasharrow Z$  with connected fibres, the orbifold base $(Z,D_Z)$ of any `neat model' of $(f,D)$ has a canonical bundle $K_Z+D_Z$ which is not pseudo-effective, if $dim(Z)>0$. 

\item[{\rm 2.}] For any ample line bundle $A$ on $X$, and some $m_A$, one has: 
$h^0(X',\pi^*(\otimes^m( \Omega^1(X,D)))\otimes \pi^*(A))=0, \forall m\geq m_A$.

\item[{\rm 2'.}] For any ample $A$ on $X$, and some $m_A$, one has $\forall m\geq m_A$: 

$h^0(X',Sym^m(\wedge^p(\pi^*(\Omega^1(X,D)))\otimes \pi^*(A))=0, \forall p>0$.

Versions checked on $X$ also exist:

\item[{\rm 3.}] For any ample line bundle $A$ on $X$, and some $m_A$, one has: 
$h^0(X,[\otimes^m]( \Omega^1(X,D))\otimes \pi^*(A))=0, \forall m\geq m_A$.

\item[{\rm 3'.}] For any ample $A$ on $X$, and some $m_A$, one has $\forall m\geq m_A$: 
$h^0(X,[Sym^m(\wedge^p)](\Omega^1(X,D))\otimes \pi^*(A))=0, \forall m\geq m_A,\forall p>0$.

\item[{\rm 4.}] One has: $\mu_{\pi^*(\alpha),min}(\pi^*(TX,D))>0$, for some movable class $\alpha$ on $X$. Moreover, the class $\alpha$ can be chosen to be `movable-big' (i.e: interior to the movable cone of $X$).
\end{enumerate}

\end{theorem}

The sheaves $[\otimes^m]( \Omega^1(X,D))$ and $[Sym^m(\wedge^p)](\Omega^1(X,D))$ are locally free sheaves on $X$ (not $X'$), called the `integral parts' of $\otimes^m( \Omega^1(X,D))$ and $Sym^m(\wedge^p)(\Omega^1(X,D))$ respectively (hence the notation), and defined in \S.\ref{sipot}. They are expected (see \S.\ref{Ginv}) to `approximate' with an accuracy tending to $0$ the corresponding `fractional' sheaves when $m\to +\infty$.

The main implication (sketched at the end of the introduction) of the Theorem \ref{torc} is that 1$\Longrightarrow$ 4.

\begin{definition} \label{def OSRC}
A smooth projective orbifold pair satisfying the equivalent properties of Theorem \ref{torc} will be said to be `slope-Rationally Connected', or $sRC$ for short. 
\end{definition}

\medskip

\noindent The Harder-Narasimhan filtration of $\pi^*(T(X,D))$ and its slopes relative to $\pi^*(\alpha)$ are actually independent of the choice of the ramified cover
$\pi:X'\to X$, because of the conceptually crucial\footnote{The expectation is indeed that the geometry of $(X,D)$ is a well-defined concept, independent on the auxiliary constructions used to define it.} complement to our previous text \cite{CP15}:

\begin{theorem}\label{tind'} Let $(X,D)$ be a smooth projective orbifold pair, and $\pi:X'\to X$ a cover adapted to $(X,D)$. Let $\alpha$ be a movable class on $X$, and $\pi^*(\alpha)$ be its (movable) inverse image on $X'$. Let $HN_{\pi^*(\alpha)}(\pi^*(T(X,D)))$ be the Harder-Narasimhan filtration of $\pi^*(T(X,D))$ relative to $\pi^*(\alpha)$. 

There exists a filtration denoted $HN_{\alpha}(T(X,D))$ of $TX$ such that: $HN_{\pi^*(\alpha)}(\pi^*(T(X,D)))^{sat}=(\pi^*(HN_{\alpha}(T(X,D))))^{sat}$, the saturation being taken in $\pi^*(TX)$. 

Moreover, the filtration $HN_{\alpha}(T(X,D))$, as well as the $\pi^*(\alpha)$-slopes of $HN_{\pi^*(\alpha)}(\pi^*(T(X,D)))$, do not depend on the choice of the adapted cover $\pi$.
\end{theorem}

Said otherwise: the slopes and distributions on $(X-Supp(D))$ induced by the orbifold divisor $D$ do not depend on $\pi$, only on $D$ (and $\alpha$ of course). See Theorem \ref{tind} and its corollary \ref{cind} for details.

\medskip

\noindent Theorem \ref{torc} is (except for the Property 1) the exact analogue of the classical case when $D=0$, since we showed in \cite{CP15} 
 (as a direct consequence of the classical results \cite{MiMo}, \cite{KMM}, \cite{GHS}, \cite{BDPP}, \cite{CDP}):

\begin{proposition}\label{prc} For a smooth projective connected complex manifold $X$, the following $4$ properties are equivalent:
\begin{enumerate}

\item[{\rm 1.}] $X$ is rationally connected (i.e: any two points can be connected by some rational curve)

\item[{\rm 2.}] For any ample line bundle $A$ on $X$, there is an $m(A)>0$ such that $h^0(X,\otimes^m( \Omega_X^1)\otimes A)=0, \forall m\geq m(A)$.

\item[{\rm 2'.}] For any ample line bundle $A$ on $X$, there is an $m'(A)>0$ such that $h^0(X,Sym^m(\Omega_X^p)\otimes A)=0, \forall m\geq m'(A), \forall p>0$.

\item[{\rm 3.}] For any dominant rational map $f:X\dasharrow Z$, with $Z$ smooth and $dim(Z)>0$, $K_Z$ is not pseudo-effective.

\item[{\rm 4.}] $\mu_{\alpha,min}(TX)>0$ for some movable class $\alpha$ on $X$ (see below for these notions).
\end{enumerate}
\end{proposition}

\smallskip

\noindent The `Slope Rationally Connected Quotient' in the category of smooth orbifold pairs takes the following form, entirely similar to the classical `Rational quotient' recalled above (see \cite{Ca92}, \cite{KMM}):

\begin{theorem}\label{tRQ'} Let $(X,D)$ be smooth, complex projective and connected. There exists an orbifold birational model\footnote{See Definition \ref{deforbmorph} 
 below.} $(X',D')$, and an orbifold morphism $\rho:(X',D')\to (R,D_R)$ which is a fibration onto its (smooth) orbifold base\footnote{See Definition \ref{dob} below.} $(R,D_R)$ with the following two properties:
\begin{enumerate}

\item[{\rm1.}] Its smooth orbifold fibres $(X_r,D_r)$ are $sRC$.

\item[{\rm 2.}] $K_R+D_R$ is pseudo-effective.
\end{enumerate}
This fibration is unique, up to orbifold birational equivalence. 
\end{theorem}

A similar orbifold fibration (with orbifold fibres having $\kappa^+=-\infty$, and orbifold base $\kappa\geq 0$) was defined in \cite{Ca04}, \cite{Ca07} conditionally in either an orbifold version $C_{n,m}^{orb}$ of Iitaka's $C_{n,m}$ conjecture, or in the `non-vanishing' conjecture. The present text permits to give an unconditional (conjecturally equivalent) definition. See Section \ref{motiv} for some brief details.

\medskip

Our next statement shows that the notion of \emph{slope Rationally Connected orbifold} also permits to obtain the expected strengthening of a former result of \cite{CP15} (see Theorem \ref{tofops} below for the proof, and more details):

\begin{theorem}\label{tofops'}  Assume that $\cF_D\subset \pi^*(T(X,D))$ is a $D$-foliation, and that $\mu_{\alpha',min}(\cF_D)>0$ for some movable class $\alpha$ on $X$, where $\alpha':=\pi^*(\alpha)$. Then:

1. $\cF$ is algebraic, let $f:X\dasharrow Z$, be such that $\cF=Ker(df)$.

2. On any `neat' orbifold birational model $f':(X',D')\to Z'$ of $f$, the generic orbifold fibre $(X'_z,D'_z)$ of $f'$ is sRC.

Conversely, if $(f,D)$ posesses the above property 2, the $D$-foliation $\cF_D$ associated to it\footnote{By the construction recalled before the statement of Theorem \ref{tofops}.} has $\mu_{\alpha',min}(\cF_D)>0$ for some $\alpha$ movable on $X$, and for any Kawamata cover $\pi$ adapted to $(X,D)$.
\end{theorem}

In \cite{CP15}, Theorem 1.4, the conclusion was only that $K_{X'_z}+D'_z$ was not pseudo-effective for $z\in Z'$ generic. It was also shown when $D=0$ that $X'_z$ was rationally connected (Theorem 1.1). The absence of this notion in the orbifold case made it impossible to state (and of course to prove) anything when $D\neq 0$.

\medskip

An important example of `slope Rationally connected' smooth orbifold is the following\footnote{Answering a question of B. Claudon.}:

\begin{theorem}\label{tfsrc} Let $(X,D)$ be a smooth orbifold pair which is klt\footnote{This means that all coefficients of the components of $D$ are strictly less than $1$.}, and Fano (ie: $-(K_X+D)$ is ample on $X$). Then $(X,D)$ is slope rationally connected.
\end{theorem} 

The conclusion in general fails when $(X,D)$ is not klt, as shown by the following example of \cite{Ca07}, 6.17, p. 859: $(X:=\Bbb P^2,D)$ if $D$ is the union of $2$ distinct lines, each equipped with coefficient $1$ (or, equivalently, with multiplicity $+\infty$): see Example \ref{exfanonotsrc} for details. 

We showed in \cite{CP15} that the orbifold cotangent bundle is `birationally stable, that is: $\nu^+(X,D)=\nu(X,D):=\nu(X,K_X+D)$ when $K_X+D$ is pseudo-effective, $\nu$ being N. Nakayama's `numerical dimension' of a line bundle. The preceding `birational stability' means that $\nu(X,L)\leq \nu(X,D)$ for any line line bundle $L$ on $X$ such that $p^*(L)$ admits an injective sheaf morphism in some $\otimes^m(p^*(\Omega^1(X,D)))$, for some $m>0$, if $p:Y\to X$ is any Kawamata cover adapted to $(X,D)$. We refer to \S\ref{sbirstab} for the definitions. 

The Slope rationally connected quotient permits to describe the invariant $\nu^+(X,D)$ in the general case, as follows: $\nu^+(X,D)=\nu(R,D_R)$ (Theorem \ref{tnugen}).

We moreover show (Theorem \ref{t-knef}) that if $(X,D)$ is smooth projective with $-(K_X+D)$ nef, then $\nu^+(X,D)=\nu(R,D_R)=\kappa(R,D_R)\in \{-\infty, 0\}$. 

In the first case, $(X,D)$ is $sRC$, in the second case, there exists $m>0,p>0$ such that: $h^0(X, [Sym^m(\wedge^p)](\Omega^1(X,D))\neq 0$. When $D=0$, this was essentially shown in \cite{QZ}. 

\medskip

As already said, we make here no reference to orbifold $D$-rational curves (see Definition \ref{deforc}. These are (as in \cite{Ca10}) defined in \S\ref{sorc'} to which we refer for more details and the conjectural characterisation of slope orbifold rational connectedness (resp. uniruledness) in terms of connecting (resp. covering) families of orbifold rational curves. Let us notice that even the answer to the much simpler question below is presently unknown:

\medskip

\noindent {\bf Question:}\label{qour} Let $(X,D)$ be smooth projective with $K_X+D$ not pseudo-effective. Is $X$ covered by rational curves $C$ with $(K_X+D).C<0$?

\medskip
A very weak and partial solution of the conjectures below is given in \S\ref{sorc}, assuming a positive answer to the preceding question \ref{qour}.

\medskip

\begin{conjecture}\label{corc} (See also \S.\ref{sorc'}). 1. $K_X+D$ is not pseudoeffective if and only if $X$ is covered by $D$-rational curves.

2. $(X,D)$ is slope Rationally Connected if and only if two generic points of $X$ can be connected by a chain of $D$-rational curves (or equivalently by a single $D$-rational curve). 
\end{conjecture}

\noindent We give now a very brief description of the main implication $1\Rightarrow 3$ of theorem \ref{torc}: it requires several steps and intermediate constructions. The main ingredient in the proof is the weak version of Theorem \ref{tofops'} obtained in \cite{CP15}: if $\cF\subset \pi^*(T(X,D))$ is a $D$-foliation of positive minimal slope for $\pi^*(\alpha)$, where $\alpha$ is some movable class on $X$, then $\cF$ descends to an algebraic foliation i.e. there exists a map $f:X\dasharrow Z$ such that $\cF_X=Ker(df)$ and moreover $\alpha.(K_{X/Z}+D)<0$ (on any `neat' model of $f$), with $\dim(Z)<n:= \dim(X)$.

Our proof works by induction on $\dim(X)$. The first step is to chose $\alpha$ such that $dim(Z)=p<n$ is maximal. Then we change $\alpha$ in such a way that $f_*(\alpha)=0$, and moreover $\alpha$ is `big' on the `general' fibre $X_z$ of $f$. The minimality of the rank of $\cF$ then permits to show that $\cF$ is still the maximal destabilizing foliation associated to the new $\alpha$, and is $\alpha$-semi-stable (of positive slope). If $\cF=\pi^*(T(X,D))$, we are done. 
Otherwise, we get from the induction hypothesis a movable class $\beta$ on $Z$ such that $\mu_{p^*(\beta),min}(p^*(T(Z,D_Z)))>0,$
where $p:Z'\to (Z,D_Z)$ is a Kawamata cover adapted to $(Z,D_Z)$, the orbifold base of a `neat model' $f:(X,D)\to Z$ of our initial $f$. We then chose a movable lifting $\beta_X$ of $\beta$ to $X$ such that $f_*(\beta_X)=\beta$, and show that: $\mu_{\gamma,min}(\pi^*(T(X,D)))>0,$
if $\gamma:=\lambda.\alpha+\beta_X$ and $\lambda> 0$ is sufficiently big. It is here that the relative bigness of $\alpha$ is needed. Some additional technical difficulties arise from the fact that the image of $\pi^*(df):\pi^*(T(X,D))\to (f\circ \pi)^*(T(Z,D_Z))$ is not surjective on some $f$-vertical divisors. They are handled by using a `negativity-type' lemma.
\smallskip



{\bf Thanks.} This text was essentially conceived during a stay at KIAS as a KIAS scholar in November-December 2015. Many thanks to this institution for its support and excellent working conditions, and especially to the staff of KIAS for their efficiency and amability.

I would also like to thank Mihai P\u aun for his support and encouragements during the conception of the text, for pointing out a gap in the statement an initial version  of Theorem \ref{torcrel'}, and for the proof of Proposition \ref{currents}, on which Theorem \ref{tbig} relies.

The results of the present text deeply depend on former articles \cite{CP15} and \cite{CPe} written in collaboration with him and Thomas Peternell, respectively. 

\medskip

We now start by several technical sections needed in the proof of Theorem \ref{torc}.


\section{Orbifold morphisms}\label{som}

\subsection{Orbifold morphisms}

\noindent We first recall some notions introduced in \cite{Ca07}, \S2 and \S3, which are needed here\footnote{All definitions and properties stated here work for compact complex normal spaces, and even, more generally for proper holomorphic fibrations. But we shall use them here only in the projective case.}.

Let $(X, D)$ be an orbifold pair; this consists of a connected complex projective normal variety $X$, 
together with an `orbifold divisor' 
\begin{equation}\label{1}
D:=\sum_{F\subset X} c_D(F).F,
\end{equation} 
where the components $F$ are all the pairwise distinct irreducible (Weil) divisors on $X$, and the coefficients $\displaystyle c_D(F)$ are zero for all but a finite number of divisors $F$. The coefficients $\displaystyle c_D(F)= 1-\frac{b}{a}\in [0,1]\cap \bQ$ are, moreover, rational numbers with $a\geq b$ coprime positive integers. The quotient $\frac{a}{b}\geq 1$ is called the `multiplicity' associated to the `coefficient' $c$. It carries the geometric information as an (extended) notion of ramification order.

We will thus also write: $\displaystyle c_D(F)= 1-\frac{1}{m_D(F)}$, in terms of: $m_D(F)=(1-c(F))^{-1}$, the `multiplicities of the $F's$ in $D$', with the convention that $+\infty=\frac{1}{0}$. The correspondence between coefficients and multiplicities is bijective and strictly increasing. Any irreducible divisor $F$ of $X$ not in the support of $D$, has thus $D$-coefficient $c_{D}(F):=0$, and $D$-multiplicity $m_{D}(F)=1$. 

The multiplicities have a natural geometric meaning: for example, if all $c_D(F)$ are of the form $(1-\frac{1}{m_F})$, for integers $m_F$, then $D$ is a ramification divisor. Moreover, the geometric functorial properties in the `orbifold category' are expressed in terms of multiplicities, not in terms of coefficients (see below). 

The pairs $(X, D)$ will usually be supposed to be `smooth', meaning tha $X$ is smooth, and that the support 
$\displaystyle  \Supp(D):=\cup_{\{F\subset X : c_D(F)\neq 0\}} F$ has simple normal crossings. These `smooth' orbifolds will now be given, following  \cite{Ca07}, a category structure.


\subsection{Definitions}
Let $f:X\to Z$ be a regular map {\it with connected fibres} between two complex connected normal projective varieties. Assume that orbifold structures $(X, D_X)$ and $(Z,D_Z)$ are given on $X$ and $Z$, respectively. Assume further that $f(X)$ is not contained in $Supp(D_Z)$. Let then $E$ be any irreducible divisor on $Z$. Assume moreover that: either $E$ is $\bQ$-Cartier on $Z$, or that  $Z$ is smooth, or $\Bbb Q$-factorial, or that $f(X)=Z$. We can then write, since the coefficients $c_{f^*(E)}$ are then well-defined for any irreducible divisor $F\in X$ contained in $f^{-1}(E)$:

\begin{equation}
f^*(E):= \sum_{F\subset X\vert f(F)=E} c_{f^*(E)}(F)\cdot F+R, 
\end{equation} 
where $R\subset X$ is an $f$-exceptional effective divisor (ie: such that $f(R)\subsetneq E$ is of codimension at least $2$ in $Z$). Notice that $R=0$ if $f$ has all its fibres equidimensional, or if $dim(Z)=1$.\medskip

  
\medskip


When $f:X\to Z$, if the coefficient $t:=c_{f^*(E)}(F)$ is defined for $f(F)\subset E$, we shall simply write: $f^*(E)=t.F+...$ in order to isolate the component $F$ from the other components of $f^{-1}(E)$.

\begin{definition}\label{deforbmorph}(\cite{Ca07}, D\'efinition 2.3) Assume that $c_{f^*(E)}(F)$ is well-defined for each Weil divisor $E$ in $Z$ and each Weil divisor $F$ in $X$. We say that $f$ is an `orbifold morphism' if $f(X)$ is not contained in $Supp(D_Z)$, and if, for any $E\subset Z$, and for each $F$ such that $c_{f^*(E)}(F)>0$, we have:
\begin{equation}c_{f^*(E)}(F)\cdot m_{D_X}(F)\geq m_{D_Z}(E),
\end{equation} 
\end{definition} 

\subsection{Orbifold birational equivalence}

\begin{definition}\label{deforbifoldbirationalequivalence}The orbifold morphism $f:(X',D')\to (X,D)$ is said to be an `orbifold birational equivalence' if $(X',D')$ and $(X,D)$ are smooth orbifold pairs, if 
$f: X'\to X$ is birational, and if $D= f_\star(D')$, or equivalently, if $(X,D)$ is the orbifold base of $(f,D')$. 
\end{definition} 

\begin{definition}
We say then that $D'$ is $(f,D)$-minimal if, $(X,D)$ and $(f,X')$ being given, $D'$ is the smallest orbifold divisor on $X'$ for which $f$ is an `orbifold birational equivalence'. 
This also means that, for each irreducible divisor $F'$ on $X'$, its $D'$-multiplicity $m_{D'}(F')$ is either equal to $m_D(F)$ if $f(F')=F$ is a divisor of $X$, and otherwise, if $F'$ is $f$-exceptional, equal to $max\{1, m(f,F')\}$, where $m(f,F')$ is the minimum, taken over all irreducible irreducible divisors $F$ in $X$ containing $f(F')$ of the rational numbers: $\frac{m_D(F)}{t_{F,F'}}$, where $f^*(F)=t_{F,F'}.F'+\dots$ is the multiplicity of $F'$ in $f^*(F)$.

Notice that $f:(X',D')\to (X,D)$ being an orbifold birational equivalence, we can always-in a unique way-diminish the multiplicities of the components of $D'$ in order to make it $(f,D)$-minimal. We shall often write simply $f$-minimal instead when $D$ is implicitely known.
\end{definition} 

\begin{definition}\label{defbireq}
The `orbifold birational equivalence relation' is the one generated by this binary relation among smooth projective orbifold pairs. If $(X,D)$ and $(X',D')$ are birationally equivalent (in the orbifold sense) if there exists  an integer $m\geq 0$ and `orbifold birational equivanences' $u_i: (X_{2i},D_{2i})\to (X_{2i+1},D_{2i+1})$ and $v_i: (X_{2i+2},D_{2i+2})\to (X_{2i+1},D_{2i+1})$, for $i=0,1,\dots, m$ with $(X_0,D_0):=(X,D)$, and:

 $(X_{2m+2},D_{2m+2})=(X',D')$. We then say that the birational equivalence between $(X,D)$ and $(X',D')$ is induced by: $$g:=(v_m)^{-1}\circ u_m\circ (v_{m-1})^{-1}\circ u_{m-1}\circ \dots (v_0)^{-1}\circ u_0$$. 
\end{definition} 

\begin{definition} \label{defneat} The orbifold morphism $f$ is said to be `neat' if $(X,D_X)$ and $(Z,D_Z)$ are smooth, if $f^{-1}[f(\Supp(D))\cup D_f]$ is of simple normal crossings, where $D_f$ is the divisor on $X$ where $df$, the derivative of $f$, is not of maximal rank, and if, moreover, there exists a birational morphism $u:X\to X_0$, with $X_0$ smooth, such that every $f$-exceptional divisor $F\subset X$ is also $u$-exceptional. (This last property is obtained by Raynaud's flattening theorem together with desingularisation of the fibre-product). 
\end{definition}

The reason for these definitions will appear below, when composition of fibrations are considered. See Example \ref{exneat} below for an illustration of these notions, and \cite{Ca07} for further details. 

Notice that, $Supp(D_Z)\subset f(Supp(D_X))$, and $D_Z\leq f_*(D_X)$ if $f$ is birational, but even if $D_Z= f_*(D_X)$, it is not true that $f:(X, D_X)\to (Z, D_Z)$ is an orbifold morphism, in general.

\begin{remark}\label{cremona} Let us give an example (extracted from \cite{Ca07}, and inspired by the Cremona transformation) which shows that if $(X,D_i), i=1,2$ are smooth and orbifold-birationally equivalent to some smooth $(X,D)$ for some orbifold divisors $D_i,D$ on the same $X$, and if $D^+=sup(D_1,D_2)$, it may happen that $(X,D^+)$ is no longer orbifold birational to them. 

Let $b: X\to \Bbb P^2$ be the blow-up in $3$ points $a,b,c$ in general position, together with the $3$ exceptional curves $A,B,C$. Let $\alpha,\beta,\gamma$ be the three lines going through two of these points, and $\alpha',\beta',\gamma'$ their strict transforms in $X$. We thus get a new map $b':X\to (\Bbb P^2)'$ contracting the $3$ curves $\alpha',\beta',\gamma'$. Consider the three (reduced) orbifold divisors $D:=A+B+C$, $D':=\alpha'+\beta'+\gamma'$, and $D^+:=sup\{D,D'\}=D+D'$ on $X$. 

We define in general, for orbifold divisors $D,D'$ on $X$, the orbifold divisor $sup(D,D')$ to be the one such that its multiplicity on each irreducible (Weil) divisor $E$ of $X$ is the maximum of the two multiplicities of $E$ for $D$ and $D'$. The support of $sup(D,D')$ is thus the union of the two supports.

Then $b:(X,D)\to \Bbb P^2$, $b':(X,D')\to (\Bbb P^2)'$ are birational orbifold equivalences (the bases being equipped with the zero orbifold divisors). Since so are $b:X\to \Bbb P^2$, and $b':X\to (\Bbb P^2)'$, the identity map of $X$ induces orbifold birational equivalences: $(X,D)\cong X\cong (X,D')$ as well. Let $\Delta:= b_*(D')$ and $\Delta':=(b')_*(D)$ be orbifold divisors on $\Bbb P^2$ and $(\Bbb P^2)'$ respectively. Then $b:(X,D^+)\to (\Bbb P^2,\Delta)$ and $b':(X,D^+)\to ((\Bbb P^2)',\Delta')$ are also birational orbifold equivalences, but $b:(X,D')\to (\Bbb P^2,\Delta)$ and $b':(X,D)\to ((\Bbb P^2)',\Delta')$ are not orbifold morphisms, although they map to their orbifold bases (defined below). Notice also that $D^+$ is {\bf not} orbifold-birationally equivalent to $(X,D),(X,D'), X$, since: $\kappa(X,D^+)=0>-\infty=\kappa(X,D)=\kappa(X,D')$, and birational orbifold equivalence preserves the Kodaira dimension (and more generally, the differentials, see Proposition \ref{pom} below).
\end{remark}


\subsection{Orbifold base of a fibration}

Let a fibration $f: X\to Z$ be given, $X,Z$ being normal connected, and $X$ moreover equipped with an orbifold structure $(X,D_X)$.

\begin{definition}\label{dob} {\rm (\cite{Ca07}, D\'efinition 3.2)} Let $E\subset Z$ be any irreducible divisor, and, as defined above:
\begin{equation}f^*(E)= R+ \sum_{f(F)= E}c_{f^*(E)}(F)\cdot F,
\end{equation} 
where $R$ is an $f$-exceptional divisor (i.e: $codim_X(f(R))\geq 2$). 

Define the multiplicity $m_{(f, D_X)}(E)$ of $E$ relative to $(f,D_X)$ by:
\begin{equation}m_{(f,D_X)}(E):= \inf_{f(F)= E}c_{f^*(E)}(F)\cdot m_D(F),
\end{equation} 
and the `orbifold base' $D_Z(f,D_X)= D_Z$ on $Z$ by the equality:
\begin{equation}D_Z:=\sum_{E\subset Z} \Big(1-\frac{1}{m_{(f,D_X)}(E)}\Big).E
\end{equation} 
The sum above being finite, this is an orbifold divisor on $Z$.
\end{definition}
\medskip

\noindent The orbifold base is thus the largest orbifold on the base making $f$ an orbifold morphism in codimension one over $Z$. In general, $f:(X,D_X)\to (Z,D_Z)$ is not an orbifold morphism, because the multiplicities on the $f$-exceptional divisors of $X$ are not taken into account, and may be too small. However, it is always possible to obtain a neat orbifold morphism simply by flattening $f$, desingularising the fibre product, and increasing sufficiently the multiplicities of $D_X$ on the exceptional divisors of $f$.
\smallskip

\begin{lemma}\label{lom}{\rm (\cite{Ca07}, Proposition 3.10)} Let $f:(X,D_X)\to Z$ be a fibration, wth $(X,D_X)$ smooth, and $Z$ normal. There exists a birational orbifold morphism $g:(X',D')\to (X,D_X)$ and a modification $h:Z'\to Z$, with $Z'$ smooth, together with a fibration $f':(X',D')\to Z'$ such that:
\begin{enumerate}

\item[{\rm (1)}] $f\circ g=h\circ f'$.

\item[{\rm (2)}] The orbifold base  $(Z',D_{Z'})$ of $f':(X,D')\to Z'$ is smooth.

\item[{\rm (3)}]  $f':(X',D')\to (Z',D_{Z'})$ is a neat orbifold morphism (see Definition \ref{defneat}).
\end{enumerate}

\end{lemma}


\subsection{Composition of fibrations.}

In this subsection we consider two fibrations $f:(X,D_X)\to Y$ and $g:Y\to Z$ with $(X,D_X)$, as well as $Y$ and $Z$ smooth. We aim at determining the orbifold base $(Z,D_{(g\circ f,D)})$ of the composition $(g\circ f, D_X)$  in the `smooth orbifold category'. The natural candidate is the orbifold base $\displaystyle (Z, D_{(g,D_{Y})})$ of $g$, with $\displaystyle (Y,D_Y)$ the orbifold base of $(f,D_X)$, that is: $D_Y:=D_{(f,D_X)}$ . 

The two orbifold bases $(Z,D_{(g\circ f,D)})$ and $\displaystyle (Z,D_{(g,D_Y)})$ on $Z$ differ, in general, although one always has: $\displaystyle D_{(g\circ f,D_X)}\leq D_{(g,D_{Y})}$. This is because of divisors $F\subset X$ of small multiplicity which are $f$-exceptional, but not $g\circ f$-exceptional. This phenomenon is however excluded when $f$ is an `orbifold' morphism:

\begin{proposition}\label{lobom} (\cite{Ca07}, Proposition 3.14). Let $f:(X,D_X)\to Y$ and $g:Y\to Z$ be fibrations. Let $(Y,D_Y)$ be the orbifold base of $(f,D_X)$. Assume that the induced map $f: (X,D_X)\to (Y,D_Y)$ is an orbifold morphism. Then we have $(Z,D_{(g\circ f,D)})=(Z, D_{(g,D_{Y})})$.
\end{proposition}

The following extremely simple example illustrates the above notions of orbifold base, orbifold rational equivalence, and orbifold morphism, as well as the preceding proposition.

\begin{remark}\label{exneat} Let $Z,V$ be any connected complex projective manifolds with $dim(Z)\geq 2, dim(V)\geq 1$. Define $Y:=Z\times V$, with $g:Y\to Z$ the first projection. Let $T_0\subset Z$ be any connected projective submanifold of codimension $1$. Let $T':=T_0\times \{v\}\subset Y$, for some $v\in V$. Let $f:X\to Y$ be the blow-up of $T'$ in $Y$, with $F$ the exceptional divor of $f$, and $T\subset X$ the strict transform of $T":=T_0\times V$. 

Let $m'\geq m>1$ be any rational numbers. We define on $X$ the two orbifold divisors $D:=(1-\frac{1}{m}).T$ and $D':=(1-\frac{1}{m'}).F+D$. We then observe: $(Y,D_Y:=(1-\frac{1}{m}).T")$ is the orbifold base of both $f:(X,D')\to Y$ and of $f:(X,D)\to Y$. However, the first one is an orbifold morphism (and hence an orbifold birational equivalence), while the second one is not. Next, the orbifold base of $g:(Y,D_Y)\to Z$ is obviously $(Z,D_Z:=(1-\frac{1}{m}).T")$. Thus: $(Z, D_{(g,D_Y)})=(Z,D_Z)$. An easy check shows that the orbifold base of $g\circ f: (X,D')\to Z$ is $(Z,D_Z)$, but the orbifold base of $g\circ f:(X,D)\to Z$ is $(Z,0)$ (since $g\circ f( F)=T_0\subset T"$, but the multiplicity of $F$ in $D$ is $1$). 

One may take for example: $Z=\Bbb P^2,V=\Bbb P^1$, and for $T_0$ a line in $\Bbb P^2=Z$. Blowing-down $T$ in $X$, one sees that $X$ is also the blow-up in one point of $X_0$, a smooth $\Bbb P^2$-bundle over $\Bbb P^1$. From which one may easily deduce that for any $m>0$, $(X,D)$ is (orbifold) birational to $(X_0,0)$, and that $(X,D')$ is slope Rational Connected for any $m>0$. 
\end{remark}


\subsection{Differentials and orbifold morphisms.}

\noindent We show here a crucial functoriality property of differentials for orbifold morphisms. This property is actually almost a characterisation of orbifold morphisms (see \cite{Ca07}, Proposition 2.10, for a similar result). 

Let $f:(X,D_X)\to (Z,D_Z)$ be an orbifold morphism, and $\pi:Y\to X$ and $p: W\to Z$ be Kawamata covers adapted to $(X,D_X)$ and $(Z,D_Z)$ respectively. Let $G$ and $H$ be the finite groups acting on $Y$ and $W$ respectively. Let 
\begin{equation}
T:=(Y\times_ZW)^n
\end{equation} be the normalisation of any component of the fibred product, together with the natural projections $\rho:T\to Y$ and $\tau: T\to W$. The projection $\pi\circ \tau:T\to X$ is still Galois, with group $L$, normal in $(G\times H)$, the stabilizer of $T$ in $(Y\times_ZW)^n$, the group $L$ being onto on both $G$ and $H$. The components of $(Y\times_ZW)^n$ being exchanged under the operation of $(G\times H)/L$. 

Let $df: TX\to f^*(TZ)$ be the derivative of $f$. Its lifting to $T$ induces a map:
\begin{equation}(\pi\circ \rho)^*(df): (\pi\circ \rho)^*(TX)\to (f\circ \pi\circ \rho)^*(TZ)=(p\circ \tau)^*(TZ).
\end{equation}  
We recall (see \cite{CP15}) that we have the natural inclusions $\pi^*(T(X,D))\subset \pi^*(TX)$ as well as $p^*(T(Z,D_Z)\subset p^*(TZ)$.

\begin{proposition}\label{pom}  If $f$ is an orbifold morphism, these maps extend to define natural maps of sheaves on $T$:
 \begin{equation}
 (\pi\circ \rho)^*(df):\big(\rho^*(\pi^*(T(X,D_X))\big)\to \tau^*(p^*(T(Z,D_Z))), and:
  \end{equation}
  
  \begin{equation}
 (\pi\circ \rho)^*(df): \tau^*(p^*(\Omega^1(Z,D_Z)))\to\big(\rho^*(\pi^*(\Omega^1(X,D_X))\big),
 \end{equation}
the latter one being injective.
\end{proposition}

\begin{proof} Because both sheaves $\rho^*(\pi^*(T(X,D)))$ and $\tau^*(p^*(T(Z,D_Z)))$ are locally free, it is sufficient to establish the statement in codimension one on $T$, so in codimension one over $Y$, or $X$. It is thus sufficient to consider the situation over generic points of a divisor $F$ of $X$ which is contained in $f^{-1}(\Supp(D_Z))$, since $\tau^*(p^*(T(Z,D_Z))\subsetneq \tau^*(p^*(TZ))$ only there.

Let thus $x_0\in F$ be such a point, and $z_0=f(x_0)\in f(F)\subset E=f(F)$. 
 We may thus assume that, in suitable coordinates $x=(x_1,...,x_n)$ for $X$ near $x_0$ and $z=(z_1,...,z_p)$ for $Z$ near $z_0:=f(x_0)$ the following hold true near $x_0$ and $z_0$: $E=\{z_1=z_2=...=z_{q}=0\}$, for some $1\leq q\leq p$. Write $E_h=f^{-1}(z_h=0)$, and: $f^*(E_h)=t_h.F+...$, for some integer $t_h\geq 1$, for $h\in\{1,...,q\}$.

\begin{enumerate}

\item[{\rm (a)}] The divisor $F$ is given by the equation $x_1=0$ near $x_0$.

\item[{\rm (b)}] $f$ is given, for holomorphic functions $g_{\ell}(x),\ell\in\{1,...,p\}$, by:

 $f(x_1,...,x_n)= (x_1^{t_1}.g_1, x_1^{t_2}.g_2,...,x_1^{t_q}.g_q,g_{q+1},...,g_p)$
 \end{enumerate}

\noindent Let $m$ (resp. $m'_h$) be the multiplicity of $F$ in $D_X$ (resp.$f^{-1}(E_h),\forall h\leq q$ in $D_Z$). 
Because $f$ is an orbifold morphism, we also have:

\begin{enumerate}

\item[{\rm (c)}] $t_h.m\geq m'_h$, thus: $(1-\frac{1}{t_h.m})\geq (1-\frac{1}{m'_h}),\forall h\leq q$.
\end{enumerate}

By definition, the bundle $\pi^*(T(X,D))$ is `symbolically' generated as an $\cO_Y$-module by: $\displaystyle \pi^*(x_1^{(1-\frac{1}{m})}.\partial_{x_1}, \partial_{x_2},...\partial_{x_n})$ along $F$, and $p^*(T(Z,D_Z))$  is generated as an $\cO_W$-module by $\displaystyle p^*(z_j^{(1-(1/m'_1)}.\partial_{z_1}, \partial_{z_2},\dots, \partial_{z_p}))$ along $E_1$, and by similar expressions involving the $\displaystyle p^*(z_j^{(1-(1/m'_j)}.\partial_{z_j})$ along the other $E_j's$.

The relations (b)imply that: 
$df(\partial_{x_1})=(x_1^{s_j}.g'_j.f^*(\partial_{z_j}))_{j=1,...,p}$, where the $g'_j(x)$ are holomorphic, $s_j:=t_h-1,$ for $ j=1,...,q,$ and $s_j=0,$ for $j=q+1,...,p$. 

Thus: 
$df(x_1^{(1-\frac{1}{m})}.\partial_{x_1})=(x_1^{s_j+(1-\frac{1}{m})}.g'_j.f^*(\partial_{z_j}))_{j=1,...,p}.$

\medskip

Now observe that, for $j\leq q$, we have: 

$x_1^{t_j-1+(1-\frac{1}{m})}.(f^*(z_j^{(1-(1/m'_j)})^{-1}=x_1^{t_j-1+(1-\frac{1}{m})-t_j.(1-\frac{1}{m'_j})}=x_1^{(\frac{t_j}{m'_j}-\frac{1}{m})}.$

The conclusion follows from (c) above, which implies that the `fractional vector field' $df(x_1^{(1-\frac{1}{m})}.\partial_{x_1})$ is in $f^*(T(Z,D_Z))$ along $E=f(F)$ if one considers its $j$-th component for $j\leq q$, and one uses that $F\subsetneq f^*(E_j)$ for $j=q+1,...,p$. The assertion for $df(\partial_{x_k}),k\geq 2$ follows from the same computation of $df(\partial_{x_2})$, and simpler estimates of exponents, with $s_j$ replaced by $s_j+1$ for $j\leq q$, based on: $m'_j\geq m,\forall j\leq q$.

The dual case of the cotangent bundles is similar.

\medskip
 
\noindent Now these `symbolic' considerations easily imply the assertion, since the ramified covers $\pi$ and $p$ take the simplest possible form $x_1=y_1^k, z_1=w_1^{\ell}$, and $x_h=y_h,z_h=w_h$, for $h\geq 2$, near the points $x_0,z_0$, for suitable integers $k,\ell$. We leave the simple verifications to the reader.
\end{proof}

The preceding computation moreover shows that\footnote{Up to $f$-exceptional divisors $F\subset X$.}, since $\Supp(D_X)\subset f^{-1}(D_Z)$ when $(Z,D_Z)$ is the orbifold base of $(f,D_X)$, the quotient sheaf $\tau^*(p^*(T(Z,D_Z))/(\pi\circ \rho)^*(df)(\rho^*(\pi^*(T(X,D_X)))$ is supported on the inverse image in $T$ of the components $F$ of $f^{-1}(D_Z)$ for which: $t.m>m'$ in the above notations. The union $\widetilde{F}\subset X$ is thus a divisor \emph{partially supported on the fibres of $f$} in the following sense.
 
\begin{definition}\label{ddpsf}(See \cite{Ca04}) Let $f:X\to Z$ be as before. Let $F:=\cup_{s\in S}F_s$ be a finite union of irreducible divisors of $X$. We say that `$F$ is partially supported on the fibres of $f$' if, for each $s\in S$, either $F_s$ is $f$-exceptional (ie: $codim_Z(f(F_s))\geq 2$), or if the following two properties hold:

1. $f(F_s):=E_s$ is a divisor of $Z$.

2. $f^{-1}(E_s)\cap F\subsetneq f^{-1}(E_s)$. (ie: at least one irreducible component of the RHS is not a component of $F$).
\end{definition} 
\medskip

\noindent We thus have the following corollary.

\begin{proposition}\label{pdpsf} Let the situation $f:(X,D_X)\to (Z,D_Z)$ be as above, with $(Z,D_Z)$ the orbifold base of $(f,D_X)$. 

The quotient sheaf $\tau^*(p^*(T(Z,D_Z))/(\pi\circ \rho)^*(df)(\rho^*(\pi^*(T(X,D_X)))$ is supported in $(\pi\circ\rho)^{-1}(F)$, where $F\subset X$ is a divisor partially supported on the fibres of $f$.

The dual statement also holds similarly, the saturation being taken in $(\rho^*(\pi^*(\Omega^1(X,D_X)))$: $$[(\pi\circ \rho)^*(df)(\tau^*(p^*(\Omega^1(Z,D_Z)))]^{sat}/(\pi\circ \rho)^*(df)(\tau^*(p^*(\Omega^1(Z,D_Z)))$$ is supported in $(\pi\circ\rho)^{-1}(F)$, where $F\subset X$ is a divisor partially supported on the fibres of $f$.

\end{proposition}


\subsection{Integral parts of orbifold tensors}\label{sipot}

Let $(X,D)$ be a smooth (complex projective, connected) orbifold pair. Let $\pi:Y\to X$ be an adapted Kawamata cover, together with its vector bundle $\pi^*(\Omega^1(X,D))$. We have the notion of orbifold tensors $\otimes^m(\pi^*(\Omega^1(X,D)))$ on $Y$, and of their global sections (which may be proved again to be invariant under orbifold rational equivalence). These a priori depend on the choice of $\pi$. But we shall give here a notion of `integral part' of both these sheaves and of their global sections defined directly on $X$, permitting to show that certain orbifold invariants of the differentials on $(X,D)$ defined on $Y$, are, in fact, independent on the choice of $\pi$.

We shall define, for any $m>0$ the locally free sheaf $[\otimes^m](\Omega^1(X,D)$ in local coordinates $(x_1,\dots,x_n)$ adapted to the snc orbifold divisor $D=c_1.H_1+\dots +c_n.H_n$ in these coordinates. Here $H_j$ is the coordinate hyperplane of equation $x_j=0, j\in \{1,\dots,n\}$, and $c_j\in [0,1]\cap \Bbb Q$.

In these coordinates, $[\otimes^m](\Omega^1(X,D)$ is the locally free sheaf of $\cO_X$-modules generated by the $n^m$ elements:

 $t_{(M_1,M_2,\dots,M_n)}=\frac{dx_1^{M_1}}{x_1^{[m_1.c_1]}}\otimes\dots \otimes \frac{dx_n^{M_n}}{x_n^{[m_n.c_n]}}$,where:

1. $(M_1,M_2,\dots ,M_n)$ is a partition of $M:=\{1,\dots,m\}$

2. $m_j$ is the cardinality of $M_j$, $j\in \{1,\dots,n\}$.

3. $dx_1^{M_1}\otimes\dots \otimes dx_n^{M_n}:=dx_{h(1)}\otimes dx_{j(2)}\otimes \dots \otimes dx_{h(m)},$ where $h(k)=j$ if and only if $k\in M_j$, for $k\in M,j\in \{1,\dots,n\}$.

It is easy to see that these sheaves are independent on the choosen adapted coordinates. 

The (locally free) sheaves $[Sym^m](\Omega^1(X,D))$ and $[Sym^m(\wedge^p)](\Omega^1(X,D)):=[Sym^m(\Omega^p)](X,D)$ are defined analogously. They are subsheaves of  $[\otimes^m](\Omega^1(X,D)$, and $[\otimes^{m.p}](\Omega^1(X,D)$, respectively. Let us notice that the preceding construction can also be similarly done for orbifold tensors of any type $(s,m)$) as well, and for arbitrary representations of $Gl(n,\Bbb C)$.

By the same proof (but simpler, since one does not need to consider covers) as for Proposition \ref{pom}, one shows:

\begin{proposition} \label{pom'} Let $f:(X,D)\to (Z,D_Z)$ be an orbifold morphism between smooth (complex projective) orbifolds. For any $m>0$, $f$ induces a morphism of sheaves of $\cO$-modules: $f^*: [\otimes^m](\Omega^1(Z,D_Z))\to [\otimes^m](\Omega^1(X,D)$. 

If $f$ is birational, then $f^*$ is isomorphic at the level of global sections. 

Moreover, $f_*:[\otimes^m](\Omega^1(X,D))\to [\otimes^m](\Omega^1(Z,D_Z)$ (resp. $(f^*)^{sat}: [\otimes^m](\Omega^1(Z,D_Z))\to [\otimes^m](\Omega^1(X,D))$ is then (if $f$ is birational) an isomorphism at the level of sheaves.

The same statements hold for $[Sym^m]$ and $[Sym^m(\wedge^p)],\forall m,p>0$. 
\end{proposition}

\begin{corollary} Let $(X,D)$ and $(X',D')$ be smooth projective orbifolds which are birationally equivalent in the orbifold sense, this equivalence being induced by $g=:=(v_m)^{-1}\circ u_m\circ (v_{m-1})^{-1}\circ u_{m-1}\circ \dots (v_0)^{-1}\circ u_0,$ in the notations of Definition \ref{defbireq}. 

Then $g^*:=(v_m^*)^{sat}\circ (u_m)_*\circ (v_{m-1}^*)^{sat}\circ (u_{m-1})_*\circ \dots \circ(v_0^*)^{sat}\circ (u_0)_*$ is, for any $N>0$, an isomorphism of sheaves between $[\otimes^N](\Omega^1(X',D')$ and $[\otimes^N](\Omega^1(X,D)$. 
The same statements hold for $[sym^m(\wedge^p)]$, for any $m>,p>0$.
\end{corollary}

\begin{remark}\label{remg^*} Notice that $g^*$ above is just the saturation in the relevant sheaves of the isomorphic map $g^*: \otimes^N\Omega^1_{X'}\to \otimes^N \Omega^1_X$ defined by the same formula.
\end{remark}

\begin{example}\label{exinttensors} Let $(Z,D_Z)$ be a smooth (projective connected) orbifold, $D_Z=\sum_J (1-\frac{b_j}{a_j}).D_j$, let $m>0$ be an integer divisible by each of the $a_j's$. Then $m.(K_Z+D_Z)\in [Sym^m(\wedge^p)](\Omega^1(Z,D_Z))$ if $p:=dim(Z)$. If $f:(X,D)\to (Z,D_Z)$ is an orbifold morphism, then $f^*(m.(K_Z+D_Z))\subset [Sym^m(\wedge^p)](\Omega^1(X,D))$.
\end{example}

\begin{theorem}\label{thmintegraltensors} Let $(X,D)$ be smooth projective, with $\pi:Y\to X$ be a Kawamata cover adapted to $D$, Galois of group $G$. Then, for nay $m>0$, we have: $[\otimes^m](\Omega^1(X,D)$ is the $G$-invariant part $\pi_*^G(\otimes^m(\pi^*(\Omega^1(X,D))))$of $\pi_*(\otimes^m(\pi^*(\Omega^1(X,D))))$. Here, for any torsionfree coherent $G$-sheaf $E$ of $\cO_Y-$modules on $Y$, $\pi_*^G(E)$ is the (coherent) sheaf of $\cO_X$-modules on $X$ associated to the presheaf $H^0(\pi^{-1}(U), E)^G$. 
\end{theorem}

For some details on the functor $\pi_*^G$, see Section \ref{Ginv} below.

\begin{proof} Because $\pi_*^G(\pi^*(\Omega^1(X,D))=\Omega^1(X)$ over $X-Supp(D)$, and since all three sheaves $\otimes^m(\pi^*(\Omega^1(X,D)))$,$\pi_*^G(\otimes^m(\pi^*(\Omega^1(X,D))))$ and $[\otimes^m](\pi^*(\Omega^1(X,D)))$ are reflexive, we just need to check the equality outside a Zariski closed set of codimension $2$ on $X$. We can thus check this over the domain $U\subset X$ of a local coordinate chart $(x_1,\dots,x_n)$ centered at $x_0\in Supp(D)$ over which $Supp(D)$ is smooth defined by the equation $x_1=0$, does not meet the ramification locus of $\pi$ elsewhere as over $Supp(D)\cap U$, and such that some (hence any) connected component $V$ of $\pi^{-1}(U)$ is the domain of a chart with coordinates $(y_1,\dots , y_n)$ such that the restriction of $\pi$ to $V$ is given by: $\pi_V:y=(y_1,\dots,y_n)\to x=(x_1:=y_1^{ka},x_2:=y_2,\dots, x_n:=y_n)$. Here the multiplicity of $(x_1=0)\cap U$ in $D$ is equal to $\frac{a}{b}$, with $a\geq b\geq 0$ non negative coprime integers, and $k>0$ is any given positive integer. The cover $\pi_V:V\to U$ is thus Galois of group $H$ cyclic of order $ka$, and since the orbit of $V$ under  the action of $G$ is $\pi^{-1}(U)$, with stabiliser $H$, we have over $U$ the identification $\pi_*^G(E)_{\vert U}=(\pi_V)_*^H(E_V)$ for any coherent $G$-sheaf on $X$. We shall thus now check the claimed equality for $X=U, Y=V, H=G$ cyclic as above, in the given coordinates. In order to simplify the notations, we show this for the symmetric differentials. 

Any local section $s(y)$ of $E:=\otimes^m(\pi^*(\Omega^1(X,D))$ can be uniquely writtent in the form: $s(y)= \Sigma_{h=0}^{m}f_h(y).y_1^{kbh}(\frac{dy_1}{y_1})^{\otimes h}\otimes \pi^*(w_h)$, for holomorphic functions $f_h(y)=\Sigma_{t=0}^{t=ka-1}\pi^*(f_{h,t}(x)).y_1^t$, and some holomorphic symmetric $(m-h)$-form in $(dx_2,\dots, dx_n)$. Let $\zeta\in \Bbb C^*$ be a $ka$-th primitive root of unity, and $g$ a generator of the group $H$ (of cardinality $ka$). The action of $g$ on $f_{h,t}(y)$ (resp. $y_1$) is thus given by multiplication by $\zeta^t$ (resp. $\zeta$), and $g.s=s$ if and only if $f_{h,t}=0$ each time $kbh+t$ is not divisible by $kba$. In other terms, $s$ is $H$-invariant if and only if it is a sum of terms $\pi^*(f_{h,t}(x)).y_1^{t+bkh}.(\frac{dy_1}{y_1})^{\otimes h}\otimes \pi^*(w_h),$ with $t+bkh=q.ka$ for some nonnegative integer $q$. In this case, $t=ku$ for some integer $u, 0\leq u\leq (a-1)$, and $1>\frac{u}{a}=q-\frac{bh}{a}$. Thus: $q=\lceil\frac{bh}{a}\rceil$, and $\pi^*(f_{h,t}(x)).y_1^{t+bkh}.(\frac{dy_1}{y_1})^{\otimes h}\otimes \pi^*(w_h)=\pi^*(f_{h,t}(x)).(x_1^{\lceil\frac{bh}{a}\rceil}).(\frac{dx_1}{x_1})^{\otimes h}\otimes w_h)$ (up to the multiplicative constant $(ka)^{qa}$). This means that $s\in \pi^*([\otimes^m]\Omega^1(X,D))$, the reverse inclusion being obvious \end{proof}

The integral parts of orbifold tensors also naturally appear in the context of fibrations:

\begin{theorem}\label{taller-retour}\cite{AC'}  Let $f:(X,D)\to (Z,D_Z)$ be a fibration which is also a`neat' orbifold morphism between smooth projective orbifold pairs, such that $(Z,D_Z)$ is also the `orbifold base' of $f$. Then, for any $m\geq 0$, one has a natural isomorphism of sheaves: $[\otimes^m]\Omega^1(Z,D_Z)=f_*((f^*)^{sat}(\otimes ^m \Omega^1_Z))$, the saturation taking place in $[\otimes^m]\Omega^1(X,D_X)$.
\end{theorem}


\medskip

\subsection{Birational equivalence of orbifold bases: a question.}

The following question is important, but delicate.

\begin{question}\label{qbireqorbbases} Let $(X_i,D_i), i=1,2,$ be two smooth (projective say) orbifold pairs, given together with fibrations $f_i:(X_i,D_i)\to (Z_i,D_{Z_i})$ onto their (smooth) orbifold bases. Assume that both $f_i$ are `neat', and birational in the sense that $f_2\circ g=h\circ f_1$ for some suitable birational maps $g:X_1\to X_2$ and $h:Z_1\to Z_2$. 
Does $v: (Z_1,D_{Z_1})\to (Z_2,D_{Z_2})$ then induce an orbifold birational equivalence?
\end{question}

A weaker statement (which were a consequence of a positive answer to Question \ref{qbireqorbbases}) is the following:

\begin{theorem}\label{tdiffeqfibr} In the situation above, for any $N>0$, the isomorphism of sheaves: $h^*:\otimes ^m\Omega^1_{Z_2}\to \otimes^m \Omega^1_{Z_1}$.

We assume that the birational orbifold equivalence between $(X_1,D_1)$ and $(X_2,D_2)$ is induced by $g=(v_m^{-1}\circ u_m\circ\dots\circ v_0^{-1}\circ u_0)$, as in Definition \ref{defbireq}. We shall also decompose $h=h_2\circ h_1^{-1}$ for suitable birational maps $h_2:Z\to Z_2, h_1:Z\to Z_1$, and define $h^*:=(h_2^*)^{sat}\circ (h_1)_*$ extends to an isomorphism of sheaves $h^*: [\otimes^N]\Omega^1(Z_2,D_{Z_2})\to  [\otimes^N]\Omega^1(Z_1,D_{Z_1})$.
\end{theorem}

\begin{proof} Observe that $g^*(f_2^*([\otimes^N](\Omega^1(Z_2,D_{Z_2}))))\subset (f_1^*)^{sat}(h^*(\otimes^N(\Omega^1_{Z_2})))$ $=(f_1^*)^{sat}(\otimes^N\Omega^1_{Z_1})\subset (f_1^*)^{sat}([\otimes^N](\Omega^1(Z_1,D_{Z_1})))$. Taking $(f_1)_*$ on both sides, we get: 

$(f_1)_*(g^*(f_2^*([\otimes^N](\Omega^1(Z_2,D_{Z_2})))))\subset (f_1)_*((f_1^*)^{sat}([\otimes^N](\Omega^1(Z_1,D_{Z_1}))))$.

Since $(f_1)_*((f_1^*)^{sat}([\otimes^N](\Omega^1(Z_1,D_{Z_1}))))=[\otimes^N](\Omega^1(Z_1,D_{Z_1}))$, by Theorem \ref{}, we get: 

$(f_1)_*(g^*(f_2^*([\otimes^N](\Omega^1(Z_2,D_{Z_2})))))\subset [\otimes^N](\Omega^1(Z_1,D_{Z_1}))$, which provides a natural map (deduced from $g^*$) from $[\otimes^N](\Omega^1(Z_2,D_{Z_2}))$ to $ [\otimes^N](\Omega^1(Z_1,D_{Z_1}))$ extending $h^*$. One obtains an inverse map by considering $g^{-1}$ instead of $g$, reverting the r\^oles of $Z_1$ and $Z_2$. \end{proof}

\begin{remark} The preceding result strengthens a former result in \cite{Ca04} according to which $\kappa(Z_1,D_{Z_1})=\kappa(Z_2,D_{Z_2})$ in the previous situation.
\end{remark}


\section{Relative movable classes}\label{srelmov}

We consider here a fibration $f:X\to Z$ between two connected complex projective manifolds. We denote by $Mov(X)\subset N_1(X)$ the closed cone of movable classes on $X$, by $Mov^0(X)$ its interior, and similarly for $Z$. There is a natural surjective direct image map: 
\begin{equation}f_*:N_1(X)\to N_1(Z),
\end{equation} dual to the inverse image map $f^*:N^1(Z)\to N^1(X)$. The kernel $N_1(X/Z)$ of $f_*$ is the orthogonal of the image of $f^*$.

Let $Mov(X/Z):=Mov(X)\cap N_1(X/Z)$ be the closed cone of relative movable classes for $f$, and let $N_1^0(X/Z)\subset N_1(X/Z)$ be the real vector space generated by $Mov(X/Z)$. We denote the interior of $Mov(X/Z)$ in $N_1^0(X/Z)$ by: $Mov^0(X/Z)$.

We state without proof a result not used in the sequel (the proof can be given by the argument proving Proposition {\ref{prelmov}.6 below):

\begin{proposition}\label{pn_1} The quotient vector space $N_1(X/Z)/N_1^0(X/Z)$ is generated by the classes of complete intersection curves on the irreducible divisors of $X$ which are partially supported on the fibres of $f$. More precisely, if $F$ is such an irreducible divisor, the corresponding classes are of the form\footnote{$F.A^{n-2}$ if $r=0$.}: $F.A^{n-r-1}.B^{r-1}$, if $f(F):=E, dim (E)=r\geq 1$, and if $A$ is ample on $X$, $B=f^*(B_E),B_E$ ample on $E$.
\end{proposition}

For any smooth\footnote{One can of course compose with a desingularisation of $Y$, also.} irreducible  subvariety $Y$ on $X$, we have natural inclusion maps $N_1(Y/Z)\to N_1(Y)\to N_1(X)$, where $N_1(Y/Z)$ denotes the Kernel of the composition maps $N_1(Y)\to N_1(X)\to N_1(Z)$. In general, of course, $Mov(Y/Z):=Mov(Y)\cap N_1(Y/Z)$ is not contained in $Mov(X/Z)$. This is, however, the case if, for example, $Y=X_y=f^{-1}(z),z\in Z$ is a `general' fibre of $f$, in the following sense.

\begin{definition}\label{dgen} A point $z\in Z$ of an irreducible complex space $Z$ is `general' if it lies in the complement of countably many specified strict Zariski closed subsets of $Z$, a fibre $X_z$ of $f:X\to Z$ is `general' if it lies over a `general' point of $z\in Z$. \end{definition}

\begin{proposition}\label{pj_z} Let $f:X\to Z$ be as above. For $z\in Z$ `general', the natural inclusion map $j:X_z\to X$ induces an isomorphim of real vector spaces $j_*:N_1(X_z)\to N^0_1(X/Z)$, as well as bijections $j_*:Mov(X_z)\to Mov(X/Z)$, and $j_*:Mov^0(X_z)\to Mov^0(X/Z)$. For such a $z$, we can thus define a natural restriction map: $(j_*)^{-1}:Mov(X/Z)\to Mov(X_z)$. For each divisor $F$ on $X$, and each $\alpha\in Mov(X/Z)$, we have further: $\alpha.F=\alpha_z.F_z$, for $\alpha_z:=(j_*)^{-1}(\alpha)$, and $F_z:=j^*(F)$.
\end{proposition}

\begin{proof} Let $C_1(X/Z)$ be the Chow-Barlet space of $1$-dimensional algebraic cycles $\Gamma\subset X$ such that $f_*(\Gamma)=0$, that is: which are contained in some fibre of $f$. The space $C_1(X/Z)$ has countably many irreducible components $T_m$, which are compact (ie: projective), and which parametrise irreducible curves $C_t,t\in T_m$ generic. The support $X_m$ of $T_m$, which is the union of all curves $C_t$ parametrised by $T_m$, is thus a Zariski-closed subset of $X$. Let $Z_m:=f(X_m),\forall m$. We now consider only the $T_m's$ such that $X_m=f^{-1}(Z_m)$, and define a point $z\in Z$ to be general if it does not belong to any of the $Z_m$ such that $Z_m=Z$. For such a point $z$, if $[C]$ is the class of an irreducible curve $C$ contained in $X_z$, one thus has the equivalence between the two properties: $C$ moves in a $Z$-covering algebraic family of curves, and: $C$ moves in an $X$-covering family of curves such that $f_*([C])=0$. 

This shows, for such a $z$, the existence and surjectivity of the map $j_*:Mov^0(X_z)\to Mov^0(X/Z)$, and thus also the surjectivity of $j_*:N_1(X_z)\to N_1(X/Z)$. In order to show the injectivity of this last map, we only need to show, by transposition, the surjectivity of the dual restriction map: $j^*:N^1(X/Z):=(N^1(X)/(f^*(N^1(Z))+P))\to N^1(X_z)$, where $P\subset N^1(X)$ is the vector space generated by irreducible divisors partially suported on fibres of $f$. This is achieved by the same argument as before, but applied to $f$-relative divisors on $X$. One however possibly needs in the process to suppress countably many new Zariski-closed subsets of $Z$. 

For such a `general' $z\in Z$, the map $j_*:N_1(X_z)\to N^0_1(X/Z)$ is thus linear bijective, hence homeomorphic. The induced map $j_*:Mov^0(X_z)\to Mov^0(X/Z)$ is thus bijective, and so is $j_*:Mov(X_z)\to Mov(X/Z)$. The last assertion is then obvious, since true for any curve contained in a fibre of $X_z$.
 \end{proof}

\begin{definition}\label{dbig} A movable class on a smooth connected complex manifold $X$ is said to be `big' if it lies in the interior of the movable cone of $X$. 
\end{definition}

\begin{definition}\label{dbig'} Let $f:X\to Z$ be a fibration as before. A class $\alpha\in Mov(X/Z)$ is said to be `big on the general fibre' if its restriction $(j_*)^{-1}(\alpha):=\alpha_z$ is `big' in $X_z$ for $z\in Z$ `general'. We shall denote this simply by: $\alpha_z$ is `big' ($z\in Z$ being implicitly assumed to be `general').\end{definition}
\medskip

\noindent In this context, we have the following statement.

\begin{proposition}\label{prelmov} Let $f:X\to Z$ be as above, with $n:=dim(X), d:=n-dim(Z)$. Let $A, B_Z$ be very ample divisors on $X$ and $Z$, respectively, with $B:=f^*(B_Z)$. Let $\alpha\in Mov^0(X/Z)$, and $F:=\cup_{s\in S} F_s$ be an effective reduced divisor on $X$, partially supported on the fibres of $f$. 

1. $A^{d-1}.B^{n-d}\in Mov^0(X/Z)$.

2. $N^0_1(X/Z)$ is generated as a real vector space by classes: $A^{d-1}.B^{n-d}$.

If $\beta\in N_1(X/Z)$, we write $\alpha\geq \beta$ to mean that $\alpha-\beta\in Mov(X/Z)$.

3. There exists $\varepsilon>0$ such that $\alpha\geq \varepsilon. A^{d-1}.B^{n-d}$

4. For any $\vartheta\in N^0_1(X/Z)$, $k.\alpha\geq\vartheta,\forall k\geq k_0$, for some $k_0\in \mathbb R$.

5. For any set of given reals $b_s,s\in S$, there exists $\vartheta\in N_1(X/Z)$ such that $\vartheta.F_s=b_s,\forall s\in S$.

6. Let $\beta\in Mov(Z)$. There exists $\beta'\in Mov(X)$ such that: 

$f_*(\beta')=\beta$, and moreover: $\beta'.F_s=0,\forall s\in S$.
\end{proposition}

\begin{proof} The assertion 1 is clear because complete intersections of ample divisor classes of $X_z$ are big on $X_z$.

2 follows immediately from \cite{GT}, Proposition 6.5, which states that the set of classes of the form $(A_z)^{d-1}$, with $A_z$ an ample divisor on $X_z$ is open in $N_1(X_z)$. 

If $\alpha_z$ is big, $Mov^0(X_z)$ being open contains $(\alpha_z-\varepsilon.A^{d-1}.B^{n-d})$ for some $\varepsilon>0$. Hence 3, and 4, by the same argument applied to $\vartheta$.

We now prove Claim 5, which is more involved. 
Let thus $F_s$ be given. Define the class $\vartheta_s:=F_s.A^{d-1}.B^{n-d-1}\in N_1(X/Z)$, by choosing a suitable representative of this class, one sees that it is a movable class of $N_1(F_s/E_s)$.

Moreover, it is easy to see also that we the orthogonality relations: $\vartheta_s.F_{s'}=0=\vartheta_{s'}.F_s$, unless $E_s=E_{s'}$.

The real numbers $b_s$ being given, we will look for real numbers $t_s$ such that $\vartheta:=\sum_{s\in S}t_s.\vartheta_s$, and with the property that $\vartheta.F_s=b_s,\forall s\in S$. By the orthogonality relations above, we just need to determine these real numbers in the case where $f(F_s)=E_s=E$ is the same $E$, for all $s\in S$. But in this case we can even reduce to: $X$ is the surface $X'$ cut out by (sufficiently generic and large multiple of) $A^{d-1}.B^{n-d-1}$, and mapped by $f$ to the curve $Z':=B_Z^{n-d-1}\subset Z$. In this case, because $F$ is partially supported on the fibres of $f$, the same is true of $F':=F\cap X'$ over $Z'$. The assertion is then a consequence of the `negativity lemma' on quadratic forms of \cite{BPV}, Lemma 2.10, p. 19. This lemma indeed says that if $g:S\to B$ is a projective connected morphism from a smooth connected surface onto a curve with fibre $S_b:=g^*(b),b\in B$, the intersection matrix of the components of $S_b$ is semi-negative with Kernel $S_b$, which implies that it is strictly negative, hence invertible, on the vector space generated by the components of any curve $C\subset S_b$ with support strictly contained in $S_b$.

We now prove assertion 6. Remark first that any $\beta\in Mov(Z)$ can be lifted to a $\beta"\in Mov(X)$ such that $f_*(\beta")=\beta$. This is clear if $\beta=[C_t]$ is the class of an algebraic irreducible family of curves $(C_t),_{t\in T}$ of $Z$ which is $Z$-covering: one just need to choose, for example: $\beta":=[A]^{d-1}.[f^{-1}(C_t)]$, for $t\in T$ generic. This construction extends linearly (and thus continuously) to any $\beta\in Mov(Z)$.

Let us now consider an arbitrary irreducible divisor $E\subset Z$. Let $f^*(E)=\sum_{k\in K} c_k.F_k+G'$, where $G'$ is an effective $f$-exceptional divisor, while $f(F_k)=E,$ and $c_k>0, \forall k\in K.$ Let us write: $K=J\cup L$, for some (possibly empty)$ J\subset K$, and some (nonempty) $L\subset K$, where $J:=(K\cap S)$ is the set of indices of the components of $f^*(E)$ which are contained in the given divisor $F$ which is `partially supported on the fibres of $f$'.

Let, as above, $\vartheta_k:=F_k.A^{d-1}.B^{n-d-1}$, for any $k\in K$. We obviously have: $f_*(\vartheta_k)=0,\forall k\in K$. We shall choose $\beta':=\beta"+\vartheta$, with $\vartheta:=\sum_{k\in J}\nu_k.\vartheta_k$.  Because of the orthogonality relations mentioned above, $\beta'.G=\beta".G$ for any $f$-vertical divisor $G$ not mapped onto $E$ by $f$. In order to have $\beta'$ movable, together with $\beta'.F_s=0,\forall s\in S$, we thus only need  to check the existence and non-negativity of the solutions $\nu_k,k\in J,$ of the equations: $-(\sum_{k\in J}\nu_k.\vartheta).F_j=\beta".F_j\geq 0,\forall j\in J$. Indeed, the non-negativity of these coefficients $\nu_k$ then implies that $\beta'.H\geq \beta".H$ for any divisor $H$ on $X$ different from any of the $F_j's$, which gives the movability of $\beta'$. By the negativity lemma above, we have the uniqueness and existence of the solutions $\nu_k,k\in J$. Now the non-negativity of all the $\nu_k's$ simultaneously is precisely the assertion of the (elementary but crucial) \cite{KM98}, Corollary 4.2, pp. 112-113. 

This solves the problem for the divisors $F_s$ which are not $f$-exceptional. Let now $F_s$ be an $f$-exceptional divisor, with $T:=f(F_s)$ of codimension $2$ at least in $Z$. Let $F':=\cup_{s'\in S'\subset S}$ be a connected component of the codimension $1$ locus of $f^{-1}(T)$. Let $t:=dim(T)\geq 0$, and $B_T:=f^*(B_{\vert T})$. We then deal with $F'$ exactly as before with $\cup_{k\in J}F_k$, just replacing the surface $X':=A^{d-1}.B^{n-d-1}$ by $X_T:=A^{n-t-1}.B_T^{t-1}$ if $t\geq 1$, and by $X_T:=A^{n-2}$ if $t=0$, and the classes $\vartheta_s$ by the classes $\vartheta_{s'}:=F_{s'}.X_T$. The same negativity lemmas then apply.
\end{proof}


\section{Harder-Narasimhan filtrations for orbifold bundles}

In this section, we establish the fact that the fibration on $X$ constructed in \cite{CP15} from an orbifold foliation of positive minimal slope relatively to a movable class $\alpha$ on $X$ depends  only on the orbifold pair $(X,D)$ and $\alpha$, and not on the Kawamata cover used to define it, and that moreover it is preserved under orbifold birational equivalence. These results are natural complements of those in \cite{CP15}. The arguments lead to more general statements however.

\begin{remark}\label{remarkindependenceHN} In Section \ref{Ginv}, we shall give a proof that the maximal slope of the HN-filtration of $\pi^*(\Omega^1(X,D)$ relative to $\pi^*(\alpha)$ can be computed directly from $(X,D)$ on $X$ using the sheaves $[\otimes^m](\Omega^1(X,D))$. This will give a second-simpler-proof that this slope is independent from the choice of $\pi$.  \end{remark}

We refer to \cite{CPe} and \cite{CP15} for the definitions and basic properties of slope, (semi-)stability, and Harder Narasimhan filtration of a (reflexive) sheaf $\cE$, relative to a movable class $\alpha$ on a projective manifold. We used in \cite{CP15} the following notion:

{\bf Conventions:} \label{conv} Let $\pi:Y\to X$ be a generically finite morphism between two complex connected projective manifolds, which is Galois for a certain finite group $G$ (which acts effectively and transitively on the generic fibres of $g$). Let $\alpha$ be a movable class on $X$. Then $\pi^*(\alpha)$ is still a movable class on $Y$.

\begin{lemma}\label{lmumin} Let $g:X'\to X$ be a surjective morphism between normal irreducible complex projective\footnote{This restriction may be replaced by the properness of $g$ in most places below.} spaces. Let $\cE$ be a reflexive coherent sheaf on $X$, and $\cF'\subset g^*(\cE)$ be a $G$-invariant saturated subsheaf. Denoting with $(.)^{sat}$ the saturation inside $g^*(\cE)$, we have:
: 

1. $\cF'=g^*(\cF)^{sat}$ with $\cF:=g_*(\cF')$ if $g$ is birational.

2. $\cF'=g^*(\cF)^{sat}$ if $\cF=(g_*(\cF'))^G$ if $g$ is finite, Galois of group $G$, and if $\cF'$ is preserved by the action of $G$ on $g^*(\cE)$. Here $(g_*(\cF'))^G$ is the subsheaf of $g_*(g^*(\cE))=\cE$ generated by the presheaf of $G$-invariant sections of $\cF'$ defined on open sets of $X'$ of the form $g^{-1}(U),U\subset X$ open. Moreover, the support of $(g^*(\cF)^{sat}/g^*(\cF))$ is contained in the codimension at least $2$ subset $S:=g^{-1}(Sing(X)\cup g(Sing(X')))$ of $X'$.\end{lemma}

\begin{proof} Assume first that $g$ is birational. 

The evaluation map 
$v:g^*(g_*(\cF'))\to \cF$ is then isomorphic over the generic point of $X$, and $\cF:=g_*(\cF')\subset \cE:=g_*(g^*(\cE))^G$. Since $\cF'$ is saturated in $g^*(\cE)$, we obtain the given isomorphism. 

Assume now that we are in the second situation. By its definition, the sheaf $\cF:=(g_*(\cF'))^G\subset \cE$ is generated locally by sections of the form $s:=\frac{1}{N}.(\sum_{h\in G} h^*(s'))$, for $s'$ a local section of $\cF'$. Thus it is locally of finite type, and hence coherent as an $\cO_X$-module. From the flatness of $g$ outside of $S$, we deduce that $\cF'=g^*(\cF)=g^*(\cF)^{sat}$ there (See \cite{Cl}, Lemme 2.13 for details). We conclude using the evaluation map $v:g^*(\cF)\to \cF'$.
\end{proof}

\begin{corollary}\label{csat} Assume that $g:X'\to X$ is as in \ref{conv} above,
 with $X,X'$ smooth. Assume that $g=h\circ k$, with $h:Y\to X$ finite, Galois, of group $G$, and $k:X'\to Y$ birational and $G$-equivariant. Let $\alpha$ be a movable class on $X$, and $g^*(\alpha):=\alpha'$ be its inverse image on $X'$. Let $\cE$ be a reflexive sheaf on $X$. Then $\mu_{\alpha,min}(\cE)=\mu
_{\alpha',min}(g^*(\cE))$, and the $g^*(\alpha)$-maximal destabilising sheaf $\cF'_{max}$ of $g^*(\cE)$ is equal to $g^*(\cF)^{sat}$, where $\cF$ is the $\alpha$-maximal destabilising sheaf of $\cE$ for $\alpha$. More generally, the Harder-Narasimhan filtration $HN_{g^*(\alpha)}(g^*(\cE))$ is equal to the saturation of $g^*(HN_{\alpha}(\cE))$ in $g^*(\cE)$, with equality of corresponding slopes.

Moreover, if $\cE'\subset g^*(\cE)$ is a $G$-invariant subsheaf such that $\cE'^{sat}=g^*(\cE)$, then: $HN
_{g^*(\alpha)}(\cE')^{sat}=g^*(W^{\bullet}(\cE',\alpha))$, for some suitable, uniquely defined, filtration $W^{\bullet}(\cE',\alpha)$ on $\cE$. 

Conversely, 
$HN
_{g^*(\alpha)}(\cE')=g^*(W^{\bullet}(\cE',\alpha))\cap \cE'$.\end{corollary} 

\begin{proof} Let $\cF'\subset g^*(\cE)$ be any $G$-invariant saturated subsheaf. Define $\cF_Y:=k_*(\cF')^{sat}$ (saturation in $h^*(\cE)$ here), and $\cF:=h_*(\cF_Y)=h_*(\cF_Y)^{sat}$. All of these sheaves are $G$-invariant. From the preceding lemma, we obtain that $\cF'=g^*(\cF)^{sat}$. Hence, using the last assertion of the preceding lemma, we obtain: $det(\cF')=g^*(det(\cF))+E$, where $E$ is a $g$-exceptional divisor on $X'$ (i.e: all of its components are mapped in codimension at least $2$ in $X$ by $g$). Thus: $g^*(\alpha).det(\cF')=\alpha.det(\cF)$, since $g^*(\alpha).E=0$. 

If we apply this to the $g^*(\alpha)$-maximal destabilising 
subsheaf $\cF'_{max}$ of $g^*(\cE)$, we obtain that $\cF'=g^*(\cF)^{sat}$, if $\cF$ is the $\alpha$-maximal destabilising subsheaf of $\cE$. The statements concerning the Harder-Narasimhan filtrations, the slopes, and $\mu
_{g^*(\alpha),min}(g^*(\cE))$ follow immediately. The last statement is proved in the same way, since the Harder-Narasimhan filtration of $\cE'$ relatively to $g^*(\alpha)$ is $G$-invariant.
\end{proof}

We shall need a slight generalisation, too:

\begin{corollary}\label{csat'} Let $g:X'\to X,G,Y,\cE,\cE'=g^*(\cE),\alpha$ be as in \ref{conv} above. 

Let moreover $h':X"\to X'$ be birational, $G$-equivariant, with $X"$ smooth. Let $\cE"\subset (g\circ h')^*(\cE):=(h')^*(\cE')$ be a reflexive $G$-invariant subsheaf such that the support of $((h')^*(\cE')/\cE")$ is contained in the exceptional divisor of $h'$. The conclusions of the preceding corollary then hold for $\cE"$ and $\alpha":=(g\circ h')^*(\alpha)$ in place of $g^*(\cE)$ and $g^*(\alpha)$. More precisely, we have: $HN^G_{(g\circ h')^*(\alpha)}(\cE"))=(g\circ h')^*((HN_{\alpha}(\cE))^{sat})\cap \cE"$, with equality of corresponding slopes.
\end{corollary}

\begin{proof} The conclusions of the preceding lemma apply to $(g\circ h')^*(\cE)$. The $(g\circ h')^*(\alpha)$-maximal destabilising $G$-subsheaf $\cF"$ of $\cE"$ is the intersection with $\cE"$ of the $(g\circ h')^*(\alpha)$-maximal destabilising $G$-subsheaf $\cF^+$ of $(g\circ h')^*(\cE))$, and so its $(g\circ h')^*(\alpha)$-slope is the same as the one of $\cF^+$, by the assumption on the support of $((h')^*(\cE')/\cE")$. The other assertions follow immediately.
\end{proof}


\subsection{Independence of the adapted cover.}

Let $(X,D)$ be, as always here, a smooth projective orbifold pair, and let $\pi:Y\to X$ be a Kawamata cover (see \cite{CP15},\S5) adapted to $(X,D)$. Recall that $Y$ is smooth, $\pi$ being $G$-Galois, Kummer and finite, for some group $G$. Moreover, the locally free sheaves $\pi^*(T(X,D))$ and its dual $\pi^*(\Omega^1(X,D))$ are defined on $Y$. More precisely, if $D=\sum_jc_j.D_j$, with $c_j=\frac{a_j}{b_j}$ as above, then $\pi$ ramifies to a certain order $m>0$, divisible by each of the $b_j's$ over $\Supp(D)=\cup_jD_j$ (and also over some additional components $H$ needed to construct $\pi$ globally).

Let $\alpha$ be a movable class on $X$, we thus get $HN^G_{\pi^*(\alpha)}(\pi^*(T(X,D)))$, and it is proved in \cite{CP15}, that the saturation $HN^G_{\pi^*(\alpha)}(\pi^*(T(X,D)))^{sat}$ inside $\pi^*(TX)$ is of the form $\pi^*(W^{\bullet}(D,\alpha,\pi))$ for a certain filtration $W^{\bullet}$ on $TX$ (which is, of course, not $HN_{\alpha}(TX)$ in general, if $D\neq 0$).

Our aim here is to prove that the filtration $W^{\bullet}(D,\alpha,\pi)$ is in fact\footnote{In accordance with the `orbifold' nature of these notions.} independent of the adapted cover $\pi$ and that it only depends on $D$ and $\alpha$. It will thus simply be written as $W^{\bullet}(D,\alpha)$ in the sequel. This permits, once $\pi$ is given, to reconstruct $HN_{\pi^*(\alpha)}(\pi^*(T(X,D))$ by simply intersecting $\pi^*(W^{\bullet}(D,\alpha))$ with $\pi^*(T(X,D))$.

So we consider a second Kawamata cover $\pi':Y'\to X$ adapted to $(X,D)$, of group $G'=Gal(Y'/X)$. Let 
\begin{equation}
Z\subset (Y'\times_XY)^n
\end{equation} be any component of the normalisation of fibered product, together with the natural projections $q:Z\to X$, $p:Z\to Y$ and $p':Z\to Y'$. it is acted on by its normal stabilizer $L\subset G'\times G$, which is onto on both $G$ and $G'$. Indeed: $Z$ is surjective and finite on both $Y'$ and $Y$, which are irreducible. Moreover, $q$ is $L$ Galois. The projections $p$ and $p'$ are respectively $H'$-Galois (resp. $H$-Galois) for some subgroups $H'$ (resp. $H$) of $G'$ (resp. $G$). The components of $(Y\times_ZW)^n$ are exchanged under the action of $(G\times G')/L$. 

Let finally $r:Z'\to Z$ be an $L$-equivariant resolution of $Z$. Let $q:=\pi\circ p\circ r=\pi'\circ p'\circ r$.

 of $(G\times H)/L$.

\begin{theorem}\label{tind} We have: 

1. $p'^*(\pi'^*(T(X,D)))=p^*(\pi^*(T(X,D))):=T$. 

2. $HN
_{q^*(\alpha)}(T)=(p\circ r)^*(HN
_{\pi^*(\alpha)}(\pi^*(T(X,D))))^{sat}=$

$(p'\circ r)^*(HN
_{\pi'^*(\alpha)}(\pi'^*(T(X,D))))^{sat}$.

3. We have equality of the corresponding slopes in the above filtrations.
\end{theorem}

\begin{proof} The first equality is easily computed outside of $S\subset X$ consisting of the union of $Sing(\Supp(D))$ and of the intersection of $\Supp(D)$ with the union of all the additional components $H$ of the ramification loci of $\pi$ and $\pi'$, since it just consist in lifting  `fractional' vector fields of the form $x^c.\partial x$, $c\in \Q\cap]0,1]$, under maps $x=z^{m.m'}=z^{m'm}$ for integers $m,m'$ such that $mc$ and $m'c$ are integers. Since both $\pi^*(T(X,D))$ and $\pi^*(T(X,D))$ are locally free, this equality extends to $Z$, and then to $Z'$. 

The last two statements are then immediate applications of Corollary \ref{csat}.\end{proof}

\begin{corollary}\label{cind} For given $(X,D)$, $\pi,G$ and $\alpha$ as above, the filtration $W^{\bullet}(D,\alpha)$ on $TX$ such that $\pi^*(W^{\bullet}(D,\alpha))=HN_{\pi^*(\alpha)}(\pi^*(T(X,D)))^{sat}$, the saturation being in $\pi^*(TX)$, is independent on $\pi$.
\end{corollary} 

\begin{proof} Indeed, for $r,\pi',p',p$ as above, the saturations in $(\pi\circ p\circ r)^*(TX)$ of $(p\circ r)^*(HN
_{\pi^*(\alpha)}(\pi^*(T(X,D))))$ and $(p'\circ r)^*(HN_{\pi'^*(\alpha)}(\pi'^*(T(X,D))))$ coincide with that of $HN
_{q^*(\alpha)}(T)$.
\end{proof}

Assume now that $\alpha.(K_X+D)<0$. Let $\cF\subset \pi^*(T(X,D))$ be the maximal destabilising 
subsheaf: it has a positive $\alpha$-slope. It is proved in \cite{CP15} that its saturation in $\pi^*(TX)$ is of the form $\pi^*(\cF_X)$ for some algebraic foliation $\cF_X$ on $X$, with $\cF_X=Ker(df)$ for some rational fibration $f:X\dasharrow Z$. From Corollary \ref{cind}, we obtain:

\begin{corollary}\label{indf} The foliation $\cF_X$ and the fibration $f$ (up to birational equivalence) depend on $D$ and $\alpha$, but not on $\pi$. \end{corollary}


\subsection{Birational orbifold invariance}

We shall apply the preceding  Proposition \ref{pom} together with the Corollary \ref{csat'} in order 
to show the birational orbifold invariance of the equivariant Harder-Narasimhan filtration of $\pi^*(T(X,D))$ relative to a movable class $\alpha$.

The situation is the following: Let $f(X',D')\to (X,D)$ be a birational orbifold morphism between smooth projective complex orbifold pairs.
We thus have, in particular: $f_*(D')=D$. 

Let $\pi:Y \to (X,D)$ and $\pi':Y'\to (X',D')$ be adapted Kawamata-covers of groups $G,G'$. We consider, as before, an irreducible component $T\subset (Y'\times _XY)^{n}$ of its normalised fibre-product $(Y'\times _XY)^{n}$  over $X$, together with its projections $\rho:T\to Y$ and $\rho':T\to Y'$. We have a natural action of $L\subset G'\times G$ on $T$. We also choose an $L$-equivariant resolution of singularities $r:R\to T$ for this action.

Define two locally-free sheaves $\cE':=(\rho'\circ r)^*((\pi')^*(T(X',D')))$ and $\cE:= (\rho\circ r)^*((\pi)^*(T(X,D)))$ on $R$.
By the previous Proposition \ref{pom}, we get a sheaf homomorphism $(\pi'\circ \rho'\circ r)^*(df): \cE'\to \cE$, since 
$f$ is an orbifold morphism.

\begin{lemma}\label{lcok} The cokernel of $(\pi'\circ \rho'\circ r)^*(df): \cE'\to \cE$ is contained in the exceptional divisor of $(\rho\circ r):R\to Y$. 
\end{lemma}

\begin{proof} This follows from the fact that $f_*(D')=D$ by the same argument used to prove Theorem \ref{tind}, which shows that the sheaves $\cE$ and $\cE'$ coincide in codimension one outside of the inverse image in $R$ of the exceptional divisor of $f$ in $X'$. \end{proof}

Applying now Corollary \ref{csat'} to $\cE',\cE$ and to $\rho\circ r:T\to Y$, writing $h=(f\circ \pi'\circ \rho'\circ r)=(\pi\circ \rho\circ r):R\to X$, we get:

\begin{proposition} \label{pindbir} Let $f:(X',D')\to (X,D)$ be an orbifold birational morphism as above. Then, taking saturation inside $\cE$, we have:
$$[(\rho' \circ r)^*(HN_{(f\circ \pi')^*(\alpha)}^{G'}(\pi'^*(T(X',D')))]^{sat}= [(\rho\circ r)^*(HN_{\pi^*(\alpha)}^G(\pi^*(T(X,D)))]^{sat}$$
\end{proposition}

We now have two filtrations $W^{\bullet}(D,\alpha))$ on $TX$ and $W^{\bullet}(D',\alpha')$ on $TX'$, with $\alpha':=f^*(\alpha))$, obtained by descending the Galois-invariant Harder-Narasimhan filtrations of $\pi'^*(T(X',D'))$ relative to $(f\circ \pi')^*(\alpha)$, and similarly for $(X,D)$ and $\pi^*(\alpha)$. See Theorem \ref{tind} and Corollary \ref{cind}  

\begin{corollary}\label{cindbir} The situation being the same as in \ref{pindbir}, we have: 
$W^{\bullet}(D',\alpha')=[f^*(W^{\bullet}(D,\alpha)]\cap TX'$. \end{corollary}

In other words: these filtrations are birationally invariant and independent on the Kawamata-covers, with respect to liftings of the given movable class.

\begin{remark}\label{rkawcov} When a fibration $f:(X,D)\to (Z,D_Z)$ (possibly birational) is given, together with a Kawamata cover $p:W\to Z$ adapted to $(Z,D_Z)$, it is always possible to construct, after suitable blow-ups of $(X,D)$, some $(X',D')$, and a Kawamata cover $\pi':Y'\to (X',D')$ adpated to $D'$, such that one has a regular map $f':Y'\to W$ with: $f\circ \pi'=p'\circ f':X'\to Z$. This permits to simplify the diagrams to come, by choosing $R=T=Y'$, but at the (small) expense of not taking an arbitrary $\pi:Y\to (X,D)$.
\end{remark}


\subsection{Tensor products and $G$-Harder-Narasimhan filtrations}\label{Ginv}

\

\

We consider in this section a normal connected complex space $Y$ together with a holomorphic faithfull action of a finite group $G$, we denote with $\pi:Y\to X:=Y/G$ its associated (normal) quotient $X$. We assume moreover that $Y,X$ are $\Bbb Q$-factorial, and that we are given on $Y$ a reflexive coherent sheaf $E$ with an equivariant linear action of $G$. We shall also denote by $\alpha\in Mov(X)$ a movable class (by $\Bbb Q$-factoriality, this notion is well-defined, and by $\pi^*(\alpha)$ its (movable) lifting to $Y$. See \cite{Cl}, 2.13,  and \cite{GKKP}, Appendix B for some more details.

We denote with $\pi_*^G(E)\subset \pi_*(E)$ the sheaf of $\cO_X$-modules associated with the presheaf of $G$-invariant sections of $E$ over open analytic subsets of $Y$ of the form $\pi^{-1}(U)$ for $U\subset X$ open. By averaging over $G$, we obtain an $\cO_X$-splitting $\pi_*(E)\to \pi_*^G(E)$ of the natural injection $\pi_*^G(E)\to \pi_*(E)$, which makes $\pi_*^G(E)$ a direct factor of $\pi_*(E)$. This shows that $\pi_*^G(E)$ is coherent.

Since $E$ is reflexive, so is also $\pi_*^G(E)$, by normality. Moreover, if $E=\pi^*(F)$ for some coherent sheaf $F$ on $X$, then $\pi_*^G(E)=F$, as seen from Lemma \ref{lmumin}. 

The aim of this section is to compute the maximal slope of $HN_{\pi^*(\alpha)}(E)$ in terms of the slopes of $HN_{\alpha}(\pi_*^G(\otimes^m(E)/Torsion))$, for $m>0$ either going to $+\infty$, or sufficiently great when $\alpha$ is a rational class, which are data computable on $X$, instead of $Y$. We do not know the answer to the much more delicate question of whether the associated maximal destabilizing subsheaf of $E$ can be described as an appropriate limit of those of $\pi_*^G(\otimes m E)$.

\medskip

{\bf Notations:} In order to simplify the reading, we shall use the following notations: $\otimes^m(E)/Torsion:=E_m, \pi_*^G(E):=E^X, \pi_*^G(E_m):=E_m^X$, $\pi^*(\alpha)=\alpha'$.

\begin{theorem}\label{thGslopes} For $X,Y,G, \alpha, E, r:=rk(E)$ as above, we have: $$(1) \mu_{\pi^*(\alpha), max}(E)=lim_{m\to +\infty}(\frac{1}{m}.\mu_{\alpha,max}(\pi_*^G(E_m))).$$  Assume now that $\alpha$ is an integral class, and that $m>m_0=r!(\alpha.R).\delta$, then:  $$(2) \mu_{\pi^*(\alpha), max}(E)=(\frac{1}{r!m}.\lceil r!\mu_{\alpha,max}(\pi_*^G(E_m))\rceil),$$ Here $R\subset X$ is the divisorial part of the branch locus of $\pi$, that is, of the complement of the smooth locus of $X$ over which $\pi$ is \'etale, and $\delta$ is the geometric degree of $\pi$.  \end{theorem}

Theorem \ref{thGslopes} is motivated by the:

\begin{corollary}\label{corGslopes} Let $(X,D)$ be a smooth projective orbifold pair, and $\pi:Y\to X$ a Kawamata cover associated to it, of Galois group $G$. Then: $$(1) \mu_{\pi^*(\alpha), max}(\pi^*(\Omega^1(X,D)))=lim_{m\to +\infty}(\frac{1}{m}.\mu_{\alpha,max}([\otimes^ m]\Omega^1(X,D))).$$

If $\alpha$ is integral, $dim(X)=n$, and $m>m_0=n!(\alpha.Supp(D)).\delta$, then: $$(2) \mu_{\pi^*(\alpha), max}(\pi^*(\Omega^1(X,D)))=(\frac{1}{n!m}.\lceil n!\mu_{\alpha,max}([\otimes^ m]\Omega^1(X,D))\rceil).$$. 
\end{corollary}

This corollary is an immediate consequence of Theorem \ref{thGslopes}, applied to the locally free $E:=\pi^*(\Omega^1(X,D))$, taking Theorem \ref{thmintegraltensors} into account. The fact  that $R$ can be replaced by $Supp(D)$ follows from the remark \ref{rR=D} below.

\begin{proof} (of Theorem \ref{thGslopes}) Claims (1) and (2): they are simple consequence of the following elementary:

\begin{lemma}\label{leminclus} Let $X,Y,G, \pi$ be as above. Then we have:

1. inclusions of sheaves:
$\pi^*(\pi_*^G(E))\subset E\subset \pi^*(\pi_*^G(E)\otimes \cO_X(R))$, and:

2. the inequalities of slopes:

$\mu_{\alpha,max}(\pi_*^G(E))\leq \mu_{\alpha,max}(E)\leq \mu_{\alpha,max}(\pi_*^G(E))+(\alpha.R)$.
\end{lemma}

\begin{proof} First claim: the first inclusion is clear, so we consider the second one. By reflexivity, we can reduce (as in the proof of Theorem \ref{thGslopes}) to a local situation in which $\pi:V\to U$ is given in coordinates adapted to $D$ and $\pi$, by: $\pi(y)=\pi(y_1,y_2,\dots, y_n)=(x)=(x_1=y_1^{ak},x_2=y_2,\dots, x_n=y_n)$, where $D=(1-\frac{b}{a}).(x_1=0)$, and $G=\Bbb Z_{ak}$, the action of its generator $g$ on $y_1$ being by multiplication by $\zeta$, a primitive $ak$-th root of unity, and being trivial on the other $y_j's$.

Now $E$, seen as a sheaf of $\cO_X$-modules, naturally decomposes as $E=\oplus_{h=0}^{h=ak-1}E_h$, where $E_h:=Ker (g-\zeta^h.1_E))$. Thus, for each $h$, we have: $(y_1^{ak-h}.E_h)\subset E_0$, which implies that $E\subset \oplus_{h=0}^{h=ak-1}(\frac{1}{y_1^{ak-h}}.E_0)\subset \frac{1}{\pi^*(x_1)}.\pi^*(\pi_*^G(E)),$ as claimed.

Let us show how to derive the second claim from the first: let $F$ be the maximal $\alpha$-destabilizing subsheaf of $E$. Let $F^+$ be its saturation in $\pi^*(\pi_*^G(E)\otimes \cO_X(R))$, and let $F^-$ be its intersection with $\pi^*(\pi_*^G(E))$. We thus have inclusions $F^-\subset F\subset F^+$ between sheaves of the same rank. We thus have the corresponding inequalities at the level of $\pi^*(\alpha)$-slopes.

Now the inclusions of the lemma shows that:

 $\mu_{\pi^*(\alpha), max}(F^-)\leq \mu_{\alpha,max}(\pi_*^G(E))=\mu_{\pi^*(\alpha), max}(\pi^*(\pi_*^G(E)))\leq$ 
 
 $\leq \mu_{\pi^*(\alpha), max}(E)):=\mu_{\pi^*(\alpha), max}(F)\leq \mu_{\pi^*(\alpha), max}(F^+)\leq$ 
 
 $ \leq\mu_{\pi^*(\alpha)}(F^-)+(\alpha.R)$
 
 The first equality follows from Lemma \ref{lmumin}, the last inequality from the fact that $F^+\subset F^-\otimes \pi^*(R)$. \end{proof}
 
In order to prove Theorem \ref{thGslopes}, claims (1) and (2), we now apply these inequalities to $E_m:=\otimes^m(E)/Torsion$, and get:

$\mu_{\alpha,max}(\pi_*^G(E_m))\leq \mu_{\alpha,max}(E_m)\leq \mu_{\alpha,max}(\pi_*^G(E_m))+(\alpha.R)$, and use the (difficult) fact that $\mu_{\pi^*(\alpha),max}(E_m)=m. \mu_{\pi^*(\alpha),max}(E)$. Dividing the inequalities by $m$, and letting $m$ tend to $+\infty$, we get the first assertion. The second assertion follows from the fact that, $\alpha$ being an integral class, $r!\mu_{\pi^*(\alpha),max}(E)\in \Bbb Z$, and that an interval of $\Bbb R$ of length less than $1$ contains no more than one integer. 
\end{proof}

\begin{remark}\label{rR=D}  In the corollary, we can replace $R$ by $Supp(D)$ because outside of $Supp(D)$, the sheaf $\pi^*(\Omega^1(X,D))$ is the lift to $Y$ of a $\Omega^1_X$, so that the additional ramification does not influence the result.
\end{remark}

\begin{question}\label{q} If $\alpha$ is rational, is it true that, for $m>0$ sufficiently great, one has: that $(\pi^*((E_m^X)^{max}))^{sat}=(E_m)^{max}$, the saturation taking place in $E_m$?
\end{question}

This question is motivated by the following:

\begin{proposition}\label{p} Assume that the question \ref{q} has a positive answer. Then, over $U:=X-R$, one has: $$(\pi_*^G(E^{max}))_{\vert U}=(((E_m^X)^{max})_{\vert U})^{\frac{1}{\otimes m}},$$
the notation meaning that the left-hand side is the $m-th$-tensor power of the right-hand side. 
\end{proposition}

\begin{proof}: We then have, over $U$, since $(E_m)^{max}=(E^{max})_m$: $(E_m^X)^{max}=(\pi_*^G((E_m)^{max}))=(\pi_*^G(E^{max})_m)=((\pi_*^G(E^{max}))_m)$, the last equality because $\pi$ is \'etale other $U$, so that $\pi_*^G$ commutes there with the tensorial operations.\end{proof}

\begin{remark} The preceding proposition shows that an affirmative answer to question \ref{q} solves our problem of determining $HN_{\alpha'}(E)$ from $HN_{\alpha}(\pi_*^G(E_m))$ if $\alpha$ is rational. Indeed: $E^{max}$ is determined by $(E_m^X)^{max}$, and one can then iterate the process by considering $E':=E/E^{max}$. Our motivation is of course the case where $E=\pi^*(\Omega^1(X,D))$. 
\end{remark}



\section{Criteria for the positivity of minimal slope}

The following criteria, which permits to check that $\mu_{\alpha,min}(\cF)>0$ by restricting to a general fibre and the base of a fibration, will be crucial for the proof of Theorem \ref{torc} in the next section. We start with a basic situation which will be gradually extended in several small steps. The extensions are motivated by the differentials on orbifold cotangent bundles explained in section \S.\ref{sscom} below.

Let $\tau:T\to W'$ be a fibration (with connected fibres) between complex projective connected normal spaces Let $r:R\to T$ and $s:W"\to W$ be resolutions of singularities.

The corresponding diagram is thus the following:

$$\xymatrix{
R\ar[r]^{r}&T\ar[d]^{\tau}\ar[r]^{\rho}&Y\\
W"\ar[r]^{s} & W'&\\}$$

\

Assume we have moreover movable classes $\alpha_R,\beta_R$ on $R$, and $\beta"$ on $W"$ such that: $(\tau\circ r)_*(\alpha_R)=0$, and 
$(\tau\circ r)_*(\beta_R)=s_*(\beta")$.

Let also an exact sequence of locally free sheaves be given on $R$: $$0\to \cF\to \cE\to (\tau\circ r)^*(\cG)\to 0,$$
where $\cG$ is locally free on $W'$.

\subsection{First Criterion}

\begin{theorem}\label{tsp} In the preceding situation, we have: $\mu_{\gamma_R,min}(\cE)>0$ if $\gamma_R:=k.\alpha_R+\beta_R$, for any real number $k>0$ sufficiently large, and if moreover, the following properties are satisfied: 

1. $\mu_{\alpha_R,min}(\cF)>0$, $\mu_{\beta",min}(s^*(\cG))>0$, and:

2. $\alpha_R$ is `big' on the `general' fibre $R_{w'}$ of $\tau\circ r$ (see definitions \ref{dbig} and \ref{dgen}). 
\end{theorem}

\begin{proof} (of Theorem \ref{tsp}) Let $\cQ$ be a quotient of $\cE$: it fits in an exact sequence: $$0\to  \cQ_F\to \cQ\to \cQ_G\to 0,$$
with $\cQ_F$ and $\cQ_G$ quotients of $\cF$ and $\cG$ respectively. Let $\gamma_R:=k.\alpha_R+\beta_R$ for some $k>0$. Then $\mu_{\gamma_R}(\cQ)$ is a linear combination with positive coefficients of $\mu_{\gamma_R}(\cQ_F)$ and $\mu_{\gamma_R}(\cQ_G)$.

We shall treat them separately.

\begin{lemma}\label{lsheaf'} For $k>0$ sufficiently large, $\mu_{\gamma_R}(\cQ_F)=k.\mu_{\alpha_R}(\cQ_F)+\mu_{\beta_R}(\cQ_F)>0$. Explicitely, one may choose any $k>k_0:=\frac{-\mu_{\beta_R,max}(\cF)}{\mu_{\alpha_R,min}(\cF)}$.

\end{lemma}

\begin{proof} Indeed, by assumption, $\mu_{\alpha_R}(\cQ_F)\geq \mu_{\alpha_R,min}(\cF)>0$ for any such $\cQ_F$. \end{proof}

The main observation is the following

\begin{lemma}\label{lsheaf} Either $det(\cQ_G)\geq(\tau\circ r)^*(det(\cQ'_G))$ for some sheaf $\cQ'_G$ on $W'$, or $det(\cQ_G)$ is an effective non-zero divisor when restricted to a generic fibre of $(\tau\circ r): R\to W'$. (The symbol $A\leq B$ between divisors means that $B-A$ is effective).
\end{lemma}

Before proving this lemma, let us show that it implies Theorem \ref{tsp}: in the first case, we have: $det(\cQ_G)=(\tau\circ r)^*(det(\cQ'_G))$, and thus $\mu_{\alpha_R}(\cQ_G)=0$, since $(\tau\circ r)_*(\alpha_R)=0$. Hence: $\mu_{\gamma_R}(\cQ_G)=\mu_{\beta_R}(\cQ_G)\geq \mu_{\beta_R,min}(\cG)>0$. In the second case, we have, by the bigness of $\alpha_R$ on the fibres of $\tau\circ r$, and the fact that $(\tau\circ r)_*(\alpha_R)=0$: $\mu_{\alpha_R}(det(\cQ_G))>0$. Thus $\mu_{k.\alpha_R+\beta_R}(\cQ_G)>0$ if $k>0$ is sufficiently large (once $\cQ$ is chosen so as to minimise $\mu_{\beta_R}$ among quotients of $\cE$).\end{proof}

\begin{proof} (of lemma \ref{lsheaf}) Let thus $\cM\subset (\tau\circ r)^*(\cG)$ be a subsheaf of rank $t>0$, and let $t'\leq t$ be the rank of the direct image sheaf $\cM':=(\tau\circ r)_*(\cM)$. If $t=t'$, we are in the first case, and in the second case if $t'<t$. Indeed: in the first case, $((\tau\circ r)^*(\cM')^{sat}/\cM)$ is torsion, and $((\tau\circ r)^*(\cM'))^{sat}=(\tau\circ r)^*((\cM')^{sat})$, the saturations being taken in $(\tau\circ r)^*(\cG)$ and in $\cG$ respectively. We thus have (after taking intersections with the relevant movable classes): $det(\cM)\leq (\tau \circ r)^*(det(\cM')^{sat})$ in this case.

In the second case, we consider the natural rational map $\varphi: R\to (\tau\circ r)^*(Grass(t,\cG))$, which we may assume to be regular (by modifying suitably $R$). The property $t'<t$ means that the image of the generic fibre $R_w$ of $\tau\circ r$ by $\varphi$ is positive-dimensional. We have: $\cQ_G=\varphi^*(Univ)$, and thus $det(\cQ_G)=\varphi^*(det(Univ))$, if $Univ\to Grass(t,\cG)$ is the universal bundle of rank $t$ on $Grass(t,\cG)$. The assertion in this second case thus immediately follows from the fact that $det(Univ)$ is ample on the fibres of $Grass(t,\cG)\to W$, and that the fibres of $\varphi(R)$ over $W$ are positive-dimensional.
\end{proof}


\subsection{Case of a non-saturated quotient}

We now consider the same situation $R,T,Y,W,W",\alpha_R,\beta_R,\beta, \cF,\cE,\cG$ as before, but assume instead that the exact sequence we have is: $$0\to \cF\to \cE\to \cH\to 0,$$

where $\cH\subset (\tau\circ r)^*(\cG)$ is a sheaf such that the quotient $(\tau\circ r)^*(\cG)/\cH$ is of torsion, with support contained in a divisor $F\subset R$ such that $(\tau\circ r)(F)\subsetneq W$. In particular, we thus have: $\alpha_R. det (\cQ_{\cH})=\alpha_R.(det(\cQ_{\cG}))$ if $\cQ_{\cH}\subset \cQ_{\cG}$ are quotients of $(\tau\circ r)^*(\cG)$ such that the cokernel has support in $F$. We get the following strengthening of Theorem \ref{tsp}, under an additional condition on $\beta_R$:

\begin{corollary}\label{csp} Assume that $\beta_R=\beta'_R+ k'\alpha_R$, where $k'>0$ is real, and $\beta'_R.F_j=0$, for any component $F_j$ of $F$. The conclusion of Theorem \ref{tsp} then still holds for $\cE$. (ie: $\mu_{k.\alpha_R+\beta_R,min}(\cE)>0$ if $k>0$ is sufficiently large). 
\end{corollary}

\begin{proof} Any quotient of $\cH$ injects into a quotient of $\cG$, with cokernel with support in $F$. The determinants of these quotients thus differ by a divisor supported on $F$, which has the same intersection with $\beta_R$ and $\beta'_R$ by assumption made that $(\tau\circ r)(F)\subsetneq W$. 
\end{proof}


\subsection{Equivariant version}\label{ssev}

The diagram we consider is now:

$$\xymatrix{
R\ar[r]^{r}&T\ar[d]^{\tau}\ar[r]^{\rho}\ar[rd]^{\sigma}&Y\ar[r]^{\pi}&X\ar[d]^{f}\\
W"\ar[r]^{s} & W'\ar[r]^{p'}&W\ar[r]^p&Z\\}$$

\

We assume here that both $\pi:Y\to X$, and $p:W\to Z$ are finite, Galois, of groups $G$ and $H$ respectively, and that $T$ is a component of the normalisation of $Y\times_ZW$, with projections $\rho:T\to Y$ and $\sigma:T\to W$. Thus $T$ is naturally $L$-Galois, equipped with an action of $L=G'\times H\subset G\times H$, the stabilizer of $T$, such that $\rho$ is finite, Galois, with group $G'\subset G$. Notice indeed that $T$ surjects finitely on both $Y$ and $X_W:=X\times _Z W$, which is irreducible because so are the generic fibres of $f$, and so also those of $f_W:X_W\to W$. Thus, $X_W\to X$ is Galois, of group $H$, so that every component of: $Y\times_ZW=Y\times_X( X\times_Z W)=Y\times_XX_W$ projects surjectively, not only on $Y$, but also on $X_W$, and is $G'$-Galois, for some $G'\subset G$. Thus $L$ has the above form $L=G'\times H$.

The composition $\sigma=p'\circ \tau$ of the diagram is the Stein factorisation of $\sigma$, with $\tau$ connected and $p'$ Galois, finite, of group $G'$, the kernel of the natural projection $L\to G'$, also the subgroup of $G$ preserving the fibres of $\tau$. 

We assume that $r$ is an $L$-equivariant resolution of $T$, and that $s$ is a $G'$-equivariant resolution of $W'$.

We then also consider locally free (although reflexive would suffice) sheaves $\cE$ on $Y$ and $\cG$ on $W$, $\cE$ (resp. $\cG$) being $G$-invariant (resp. $H$-invariant). Their liftings to $R$ will be denoted by $\cE_R$ and $\cG_R$, respectively.

We moreover assume that there exists an $L$-equivariant sheaf morphism $\Delta: \cE_R\to \cG_R$ with image $\cH\subset \cG_R$ such that the quotient $\cG_R/\cH$ is torsion on $R$, with ($L$-invariant) support contained in a divisor $F_R=(\pi\circ \rho\circ r)^{-1}(F), F\subset X$ a divisor such that $f(F)\subsetneq Z$. Since $\Delta$ is $L$-equivariant, we see that $\cF_R:=Ker(\Delta)=(\pi\circ \rho\circ r)^*(\cF)$, for some $G$-invariant subsheaf $\cF\subset \cE$.

We assume moreover that we have movable classes $\alpha,\beta'$ on $X$, and $\beta$ on $Z$ such that:

1. $f_*(\alpha)=0$, $f_*(\beta')=\beta$.

2. $\alpha$ is `big' on the general fibre $X_z$ of $f$.

3. There exists $k>0$ such that $\beta'=\beta"+k.\alpha$, where $\beta"$ is movable, and $\beta".F_j=0$, for every component $F_j$ of $F$.

4. $\mu_{\alpha_Y,min}(\cF)>0$.

5. $\mu_{\beta_W,min}(\cG)>0$.

Here $\alpha_Y,\beta_W$ denote the liftings to $Y$ and $W$ of $\alpha$ and $\beta$ respectively. (We adopt this notation for liftings, whenever defined, to any space in the diagram using the given maps. For example: $\beta'_Y:=\pi^*(\beta')$, but $\beta'_Y$ is not the lifting of $\beta$ in any natural sense).

The equivariant version we shall need is the following:

\begin{corollary}\label{cesp} In the above situation, and under the above hypothesis, we have, for any sufficiently large real number $k>0$: $\mu_{\gamma,min}(\cE)>0$, if $\gamma:=k.\alpha_Y+\beta_Y$.
\end{corollary}

\begin{proof} It mainly consists in lifting up and down the classes $\alpha,\beta',\beta$, and checking the above properties 1,2,3,4,5, together with the equivariance conditions at the levels of $R,T,Y,W'$, and in applying the arguments of Theorems \ref{tsp} and corollary \ref{csp}, in order to get the conclusion.

\begin{lemma}\label{lesp1} The properties 1,2,3 above imply:

1. $(\tau\circ r)_*(\alpha_R):=\alpha_{W'}=0$, 

1'. $(\tau\circ r)_*(\beta'_R):=\beta'_{W'}=(p\circ p')^*(\beta)$;

2. $\alpha_R$ is `big' on the general fibre $R_{w'}$ of $\tau\circ r:R\to W'$.

3. Let $\beta"_R:=\beta'_R-k.\alpha_R$. Then: $(\beta"_R).F_j^R=0,$ for each component $F_j^R$ of $F^R:=(\pi\circ \rho\circ r)^{-1}(F)$.

\end{lemma}

\begin{proof}  All of these are easily obtained from the projection formula.

1. Let $E\subset W'$ be an irreducible divisor. Then (up to positive multiplicative integers, which are degrees of finite maps): 
$$((\tau\circ r)_*(\alpha_R)).E=(\alpha_R). (\tau\circ r)^*(E)=((\pi\circ \rho\circ r)^*(\alpha)).((\tau\circ r)^*(E))$$
$$=\alpha.[((\pi\circ \rho\circ r)_*((\tau\circ r)^*(E))]=\alpha. [f^*((p\circ p')_*(E))]$$
$$=(f_*(\alpha)).[(p\circ p')_*(E)]=0, \forall E.$$

1'. The sequence of equalities is exactly the same as before. 

2. The map $(\pi\circ \rho\circ r):R_w'\to X_z,z:=(p\circ p')(w')$ is generically finite, and $\alpha_R$ 
on this fibre $R_{w'}$ is by definition, the lifting of $\alpha$ on $X_z$. This proves the assertion.

3. We have $(\tau\circ r)(F_j^R)\subset (p\circ p')^{-1}(f(F))\subsetneq W'$, the assertion then follows from the same computation as for the assertion 1 above.\end{proof}

We denote by $\cF_R,\cE_R, \cH_{R}, \cG_{W"}$ the liftings of $\cF,\cE,\cG,\cH$ to $R,R,R$ and $W"$ respectively.

Let us first observe that, by Lemma \ref{lmumin}, we have the equalities: 
$$ \mu_{\alpha_R,min}(\cE_R)=\mu_{\alpha_Y,min}(\cE)$$
$$ \mu_{\alpha_R,min}(\cF_R)=\mu_{\alpha_Y,min}(\cF)$$
 $$\mu_{\beta_{W"},min}(\cG_{W"})=\mu_{\beta_W,min}(\cG)$$
 
 We can thus lift everything to $R$, and work there, using the same hypothesis for their liftings as for $\alpha,\beta, \beta',\beta"$,
  just taking into account the $L$-equivariance properties.
  
  Let thus $\cQ$ be a $G\times H$ quotient of $\cE_R$. It fits into an exact sequence $$ \cQ_F\to \cQ\to \cQ_H\to 0,$$ for $G\times H$-invariant sequence of sheaves induced by $(\rho\circ r)^*(\Delta)$, which is an equivariant map of sheaves on $R$ (since $\Delta$ is $G$-equivariant). 
  
  The same argument as in Lemma \ref{lsheaf'} shows that for some explicit $k_0$, $\mu_{k'.\alpha_R+\beta'_R}(\cQ_F)>0$ if $k> k_0$. By the properties 1' and 3 of of the preceding Lemma \ref{lesp1}, we have, for every $k>0$:  $$\mu_{k.\alpha_R+\beta'_R}(\cQ_H)=\mu_{(k+k').\alpha_R+\beta"_R}(\cQ_H)\geq \mu_{k.\alpha_R+\beta'_R,min}(\cG)>0.$$

  Now, there are, as in Lemma \ref{lsheaf}, two cases concerning $\cQ_H$: by the property 1 of the preceding Lemma \ref{lesp1}: either $rank((\tau\circ r)_*(\cQ_H))=rank(\cQ_H)$, or $rank((\tau\circ r)_*(\cQ_H))<rank(\cQ_H)$. Using Lemma \ref{lsheaf}, we conclude in the first case that $\mu_{k.\alpha_R+\beta'_R}(\cQ_H)\geq \mu_{\beta'_R,min}(\cG_R)>0$. In the second case, we conclude from the same argument that $\mu_{k.\alpha_R+\beta'_R}(\cQ_H)>0$ for $k>0$ sufficiently large. \end{proof}


\subsection{Case of an orbifold morphism}\label{sscom}

The diagram of the preceding section will be constructed from an orbifold morphism $f:(X,D)\to (Z,D_Z)$, in which $(Z,D_Z)$ is the orbifold base of $f$, together with movable classes $\alpha, \beta'$ on $X$, and $\beta$ on $Z$ such that: $f_*(\alpha)=0,$ and $f_*(\beta')=\beta$.

Let $\pi:Y\to X$ and $p:W\to Z$ be Kawamata-covers adapted to $(X,D)$ and $(Z,Z_Z)$ respectively, of Galois groups $G,H$. 

Let $\cF\subset \cE:=\pi^*(T(X,D))$ be the maximal destabilising subsheaf of $\cE=\pi^*(T(X,D))$ relative to $\pi^*(\alpha)$. 
We assume that the associated fibration has $f$ as `neat' birational model. 

Let $T$ be a component of the normalisation of $Y\times_ZW$, together with the projections $\rho:T\to Y$ and $\sigma:T\to W$, so that $L\subset G\times H$ naturally acts on $T$. We take further an $L$-equivariant resolution $r:R\to T$ of $T$. In this way, the maps $\rho\circ r:R\to Y$ and $\tau\circ r:R\to W$ are Galois, of groups $H'\subset H$ and $G'\subset G$ respectively. We consider next $\sigma=\tau\circ p':T\to W$ the Stein factorisation of $\sigma$ as $\tau:T\to W'$ and $p':W'\to W$. It enjoys the equivariance properties of the diagram of the preceding section.

Let $\cG:= (\sigma\circ r)^*(p^*(T(Z,D_Z))),$ and $\Delta:=(\rho\circ r)^*(\pi^*(df)): \cE\to \cG$: this is an $L$-equivariant map by construction, its existence is garanteed by Proposition \ref{pom}. Let  $\cH:=[\Delta(\cE)]$, we thus have an inclusion $\cH\subset \cG$,    and an exact sequence:
$$0\to \cF\to \cE\to \cH\to 0.$$

Moreover, by Proposition \ref{pdpsf}, the cokernel sheaf $\cG/\cH$ is torsion, with support in a divisor $F\subset X$ `partially supported on fibres of $f$' (See Definition \ref{ddpsf}). This property permits to show, by Proposition \ref{prelmov}.(6), the existence of a movable class $\beta":=\beta_X$ on $X$ such that $\beta".F_j=0$ for all components $F_j$ of $F$, and to obtain from Corollary \ref{cesp}:

\begin{theorem}\label{tesp} In the preceding situation, assume also that $\mu^G_{\alpha_Y}(\cF)>0$, that $\mu_{p^*(\beta),min}(p^*(T(Z,D_Z)))>0$, and that $\alpha$ is `big' on the `general' fibre of $f$. Then $k.\alpha+\beta_X:=\gamma$ satisfies: $\mu_{\pi^*(\gamma),min}(\pi^*(T(X,D)))>0$, for any sufficiently large $k>0$ .
\end{theorem}



\section{Proof of Theorem \ref{torc}}\label{pmaint}

Before starting the proof, let us observe that the $3$ properties of the statement are `up'-birationally invariant\footnote{ `Up' because this works under blow-ups, not under blow-downs in general.} in the sense that if they hold on $(X,D)$, they also hold on $(X',D')$ for any birational $f:X'\to X$ which induces an orbifold morphism $f:(X',D')\to (X,D)$ (see Definition \ref{deforbmorph}), the movable classes involved being always the inverse images of the initial ones.

\begin{proof}

We shall show the implications $4\Longrightarrow 2 \Longrightarrow 3, 2'$, $3\Longrightarrow 3', 2'\Longrightarrow 3' \Longrightarrow 1\Longrightarrow 4$. All these implications are easy, except for the last one.

\medskip

$\bullet$ $4 \Longrightarrow 2$. This follows from the fact that $h^0(X,\cF\otimes A)=0$ if,  for some movable class $\alpha$, one has: $\mu_{\alpha,max}(\cF)< \alpha.A$. Now by our hypothesis: $$\mu_{\alpha,max}(\otimes^m(\Omega^1(X,D)))=-m.\mu_{\alpha,min}(T(X,D))<\alpha.A,$$ if $m>m(A):=\frac{\alpha.A}{\mu_{\alpha,min}(T(X,D)}$. 

\medskip

$\bullet$ $2 \Longrightarrow 3, 2'$, and: $2'\Longrightarrow 3'$. This follows from obvious inclusions of sheaves.
\medskip

$\bullet$ $3' \Longrightarrow 1$. This follows directly from the Example \ref{exinttensors}, and functoriality of sheaves of integral orbifold tensors, since $m(K_Z+D_Z)$ is a sheaves of integral orbifold tensors for sufficiently divisible $m>0$.





\medskip

$\bullet$ $1 \Longrightarrow 3$. This is the actual content of Theorem \ref{torc}, and its most involved part, requiring the preliminaries above. The proof will work by induction on $n:=dim(X)$. 

We start with $n=1$, so that $X$ is a curve, $K_X+D$ is not pseudo effective, and there is only one non-zero movable class $\alpha$ up to non-zero homothety. Then $\pi^*(\Omega^1(X,D))=\pi^*(K_X+D)$, and clearly: $\mu_{\alpha,min}(\pi^*(X,D))=-d.\alpha.(K_X+D)>0$, if $d>0$ is the geometric degree of $\pi$.

We thus assume that $n\geq 2$, and that the implication $1 \Longrightarrow 3$ holds whenever $dim(X)<n$. 

We shall produce a fibration $f:(X,D)\to (Z,D_Z)$ with $dim(Z)>0$, together with suitable classes $\alpha\in Mov^0(X/Z)$ and $\beta\in Mov^0(Z)$ and $\beta_X\in Mov(X)$ 
such that $f_*(\beta_X)=\beta$, so as to be in position to apply the results in \S\ref{sscom} to conclude that $\gamma:=k.\alpha+\beta_X$ satisfies the property $\mu_{\pi^*(\gamma),min}(T(X,D))>0$ if $k>0$ is sufficiently large.

We assume that $\pi:Y\to X$ is a Kawamata cover adapted to $(X,D)$. Let $\alpha\in Mov(X)$ be any class such that $\alpha.(K_X+D)<0$. Let $\cF\subset \pi^*(T(X,D))$ be the ($G$-invariant) $\alpha$-maximal destabilising subsheaf. It defines (by \cite{CP15}) a rational fibration $f_0:X_0\dasharrow Z_0$ with $\cF_{X_0}=Ker(df_0)$, where $\pi^*(\cF_{X_0})=\cF^{sat}$, the saturation inside $\pi^*(TX)$. We write here $X_0=X,Z_0=Z$, because we shall now consider (new) suitable birational models $X,Z$ of $X_0,Z_0$. Moreover, this fibration is non-trivial: $dim(Z)<n$. We put $d:=n-dim(Z)>0$.

We now chose a (new) `neat' birational model $f:(X,D)\to (Z,D_Z)$ of our original $f_0:X_0\dasharrow Z_0$, and a corresponding Kawamata cover of $(X,D)$. By Proposition \ref{pindbir}, the $G$-invariant Harder-Narasimhan fibration of the new orbifold cotangent sheaf is the inverse image of the initial one. We can thus assume from the very beginning that $f$ was a `strictly' neat orbifold morphism, in the sense that it is neat, and that, moreover, any $f$-exceptional divisor $F\subset X$ is $u$-exceptional for the birational map $u:X\to X_0$, and $\alpha=u^*(\alpha_0)$. This `strictness' property is obtained by first flattening the map $f_0$. 

Moreover, we have (using the notations of \cite{CP15}): $0>-\alpha.(det(\cF))=\alpha. (K_{X/Z}+D^{hor}-D(f))$, so that $K_{X_z}+D_z:=(K_{X/Z}+D)_{\vert X_z}$ is not pseudo-effective for $z\in Z$ generic. There thus exists $\alpha'\in Mov^0(X/Z)$ such that $\alpha'.(K_{X/Z}+D)<0$. Let us choose any of these. Let $\cF'\subset \pi^*(T(X,D))$ be the ($G$-invariant) $\alpha'$-maximal destabilising subsheaf. 

\begin{lemma}\label{lfibmin} In this situation, we have:

1. $\cF'\subset \cF$.

2. $\cF'$ is an orbifold foliation on $(X,D)$. It is thus algebraic, too.
\end{lemma}

\begin{proof} 1. Since $\cF= Ker(\pi^*(df):\pi^*(T(X,D))\to (f\circ \pi)^*(TZ))$, it is sufficient to show that $Hom(\cF',(f\circ\pi)^*(TZ))=0$, and thus, that $\mu_{\alpha',min}(\cF')=\mu_{\alpha'}(\cF')>\mu_{\alpha',max}((f\circ \pi)^*(TZ))$. By assumption, $\mu_{\alpha'}(\cF')>0$, since $\alpha'.(K_{X/Z}+D)<0$, and $\cF'$ is the maximal $\alpha'$-destabilizing subsheaf. On the other hand, , since $f_*(\alpha')=0$, we have: $\mu_{\alpha',max}(f\circ \pi)^*(TZ)\leq 0$, by Lemma \ref{lsheaf}, since the image of $\cF'$, if nonzero, were a subsheaf of $(f\circ\pi)^*(TZ)$, and should be, either generically coming from from $TZ$, or having an anti-effective determinant along the fibres of $\tau$.

2. The fact that $\cF'$ is an $(X,D)$-foliation in the sense of \cite{CP15} is the usual slope-argument (going back to Y. Miyaoka), since $\cF'$ is maximal-destabilizing with positive slope. The algebraicity statement is one of the main results of \cite{CP15}.
\end{proof}

 \medskip
 
 The next statements are immediate consequences of Lemma \ref{lfibmin}:

\begin{corollary}\label{cfibmin} Assume $d_{\alpha}:=d$ is minimal among all these $d_{\alpha}$, 

 $\alpha\in Mov(X)$ such that: $\alpha.(K_{X}+D)<0$. Then: 

1. $d>0$.

2. $\cF'=\cF$.

3. $\mu_{\alpha',min}(\cF)>0$.

\end{corollary}

\begin{remark}\label{rfibmin} This corollary applies in a relative setting $f:(X,D)\to Y$ as well, when $K_{\cF}$ is not pseudo-effective, by just considering classes $\alpha'$ in $Mov(X/Y)$ such that $K_{\cF_D}.\alpha'<0$, where $\cF_D:=\pi^*(Ker(df))\cap \pi^*(T(X,D))$. 
\end{remark}

We shall now fix such a class $\alpha'\in Mov^0(X/Z)$, and denote it with $\alpha$.

\

By assumption, for all fibrations $g:Z\dasharrow Y$ with $dim(Y)>0$, we have: $K_Y+D_Y$ is not pseudo-effective (replacing the initial $g$ by a `neat' birational model which induces an orbifold morphism $g:(Z,D_Z)\to (Y,D_Y)$ to its orbifold base $(Y,D_Y))$. Indeed, by the property \ref{lobom} recalled from \cite{Ca07}, the orbifold base of $g\circ f:(X,D)\to Y$ is also the orbifold base of $g:(Z,D_Z)\to Y$, because $f:(X,D)\to (Z,D_Z)$ is an orbifold morphism. 

We have two cases: either $dim(Z)=0$, and the property 3 of Theorem \ref{torc} is established, or $n>dim(Z)>0$, and we can apply induction: the property 3 is then satisfied by $(Z,D_Z)$. There thus exists $\beta\in Mov^0(Z)$ such that $\mu_{\beta,min}(p^*((T(Z,D_Z))>0$, if $p:W\to Z$ is any Kawamata cover adapted to $(Z,D_Z)$. 

We shall now show that we are in position to apply the results of \S\ref{sscom} to the classes $\alpha$ and $\beta$ thus constructed.
We thus consider again the diagram introduced in \S\ref{sscom} and \S\ref{ssev}, with the same meaning:

$$\xymatrix{
R\ar[r]^{r}&T\ar[d]^{\tau}\ar[r]^{\rho}\ar[rd]^{\sigma}&Y\ar[r]^{\pi}&X\ar[d]^{f}\\
W"\ar[r]^{s} & W'\ar[r]^{p'}&W\ar[r]^p&Z\\}$$

In order to apply Theorem \ref{tesp}, we thus choose the movable class $\beta_X=\beta"$ on $X$ constructed in Proposition \ref{prelmov}.(6). We thus have: $f_*(\beta")=\beta$ and $\beta".F_j=0$ for all components $F_j$ of the divisor $F$ `partially supported on the fibres of $f$' which contains the support of the Cokernel $(\cG/\cH)$ of the map $\Delta$ considered in the statement of Theorem \ref{tesp}, which thus applies and concludes the proof of Theorem \ref{torc}. 

We have thus proved the claimed conclusion $3$ of Theorem \ref{torc} on some orbifold birational model $g:(X',D')\to (X,D)$ of $(X,D)$. The following lemma \ref{descent} shows that the conclusion still holds for $(X,D)$ itself, when we apply this lemma to the natural inclusion $dg:T(X',D')\to T(X,D)$ lifted to a fibre-product of Kawamata covers adapted to them, as in Proposition \ref{pindbir}.

\begin{lemma}\label{descent} Let $g:X'\to X$ be a birational morphism between two projective connected complex manifolds, let $\cE'\to g^*(\cE)$ be an injection of torsionfree coherent sheaves on $X'$, which is generically isomorphic. Let $\alpha'\in Mov(X')$, and $\alpha:=g_*(\alpha')\in Mov(X)$. 

Then: $\mu_{\alpha',min}(\cE')\leq\mu_{\alpha,min}(\cE).$ 

In particular: $\mu_{\alpha,min}(\cE)>0$ if $\mu_{\alpha',min}(\cE')>0$.
\end{lemma}

\begin{proof} Let $\cQ$ be any non-zero quotient of $\cE$: it induces (by considering the intersection of the kernel corresponding to $g^*(\cQ)$ with $\cE'$, and the related quotient) a quotient $\cQ'$ of $\cE'$ together with an injection $\cQ'\to g^*(\cQ)$ which is generically isomorphic. We thus have: $det(g^*(\cQ))=det(\cQ')+E'$, where $E'$ is an effective divisor supported on the exceptional divisor of $g$. We have: $\mu_{g^*(\alpha)}(g^*(\cQ))= \mu_{\alpha}(\cQ))\geq \mu_{\alpha,min}(\cE)$. Moreover: 
$0<\mu_{\alpha',min}(\cE')\leq \alpha'.det(\cQ')\leq\alpha'.det(g^*(\cQ))=\alpha.det(\cQ),$ by the projection formula, which implies the claim.
\end{proof}

The fact that $\alpha$ can be choosen to be `movable-big' is a consequence of the following general:

\begin{lemma}\label{movbig} Let $\cE$ be a coherent torsionfree sheaf on the connected complex projective manifold $X$. Let $\alpha,\beta \in Mov(X)$, with $\beta$ big, be such that $\mu_{\alpha,min}(\cE)>0$. Let $\alpha_t:=\alpha+t\beta$. Then $\mu_{\alpha_t,min}(\cE)>0$ for some $\varepsilon>0$ and any $0\leq t\leq \varepsilon$.
\end{lemma}
\begin{proof} The property is established in \cite{CPe}, Lemma 5.6 for $\alpha$ rational, and in \cite{GKP}, Theorem 3.4 in general when $\cE$ is $\alpha$-stable. One deduces the general case by replacing $\cE$ by the successive $\alpha$-semi-stable quotients of its $\alpha$-Harder-Narasimhan filtration, and using Jordan-H\"older filtrations by stable sheaves of the same $\alpha$-slope of these quotients. We thus get a filtration of $\cE$ by $\alpha$-stable sheaves of positive slope. This property is then preserved for $\alpha_t$, if $0\leq t$ is sufficiently small.
\end{proof} 
\end{proof}


\subsection{Relative version}

The preceding result holds in a relative version as well, the proof being essentially similar. 

\begin{theorem}\label{torcrel} Let $f:(X,D)\to Z$ be a surjective morphism of complex projective manifolds\footnote{$Z$ being normal would actually suffice,here.} with connected fibres, $(X,D)$ being a smooth orbifold pair. Assume that the general orbifold fibre $(X_z,D_z)$ of $f$ is sRC. Let $\pi:Y\to X$ be any Kawamata cover adapted to $D$.

Then: $\mu_{\pi^*(\alpha),min}(\pi^*(T(X_z,D_z))>0$, for some $\alpha\in Mov^0(X/Z)$\footnote{See definition in \S\ref{srelmov}.}.
\end{theorem}

\begin{proof} It is essentially a relative version of the proof of $1 \Longrightarrow 3$ in the preceding section. We may, due to the lemma \ref{movbig}, replace $(X,D)$ by an arbitrary orbifold modification. Because $K_{X_z}+D_z$ is not pseudo-effective, by assumption, there exists a class $\alpha\in Mov^0(X/Z)$ such that $(K_X+D).\alpha<0$, and thus, taking the associated maximal destabilizing subsheaf $G$, and applying \cite{CP15} to it, we get an algebraic foliation $g:(X,D)\to Y$ over $Z$, with $d_{\alpha}:=dim(X_z)>0$, and $\mu_{\pi^*(\alpha)}(G)>0$. Choosing $\alpha$ as above such that $d_{\alpha}$ is minimal, we obtain also that $\mu_{\pi^*(\alpha),min}(G)>0$. If $d_{\alpha}:=dim(X_y)=dim(X_z)$, we are finished. Notice that the claim thus holds true if $d:=dim(X_z)=1$. 

Otherwise, we argue by induction on $d:=dim(X_z)\geq 2$, assuming the claim to hold in strictly smaller relative dimensions.
Let $(Y,D_Y)$ be the orbifold base of $f:(X,D)\to Y$, and $h:Y\to Z$ be the factorisation morphism. By assumption $dim(X)>dim(Y)>dim(Z)$, and the claim holds for both $g$ and $h$, by the induction hypothesis. From the construction of $g$ above, we have $\alpha\in Mov^0(X/Y)$ such that  $\mu_{\pi^*(\alpha),min}(F)>0$, and from the induction hypothesis, we get $\beta\in Mov^0(Y/Z)$ such that  $\mu_{p^*(\beta),min}(H)>0$, if $p:T\to Y$ is a Kawamata cover adapted to $(Y,D_Y)$, and $H:=[p^*(Ker(dh))\cap p^*(T(Y,D_Y))]^{sat}$, the saturation taking place in $p^*(T(Y,D_Y))$.

We may now, as in the preceding section, lift $\beta$ to $\beta'\in Mov^0(X/Z)$ in such a way that $g_*(\beta')=\beta$, and $\beta'.E=0$ for each component of the divisor $E'\subset X$, partially supported on the fibres of $g$, which supports Cokernel $(\cG/\cH)$, as in the preceding subsection. We conclude just by the same arguments as above.\end{proof}


\subsection{Relative differentials}

We shall prove also, in this relative context, a relative version (used in \S\ref{sbirstab} below) of the statement 2 in Theorem \ref{torc}. 

This is exactly the same statement as in \cite{Ca10}, Proposition 3.10 and Theorem , up to the fact that we consider the full orbifold tensors, instead of their `integral parts'. The proof being the same (and even slightly simpler), we will be sketchy and refer to loc.cit for further details.

Let $f:(X,D)\to (Z,D_Z)$ be a surjective orbifold morphism of smooth complex projective orbifold pairs with connected fibres, $(Z,D_Z)$ being the orbifold base of $(f,D)$. In this situation, and similarly to (but with modified notations)  \S\ref{sscom} and \S\ref{ssev} above, we can construct a commutative diagram:

$$\xymatrix{
T\ar[d]^{\sigma}\ar[r]^{\rho}\ar[rd]^{\sigma'}&Y\ar[r]^{\pi}&X\ar[d]^{f}\\
W\ar[r]^{q}&W'\ar[r]^p&Z\\}$$

with the following properties:

1. $\pi:Y\to X$ (resp. $p:T\to Z$) are Kawamata covers adapted to $D$ (resp. to $D_Z)$. 

2. $\sigma$ is connected (i.e: has connected fibres).

3. $q$ is generically finite, with $W$ smooth.

4. $p\circ q: W\to Z$ and $\pi\circ \rho:T\to X$ are Galois, $T$ being normal.

We also write: $\pi':=\pi\circ \rho$, and $p':=p\circ q\circ\sigma$.

\begin{theorem}\label{torcrel'} Let $f:(X,D)\to (Z,D_Z)$ be a surjective orbifold morphism of smooth complex projective orbifold pairs with connected fibres, $(Z,D_Z)$ being the orbifold base of $(f,D)$. Assume that the general orbifold fibre $(X_z,D_z)$ of $f$ is sRC. 
Then, in the above diagram:

1. One has, for any $m>0$: 
$$\sigma_*(\otimes^m\pi'^*(\Omega^1(X,D))))=\otimes^mp^*(\Omega^1(Z,D_Z)).$$ 

2. Moreover, if $L'\subset \otimes^m\pi'^*(\Omega^1(X,D)))$ is a pseudo-effective line bundle on $T$, then $L'\subset \sigma^*(\otimes^mp^*(\Omega^1(Z,D_Z)))^{sat}$, the saturation being taken inside $\otimes^m\pi'^*(\Omega^1(X,D))$.

\end{theorem}

\begin{proof} Let us prove the statement 1 first. Let $U\subset Z-Supp(D_Z)$ be the dense Zariski open subset over which $f$, as well as its restriction to each component of $D$ and of the nonempty intersections of the $D_j's$ is smooth. Over $U$, there is (after lifting to the Kawamata cover $Y$) a natural filtration of $\otimes^m(\pi^*(\Omega^1(X,D)))$ by subbundles with successive quotients $\otimes^Af^*(\Omega^1_U)\otimes ^BQ$, where $Q$ is the quotient bundle $[\pi^*(\Omega^1(X,D))/(f\circ \pi)^*(\Omega^1_U)]$, and $A\cup B=\{1,2,...,m\}$ is a partition in two subsets $A,B$. The expression $\otimes ^AE\otimes ^BF$ denotes the set of tensors of the form: $t_1\otimes...\otimes t_m$, with $t_k\in E$ (resp. $t_k\in F$), if $k\in A$ (resp. if $k\in B$).

Let $\alpha\in Mov^0(X/Z)$ be such that $\mu_{\pi^*(\alpha),min}(\pi^*(T(X_z,D_z)))>0$, for $z\in Z$ general. The existence of $\alpha$ is deduced from Theorem \ref{torcrel}. 

We thus get, for $m>0$: $H^0(X_z,\otimes^A(\pi^*(\Omega^1(X_z,D_z))))=\{0\}$

This implies that,  over $V':=(f\circ \pi)^{-1}(U')\subset Y$, for any open subset $U'\subset U$, one has:  
$$H^0(V', \otimes^m(\pi^*(\Omega^1(X,D))))=\pi^*(H^0(f^{-1}(U'),f^*(\otimes^m(\Omega^1_{U'}))).$$

Thus over $U$, one has the equalities of sheaves, for any $\ell >0$: $$ \sigma_*(\otimes^{\ell} \pi'^*(\Omega^1(X,D)))=p^*(\otimes^{\ell} \Omega^1_U)=\otimes^{\ell} p^*(\Omega^1(Z,D_Z)_{\vert U}).$$

We shall now show that these sections of $f^*(\otimes^m\Omega^1_U)$ over $f^{-1}(U)$, lifted to $T$ by $\pi$, extend as sections of $(p\circ \sigma)^*(\otimes^m (\Omega^1(Z,D_Z)))$. This will imply the claim.

By Hartog's theorem, it is sufficient to show this over the complement in $Z$ of the set of codimension $2$ consisting of singular points of $Supp(D_Z)$. We can thus compute locally on $Z$, and assume that:

1. $Supp(D_Z)$ is given in local coordinates $(z_1,...,z_p), p<n$ by the equation $z_1=0$, 

2. that we have local coordinates $(x_1,...,x_n)$ on $X$ such that the support of $D$ is given by the equation $x_1=0$, such that the map $f$ is locally given by: 

3. $f(x_1,...,x_n):=(z_1=x_1^t,z_2=x_2,...,z_p=x_p)$, and moreover that:

4. $D_Z=c'_1.(z_1=0)$, while $D=c_1.(x_1=0)$, $c,c'$ being as follows:

By the definition of the orbifold base of $(f,D)$, we shall, moreover, choose the chart of the $x$-coordinates centered at a point realising the minimum multiplicity of $(f,D)$ over $z_1=0$. 

This means that $c'_1=1-\frac{a}{t.b}$, if $c_1=1-\frac{a}{b}$, where $t$ is as in 3. above.

We thus get: $f^*\Big(\frac{dz_1}{z_1^{c'_1}}\Big)=f^*(z_1^{1-c'_1}.\frac{dz_1}{z_1})=t.x_1^{1-c_1}\frac{dx_1}{x_1}=t.\frac{dx_1}{x_1^{c_1}},$ and: $f^*(\frac{dz_i}{z_i})=(\frac{dx_i}{z_i})$, for $p\geq i\geq 2$. Let $c'_i=c_i=0$, for $p\geq i\geq 2$. Then: 

These equalities imply, symbolically, that, one has, for any multi-index $I:=(i_1,...,i_m)$, the equality: $f^*(\otimes_{k=1}^{k=m}\Big(\frac{dz_{i_k}}{z^{c'_{i_k}}_{i_k}}\Big))=\otimes_{k=1}^{k=m}\Big(\frac{dx_{i_k}}{x^{c_{i_k}}_{i_k}}\Big)$.

And so, symbolically: $f^*(\otimes^m(\Omega^1(Z,D_Z)))$ is,  over $z_1=0$, saturated inside $\otimes^m(\Omega^1(X,D_X))$, outside of a divisor on $X$ which is `partially supported on the fibres of $f$' (this is the same divisor $F$ as in definition \ref{ddpsf}, and the few lines before it). This implies that: $f_*(\otimes^m(\Omega^1(X,D_X)))=\otimes^m(\Omega^1(Z,D_Z)))$, in our situation of `sRC' fibres, which concludes the proof (after lifting the `symbolic' equalities above to $T$ by $\pi'$).

We now prove the statement 2. It is sufficient to show the claimed inclusion over the open set $U$ defined above. But this is an immediate consequence of the filtration introduced above on $\otimes^m(\pi^*(\Omega^1(X,D)))$, since the minimal slopes of the successive terms of the associated graduation are strictly negative, except for the last one. Thus the projections of $L'$ to the successive quotients, except for the last one, have to vanish, since $L'$ is pseudo-effective.
\end{proof}

The corollaries below will be used in the proof of Theorem \ref{tnugen}.

\begin{corollary}\label{cnu} The situation being as in Theorem \ref{torcrel'}, let $L'\in Pic(T)$ and $\ell>0$ be such that $L'\subset \otimes^{\ell}\pi^*(\Omega^1(X,D))$. Assume also that $h^0(T,mL'+\pi'^*(A)))\neq 0$ for infinitely many integers $m>0$, $A$ (resp. $B$) being ample divisors on $X$ (resp. $Z$). 

There then exist an embedding $M'\subset \otimes^{\ell}p'^*(\Omega^1(Z,D_Z))^{sat}$ for some $M'\in Pic(W)$ such that $$H^0(T,mL'+\pi'^*(A)))\subset\oplus^r\sigma^*(H^0(W,mM'+p'^*(kB)))),$$ for any sufficiently large $m>0$, the integers $k,r$ being the ones defined in Lemma \ref{ldescent}, depend only on $A$ and $B$.
\end{corollary}

\begin{proof} From the second claim of Theorem \ref{torcrel'} we deduce that $L'\subset \sigma^*(\otimes^{\ell}(p^*(\Omega^1(Z,D_Z))))^{sat}$. Let $M':=\sigma_*(L')^{**}$:  this is a reflexive rank-one coherent sheaf on $W$, by Lemma \ref{ltriv} below, applied to a generic fibre $F$ of $\sigma$ over $p^{-1}(U)$, and to the restriction of $L'$ to $F$. Thus $M'\in Pic(W)$. Moreover, we may-and shall-assume $L'$ to be saturated in $\otimes^{\ell}(\pi'^*(\Omega^1(X,D)))$. We thus have: $L'=\sigma^*(M')+E$, where $E$ is an effective divisor contained in the support of: $$\sigma^*(\otimes^{\ell}(p^*(\Omega^1(Z,D_Z))))^{sat}/\sigma^*(\otimes^{\ell}(p^*(\Omega^1(Z,D_Z)))),$$ and thus `partially supported on the fibres of $\sigma$', after Proposition\ref{pdpsf} (more precisely: $E$ is contained in $\pi'^{-1}(F)$, where $F$ is the divisor partially supported on the fibres of $f$ described above this Proposition).

From Lemma \ref{ldescent} below, we deduce that for any sufficiently large $m>0$, we have an injection $H^0(T,\sigma^*(mM')+A))\subset \oplus^r(\sigma^*H^0(W,mM'+p'^*(kB)))$.

Finally, from Lemma \ref{lsatur}, we get, for $m>0$ sufficiently large: $$H^0(T, mL'+A)=H^0(T,m.\sigma^*(M')+A),$$ which establishes the claims.
 \end{proof}

 \begin{lemma}\label{ltriv} Let $F$ be a connected complex projective manifold, and $L\in Pic(F)$ a subbundle of a trivial bundle. Assume that $L$ is pseudo-effective. Then $L$ is trivial.
 \end{lemma}
 
 \begin{proof} The dual $L^*$ is a quotient of a trivial bundle on $F$, and is thus generated by global sections. Assume that $L$ is not trivial. Then $L^*=\cO_F(D)$ for some nonzero effective divisor $D$ on $F$. 
 
 Let $A$ be ample on $F$. Since $h^0(F,mL+A)=h^0(F,-mD+A)=\{0\}$ if $m>0$ is sufficiently large, $L$ is not pseudo-effective. \end{proof}

 \begin{lemma}\label{ldescent} Let $\sigma:T\to W$ be a proper map between compact connected normal projective varieties, $T$ equipped with an effective  line bundle $A'$ and $W$ with an ample line bundle $B'$. 

There then exist positive integers $k=k(A',B')$ and $r=r(A',B')$ such that, for any torsionfree coherent sheaf $\cF$ on $W$, one has an injection:
$H^0(T,\sigma^*(\cF)\otimes A')\subset\oplus^r\sigma^*(H^0(W,\cF\otimes \cO_W(kB))).$\end{lemma}

\begin{proof} Notice that $\sigma_*(A')$ is a non-zero torsionfree sheaf which injects in its reflexive hull $A"$, with dual $(A")^*$. Now chose $k>0, r>0$ such that $(A")^*(k.B)$ is generated by $r$ of its global sections, and dualise the surjection $\oplus^r\cO_W\to (A")^*(k.B')$ so obtained. The claimed injection is then deduced from the natural sheaf injection $\sigma_*(A')\subset A"\subset \oplus^r\cO_{W}(k.B')$, since $H^0(T,\sigma^*(\cF)\otimes A')=\sigma^*(H^0(T,\cF\otimes \sigma_*(A')))$.\end{proof}

 \begin{lemma}\label{lsatur} Let $g:T\to W$ be a surjective connected morphism between complex projective manifolds. Let $E\subset T$ be an effective reduced divisor partially supported on the fibers of $g$. 
 
 There exists an integer $k>0$ such that, for any $N\in Pic(W)$, for any $s> k$, the natural injective map $H^0(T,N_k)\to H^0(T,N_s)$ is surjective, if $N_s:=f^*(N)+s.E+A$.
 
 More precisely: let $n:=dim(T)$, and $d:=dim(T)-dim(W)>0$. Let $A$(resp. $B$) be a very ample line bundle on $T$ (resp. $W$). 
Then one can chose $k:=A^d.E.g^*(B)^{n-d-1}$.
 \end{lemma}

 \begin{proof} This rests on the same considerations as in Proposition \ref{prelmov}. Denote by $S$ the smooth surface $A^{d-1}.f^*(B)^{n-d-1}\subset X$, for generic members $A,B$ (by abuse of notation) of the linear systems defined by the line bundles $A,B$. The reduced curve $S.E:=E'\subset S$ is partially supported on the fibers of $g_S:S\to B^{n-d-1}$, and the intersection number $S.E'.E"$ is thus strictly negative for each irreducible component $E"$ of $E'$. 
 
 Moreover, $g^*(N).E".S=0$, and so, for any $s>k$: $$N_s.E".S=s.E'.E".S+A.E".S\leq -s+A.E.S=-s+k<0.$$ This then implies that $H^0(E,N_s)=0$, and the surjectivity of the map: $H^0(S,N_k)\to H^0(S, N_s)$, by induction on $s>k$. Since the family of surfaces $S$ covers $X$, we have the same statement for the sections over $T$ (since the sections on $T$ of any line bundle are determined by their restrictions to the $S's$).\end{proof}

For the definition of the numerical dimension $\nu(X,L)$, we refer to \S\ref{sbirstab}. We shall now translate the preceding results using this notion.

\begin{corollary}\label{corcrel"} In the above situation, and for $L',M'$ defined as above, we have: $\nu(T,L')= \nu(W,M')$.
\end{corollary}

\begin{proof} We obviously have: $\nu(T,L')\geq \nu(W,M')$. In the other direction, we have from Corollary \ref{cnu}:  

$limsup_{s\to+\infty}\frac{h^0(T,mL'+\pi'^*(A))}{m^s}\leq r.limsup_{s\to+\infty}\frac{h^0(W,mM'+p'^*(kB))}{m^s}$, which is thus positive for the same values of $s$. 

This establishes the claim\end{proof}

 Let us notice a more general similar statement:

 \begin{corollary}\label{cnusat} In the situation of Lemma \ref{lsatur}, we have:
 
 $\nu(W,N)=\nu(T,g^*(N))=\nu(T,g^*(N)+\ell.E)$, for any $\ell\geq 0$.
 \end{corollary}
 
 \begin{proof} The first equality is shown in \ref{lcompnu}. For the second, use the inequalities, if $m.\ell\geq k$, and if $A'$ is ample on $T$ with $A'-(A+k.E)$ effective: $h^0(T, m.(g^*(N)+\ell.E)+A)=h^0(T,m.g^*(N)+k.E+A)$
 $\leq h^0(T, m.g^*(N)+A').$  The conclusion then follows from the definition of the numerical dimension.
 \end{proof}


\section{Orbifold Slope Rational quotient}\label{srq}

We shall show here the existence of a `rational quotient' (\cite{Ca92}, see also \cite{KMM} under the name `MRC-fibration') in the orbifold context.

\begin{theorem}\label{tRQ} Let $(X,D)$ be smooth, complex projective and connected. There exists (on some suitable birational model) an orbifold morphism which is a fibration $\rho:(X,D)\to (R,D_R)$ onto its (smooth) orbifold base $(R,D_R)$ which has the following two properties:

1. Its smooth orbifold fibres $(X_r,D_r)$ are $sRC$  ($X_r:=g^{-1}(r),r\in R$).

2. $K_R+D_R$ is pseudo-effective.

Of course, $R$=$X$ (resp. $R$ is a point) if and only if $(K_X+D)$ is pseudo-effective (resp. if and only if $(X,D)$ is $sRC$).

\smallskip

This fibration is unique, up to birational equivalence. It is, morerover, characterised by any one of the following two properties:

\smallskip

3. $dim(X)-dim(Z)$ is maximal among the fibrations $f:X\dasharrow Z$ such that $(X_z,D_z)$ is sRC for generic $z\in Z$.

4. $dim(Z)$ is maximal such that $K_Z+D_Z$ is pseudo-effective, among all fibrations $f:X\dasharrow Z$, $(Z,D_Z)$ being the orbifold base (on any `neat' orbifold birational model, here and also in 3. above)

\end{theorem} 

\begin{question}\label{quniqsrQ?} Is this fibration unique, up to {\bf orbifold} birational equivalence? This depends on the more general question \ref{qbireqorbbases}, see Theorem \ref{tdiffeqfibr} for a partial answer.
\end{question}

\medskip

The proof will be obtained below by combining Theorem \ref{torc} with the following factorisation criterion, and its Corollary \ref{ccomposition}:

\begin{proposition}\label{pfactorisation} Let $(X,D)$ be as above, together with two orbifold morphisms which are fibrations over their orbifold bases: $f:(X,D)\to (R,D_R)$ and $g:(X,D)\to (Z,D_Z)$. Assume that:

1. $K_R+D_R$ is pseudo-effective.

2. The generic orbifold fibre $(X_z,D_z)$ of $g$ is $sRC$.

Then: there exists a rational map $h:Z\to R$ such that $h\circ g=f$.
\end{proposition}

\begin{proof} 
Denote by $R_z:=f(X_z)\subset R$ the image of a general fibre $X_z$ of $g$ by $f$, and assume by contradiction that $dim(R_z)>0$, since the claims amounts to prove that $dim(R_z)=0$. The commutative diagram:

$$\xymatrix{
(X_z,D_z)\ar[r]^{j_z}\ar[d]^{f_z}&(X,D)\ar[d]^{f}\ar[r]^{g}&Z\\
R_z\ar[r]^{r_z} & (R,D_R)\\}$$

induces, for $m>0$ sufficiently divisible, a commutative diagram, the maps $r_z^*$ and $f_z^*$ being defined on $R_z^{reg}$ only:

$$\xymatrix{
(K_R+D_R)^{\otimes m}\ar[r]^{r^*_z}\ar[d]^{f^*} & K_{R_z}^{\otimes m}\ar[d]^{f_z^*}\\
\pi^*(\Omega^1(X+D))^{\otimes m}\ar[r]^{j^*_z}&\pi^*(\Omega^1({X_z}+D_z))^{\otimes m}\\}$$

The map $(f_z^*\circ r_z^*)$ is generically injective, and of generic rank $1$, because $dim(R_z)>0$. Thus $(j_z^*\circ f^*)$ also has generic rank $1$. By assumption, we have, for a certain movable class $\alpha$ on $X$ such that $g_*(\alpha)=0$: $\mu_{\alpha,max}(\Omega^1(X_z,D_z))<0$. Since $K_R+D_R$ is pseudo-effective, so is $(K_{R}+D_R)_z:=(K_R+D_R)_{\vert X_z}$, and so: $\alpha.(f^*(K_{R}+D_R)_z)\geq 0$, which contradicts the injection $\pi_z^*((K_{R}+D_R)_z^{\otimes m})\subset \pi_z^*((\Omega^1(X_z,D_z))^{\otimes m}),$ if $\pi_z:Y_z\to X_z$ is the restriction to $X_z$ of a Kawamata cover adapted to $(X,D)$, which is a Kawamata cover adapted to $(X_z,D_z)$. \end{proof}

\begin{corollary}\label{ccomposition} Let $(X,D)$ be smooth projective, together with an orbifold morphism $g:(X,D)\to (Z,D_Z)$ which is a fibration with $(Z,D_Z)$ its orbifold base. Assume that both $(Z,D_Z)$ and the general orbifold fibre $(X_z,D_z)$ of $g$ are $sRC$. Then so is $(X,D)$.
\end{corollary}

\begin{proof} Let $\rho:(X,D)\to (R,D_R)$ be (on some suitable birational model) an orbifold morphism which is a fibration onto its orbifold base, with $K_R+D_R$ pseudo-effective. We want to show that $dim(R)=0$. By Proposition \ref{pfactorisation}, we get a factorisation $h: Z\to R$ such that $\rho=h\circ g$. But now, $g$ being an orbifold morphism, $(R,D_R)$ is also the orbifold base of $h:(Z,D_Z)\to R$. Because we assumed $(Z,D_Z)$ to be $sRC$, we get: $dim(R)=0$ as claimed.
\end{proof}

\begin{proof} (of Theorem \ref{tRQ}) We proceed by induction on $n:=dim(X)$. When $n=1$, everything is clear, since $(X,D)$ is either $sRC$ or has pseudo-effective canonical bundle $K_X+D$. We thus assume that $n\geq 2$, and that the conclusion of Theorem \ref{tRQ} holds whenever $n'<n$.

Existence of $\rho$: If $K_X+D$ is pseudo-effective, $\rho:=id_X$. Otherwise, there exists (on a suitable birational model of our initial $(X,D)$) a fibration $g:(X,D)\to (Z,D_Z)$ as in the proposition \ref{pfactorisation} with $sRC$ orbifold fibres, and $dim(Z)<n$. The conclusion of the theorem \ref{tRQ} thus applies to $(Z,D_Z)$. By taking further birational models, we have a `rational quotient' $\rho':(Z,D_Z)\to (R,D_R)$. The Corollary \ref{ccomposition} now shows that the orbifold fibres of $\rho:=\rho'\circ f:(X,D)\to (R,D_R)$ are $sRC$. Because $f:(X,D)\to (Z,D_Z)$ was an orbifold morphism, the orbifold base of $\rho:(X,D)\to R$ is $(R,D_R)$. And $\rho$ thus possesses the two characteristic properties of a `rational quotient'. 

\medskip

Uniqueness of $\rho$: Let $\rho: (X,D)\to (R,D_R)$ and $\rho': (X,D)\to (R', D_{R'})$ be two fibrations having the two characteristic properties stated in Theorem \ref{tRQ}. By Proposition \ref{pfactorisation}, we have factorisations $h: R\to R'$ (resp. $h':R'\to R$) such that $\rho'=h\circ \rho$ (resp. such that $\rho=\rho'\circ h'$). Thus $R=R',\rho=\rho'$.

Uniqueness of $(R,D_R)$ up to birational equivalence: This is a general property of `neat' birational models of fibrations. See Section \ref{ssbireq}.

\medskip

Let us now check that $\rho$ is characterised by any of the properties 3. or 4. in the statement of Theorem \ref{tRQ}. 
Let $\rho:(X,D)\to (Z,D_Z)$ be the `sRC quotient', and let $g:(X,D)\to (Y,D_Y)$ be another (neat) fibration. 

Assume first that the generic fibres $(X_y,D_y)$ are sRC (resp. that $K_Y+D_Y$ is pseudo-effective). Then, from Proposition\ref{pfactorisation}, we deduce the existence of some $h: Y\dasharrow Z$ such that $\rho=h\circ g$ (resp. $h:Z\to Y$ such that $g=h\circ \rho$). We thus have the maximality properties of statements 3 (resp. 4) if and only if $g=\rho$
\end{proof}

We shall need a relative version of this `sRC'-quotient in the next section. We abuse notation in the sequel, by still writing $(X,D)$ for any suitable orbifold birational model of our initial $(X,D)$, in order to simplify notations. Also a factorisation of a fibration $f=g\circ r$ of $f:X\dasharrow Y$ will be a pair $(r,g)$ of fibrations $r:X\dasharrow Z$, $g:Z\to Y$ such that $f=g\circ r$. We shall always chose (after suitable orbifold modifications) $r:(X,D)\to (Z,D_Z)$ to be a neat orbifold morphism to its orbifold base, and similarly for $h$. Moreover, $Ker(dr), Ker(dg), Ker(f)$ determine $D, D_Z$, $D$-orbifold-foliations on $X,Z,X$ respectively (their construction is recalled just before Theorem \ref{tofops} below). These orbifold foliations will be denoted by $\cR_D,\cG, \cF_D$ respectively:

\begin{theorem}\label{trelsRQ} Let $(X,D)$ be a smooth orbifold pair, and $f:X\to Y$ be a fibration. There exists then a (birationally) unique factorisation $f=g\circ \rho_f$, with $\rho_f:(X,D)\to (Z,D_Z)$, and $g:Z\to Y$, such that, for $y\in Y$ general: $\rho_{f\vert X_y}: (X_y,D_y)\to (Z_y,(D_Z)_y)$ is the `sRC quotient' of $(X_y,D_y)$, the orbifold fibre of $f$ over $y$, $(Z_y,(D_Z)_y)$ being the orbifold fibre of $g:(Z,D_Z)\to Y$.
 
 Moreover, the $D_Z$-foliation $\cG$ defined by $Ker(dg)$ has a pseudo-effective canonical bundle. 
  
 The factorisation $f=g\circ \rho_f$ is characterised, among all factorisations $f=g\circ r$, by any one of the following two properties:

 2. The general fibers $(X_z,D_z)$ of $r$ are `sRC', and $dim(Z)$ is minimal for these properties.
 
 3. The $D_Z$-foliation $\cG$ determined by $Ker(dg)$ has pseudo-effective canonical bundle, and $dim(Z)$ is maximal for these properties.
 
 The factorisation $f=g\circ \rho_f$ is called `the slope rational quotient of $f$'.
 \end{theorem}

 \begin{proof} (of Theorem \ref{trelsRQ}) The uniqueness is clear. To show the existence, we proceed by induction on $d:=dim(X)-dim(Y)$, the assertion being obvious when $d=0$. If $K_{\cF}$ is pseudo-effective, $Y=X$ satisfies the assertions. Otherwise, by Corollary \ref{cfibmin}, there exists a factorisation $f=g\circ r$, with $dim(Z)<dim(X)$ such that $r:(X,D)\to (Z,D_Z)$ has `sRC' general fibres $(X_z,D_z)$. This Corollary can, indeed, be applied in the relative setting, by Remark \ref{rfibmin}. By induction hypothesis, we thus have a `slope RC quotient' factorisation $g=h\circ\rho_g$ for $g$, with $\rho_g:(Z,D_Z)\to (T,D_T)$ having `sRC' general fibres, and $h:(T,D_T)\to Y$ such that $K_{\cH}$ is pseudo-effective, where $\cH$ is the orbifold on $(T,D_T)$ associated to $Ker(dh)$. We obtain in this way By Corollary \ref{ccomposition}, the fibres of $\rho_f:=\rho_g\circ r:(X,D)\to (T,D_T)$ are `sRC', and $(T,D_T)$ is the orbifold base of this fibration, since both $\rho_g$ and $r$ have been chosen to be orbifold morphisms. The factorisation 
 $f:h\circ \rho_f$ thus fulfills the two conditions for being a relative `sRC quotient of $f$, since for any orbifold foliation $\cH$ on $(T,D_T)$ over $Y$, $K_{\cH}$ is pseudo-effective if and only if so is its restriction to any fibre $T_y$ (by \cite{CP15}, Theorem 6.2). \end{proof}

\begin{remark} In general, the orbifold rational quotient $\rho: (X,D)\dasharrow R$ is not `almost holomorphic' (see again the example 6.17, p. 859, of \cite{Ca07}: $(\Bbb P^2,L_1+L_2)$ if $L_i,i=1,2$ are two distinct lines). However, $\rho$ is almost holomorphic if $(X,D)$ is klt, by \cite{Ca07}, Theorem 9.19, p. 896 (the proof applies to our slightly more general situation).
\end{remark}


\section{Orbifold foliations of positive slope.}\label{spof}

Let $(X,D)$ be a smooth projective orbifold, $\pi:Y\to X$ a Kawamata cover adapted to $(X,D)$, and $\cF_D\subset \pi^*(T(X,D))$ a foliation on $(X,D)$. We say that $\cF_D$ is a $D$-foliation. If $f:X\dasharrow Z$ is a rational dominant fibration, it defines a foliation $\cF:=Ker(df)$ on $X$, and a $D$-foliation $\cF_D:=\pi^*(\cF)\cap \pi^*(T(X,D))\subset \pi^*(TX)$. Conversely, if $\cF_D\subset \pi^*(T(X,D))$ is a $D$-foliation  (see \cite{CP15} for this notion), it defines a foliation $\cF$ on $X$ characterised by the equality: $\cF_B^{sat}=\pi^*(\cF)$, where $\cF_D^{sat}$ is the saturation in $\pi^*(TX)$ of $\cF_D$. And $\cF$ is algebraic means that it leaves are algebraic, or equivalently, that $\cF=Ker(df)$ for some rational dominant fibration $f:X\dasharrow Z$. In this case,

\begin{theorem}\label{tofops}  Assume that $\cF_D\subset \pi^*(T(X,D))$ is a $D$-foliation, and that $\mu_{\alpha',min}(\cF_D)>0$ for some movable class $\alpha$ on $X$, where $\alpha':=\pi^*(\alpha)$ is movable on $Y$. Then:

1. $\cF$ is algebraic, let $f:X\dasharrow Z$, be such that $\cF=Ker(df)$.

2. On any `neat' orbifold birational model $f':(X',D')\to Z'$ of $f$, the generic orbifold fibre $(X'_z,D'_z)$ of $f'$ is sRC.

Conversely, if $(f,D)$ posesses the above property 2, the $D$-foliation $\cF_D$ associated to it\footnote{By the construction recalled before the statement of Theorem \ref{tofops}.} has $\mu_{\alpha',min}(\cF_D)>0$ for some $\alpha$ movable on $X$, and for any Kawamata cover $\pi$ adapted to $(X,D)$.
\end{theorem}

\begin{proof} The proof is a direct adaptation to the orbifold context of the proof given in \cite{CP15} when $D=0$. From \cite{CP15}, we know that the foliation $\cF$ on $X$ defined by the saturation of $\cF_D$ in $\pi^*(TX)$ is algebraic. Let $f:(X,D)\to Y$ be a neat orbifold birational model of the rational fibration $f:(X,D)\to Y$ defined by $\cF_D$. We know from \cite{CP15} that its generic orbifold fibres $(X_y,D_y)$ have a canonical bundle $K_{X_y}+D_y$ which is not pseudo-effective. Let $\rho_f:(X,D)\to (Z,D_Z)$ be its `sRC quotient', with the factorisation $f=g\circ \rho_f$, and $g:(Z,D_Z)\to Y$. We thus know that $K_{\cG}$ is pseudo-effective, if $\cG\subset p^*(T(Z,D_Z)$ is the $D_Z$-foliation defined by the foliation  $\cG_Z:=Ker(dg)\subset TZ$ on $Z$. We have\footnote{On a suitable finite cover of $X'$ dominating a Kawamata cover $q:Z'\to Z$ adapted to $(Z,D_Z)$ still denoted $X'$, see the constructions made in \S\ref{sscom}.} a natural derivative map: $\pi^*(d\rho_f): \pi^*(\cF_D)\to (\rho_f)^*(q^*(\cG))$ which is generically surjective (note that this map is, generically on $X$, nothing but $df:\cF\to \rho_f^*(\cG_Z)$). 

Assume that $dim(Z)>dim(Y)$, or equivalently, that $\cG\neq 0$.

We thus have: $0<\mu_{\pi^*(\alpha),min}(\pi^*(T(X,D))\leq \mu_{\pi^*(\alpha),min} (\rho_f^*(q^*(\cG)))$. But this contradicts the fact that $K_{\cG}$ is pseudo-effective.\end{proof}


\section{Birational stability of the orbifold cotangent bundle}\label{sbirstab}

\subsection{Numerical dimension}

\begin{definition} Define, if $A$ sufficiently ample, and $L\in Pic(X)$, the `numerical dimension' of $L$ to be: 

$\nu(X,L):=max\{k\in \Bbb Z\vert \overline{lim}_{m>0}(\frac{h^0(X,mL+A}{m^k})>0\}$.
\end{definition}

 Recall some easy properties:

{\bf 1} $\kappa(X,L)\leq \nu(X,L)\in \{-\infty,0,1,...,n\}$.

 {\bf 2} $\nu(X,L+P)\geq \nu(X,L)$ if $P\in Pic(X)$ is pseff.
  
  In general, there is no further relationship between $\nu(X,L)$ and $\kappa(X,L)$, except in the following extremal case:

   {\bf 3}    $\kappa(X,L)=n$ if $\nu(X,L)=n$.

   We shall need the following easy property:
   
   \begin{lemma}\label{lcompnu} Let $\pi:T\to X$ be a proper morphism between two normal connected complex projective varieties, and let $L\in Pic(X)$, together with a sufficiently ample line bundle $A$ on $X$. Then $\nu(X,L)=\nu(T,\pi^*(L))=max\{s\in \Bbb Z\vert limsup_{s\to +\infty} \frac{h^0(T, \pi^*(mL+A)}{m^s}>0\}$.
   \end{lemma}
   
   \begin{proof} We only prove the first equality, which obviously implies the second one, which we now prove. Let thus $A'$ be ample on $T$. From Lemma \ref{ldescent}, we have a natural injection $\pi_*(A')\subset \oplus^r\cO_X(kA)$, for some integers $k>0,r>0$. This imples the inequality: $h^0(T,\pi^*(m.L)+ A')\leq r.h^0(X,m.L+kA)$, which easily implies that $\nu(T,\pi^*(L))\leq \nu(X,L)$, the reverse impliction being obvious.\end{proof}

   A central result is the following:

     {\bf Theorem:} ([BDPP], [Nak]) $L$ pseff iff $\nu(X,L)\geq 0$.

    \subsection{Maximal numerical dimension of a coherent sheaf}

The following notion was introduced in \cite{Ca10}:

     \begin{definition} Let $X$ be a connected normal complex projective variety, and $\cF$ a torsion free coherent sheaf on $X$. 
     Define: $\nu^+(X,\cF):=max\{\nu(X,L)\vert L\in Pic(X), L\subset \otimes^m\cF, m>0\}.$
     \end{definition}

     We shall need the following elementary property:

   \begin{lemma}\label{lnugal} Let $\pi:T\to X$ be a generically finite Galois map between connected complex projective manifolds. Let $\cF'$ be a torsionfree coherent sheaf on $T$, equipped with an equivariant action of the group $Gal(T/X):=G$. 
   
   Define:$$\nu^+(X, \cF',G):=max\{\nu(X,L)\vert L\in Pic(X), \pi^*(L)\subset \otimes^m(\cF'), m>0\}$$
   
   Then: $\nu^+(X, \cF',G)=\nu^+(T,\cF')$.
   \end{lemma}
   
   \begin{proof} The inequality $\nu^+(X, \cF',G)\leq \nu^+(T,\cF')$ is obvious. In the reverse direction, let $L'\in Pic(T), L'\subset  \otimes^m\cF'$. We assume that $L'$ is pseudo-effective, since otherwise the claim is clear.
   
   Consider $L":=\otimes_{g\in G} g^*(L')\subset \otimes^{mN}\cF',$ if $N=Card(G)$. Since $L"$ is $G$-invariant, there exists $L\in Pic(X)$ such that $L"=\pi^*(L)+E^+-E^-$, where $E^+,E^-$ are effective divisors supported on the exceptional locus of $\pi$, and without common components. By Hartog's theorem, we thus have: $\nu(T,L")\leq \nu(T,\pi^*(L)+E^+)=\nu(X,L)$.
   
   On the other hand, we also have: $\nu(T,L")\geq \nu(T,L')$, since each of the $g^*(L')$ is effective. This implies that $\nu(T,L')\leq \nu(X,L)$, as claimed.
   \end{proof}

We shall apply this notion to the orbifold cotangent and canonical bundles.

\begin{definition}\label{dnu+} Define: $\nu(X,D):=\nu(X,K_X+D)\geq \kappa(K_X+D)$.
$$\nu^+(X,D):=max\{\nu(X,L)\vert L\in Pic(X),\pi^*(L)\subset \otimes^m(\pi^*(\Omega^1(X,D))), m>0\}$$

It is independent on $\pi$, since $\nu(Y,\pi^*(L))=\nu(X,L)$, by Lemma \ref{lcompnu}.

From Lemma \ref{lnugal}, we also see that: $\nu^+(X,D)=\nu^+(T,\pi'^*(\Omega^1(X,D))$, for any Kawamata cover $\pi$ adapted to $(X,D)$, and any $\pi'=\pi\circ \rho$ as in Theorem \ref{torcrel'}.
\end{definition}

 The same notions have been introduced in  \cite{Ca10}, using the (essentially equivalent) sheaves $[S^m](X,D)$, and also in \cite{Ca09}, but using $\kappa$ instead of $\nu$ there. Using $\kappa$ however presently leads to conjectures, rather than theorems, as below.
 
  \medskip

 Obviously:  $\nu^+(X,D))\geq \nu(X,D)$.   We shall, in the next two sections, revert this inequality.

\subsection{$K_X+D$ pseudoeffective:  $\nu^+(X,D)=\nu(X,D)$}

\
\medskip

We just recall here:

\begin{theorem}\label{tnupseff}(\cite{CP15}, Theorem 7.3) Let $(X,D)$ be a smooth (projective, connected, complex) orbifold with $K_X+D$ pseudoeffective. 
Then $\nu^+(X,D)=\nu(X,D)$
\end{theorem}

An important special case is:

\begin{corollary}(\cite{CP15}, Theorem 7.7) Let $(X,D)$ be a smooth (connected complex projective) orbifold pair. If $\nu(X,D)=0$, then $\nu^+(X,D)=0$. 

If $K_X+D\equiv 0$, then: $-L$ is pseudo-effective, if $L\in Pic(X)$ is such that $\pi^*(L)\subset \otimes^m(\pi^*(\Omega^1(X,D))$ for some $m>0$. In particular: $L\cong \cO_X$ if $h^0(Y,\pi^*(L))\neq 0$, and: $\kappa(Y,\pi^*(L))\leq 0$.
\end{corollary}

\smallskip

\subsection{General case: $\nu^+(X,D)=\nu(R,D_R)$}

\begin{theorem}\label{tnugen} Let $(X,D)$ be a smooth (projective, connected, complex) orbifold pair. Let $r:(X,D)\to (R,D_R)$ be its `slope Rational Quotient' (on a suitable orbifold birational, strictly neat, model). 
Then $\nu^+(X,D)=\nu(R,D_R)$
\end{theorem}

\begin{proof} The inequality $\nu^+(X,D)\geq \nu(R,D_R)$ is indeed obvious. 

The reverse inequality is an immediate consequence of the corollary \ref{corcrel"}, combined with the equalities $\nu^+(X,D)=\nu^+(T,\pi'^*(\Omega^1(X,D)))$ and $\nu^+(Z,D_Z)=\nu^+(W,p^*(\Omega^1(Z,D_Z)))$ observed in Definition \ref{dnu+}, the last one applied to $(Z,D_Z)=(R,D_R)$. The notations $T,W$, are those of the diagram introduced before the statement of Theorem \ref{torcrel'}.\end{proof}

In particular, let us stress that:

\begin{corollary}\label{csRC} Let $(X,D)$ be a smooth (projective, connected, complex) orbifold pair. 

Then: $(X,D)$ is sRC if and only if $\nu^+(X,D)=-\infty$.
\end{corollary}

We recover in a more natural way, and in a more general context than in \cite{CP15} the criterion for being of Log-general type:

\begin{corollary}\label{cnubig}(\cite{CP15}, Theorem 7.6) Let $(X,D)$ be a smooth (projective, connected, complex) orbifold pair. 

Assume that $\nu^+(X,D)=n$. Then: $\kappa(X,D)=n$.
\end{corollary}

\begin{proof} Since $n=\nu^+(X,D)=\nu(R,D_R)\leq dim(R)\leq n$, we have $dim(R)=n$, and $X=R$, which means that $(X,D)=(R,D_R)$, and so $\nu(X,D)=\nu(R,D_R)=n$, as asserted. \end{proof}


\

\section{Orbifolds with nef or ample anticanonical bundle.}

\

We prove here Theorem \ref{tfsrc} as an application of Theorem 2.11 in \cite{CP13}. Recall the statement to be proved:

\begin{theorem}\label{tfsrc} Let $(X,D)$ be a smooth orbifold pair which is klt\footnote{This means that all coefficients of the components of $D$ are less than $1$, strictly.}, and Fano (ie: $-(K_X+D)$ is ample on $X$). Then $(X,D)$ is slope rationally connected.
\end{theorem} 

 The klt condition is not superfluous by the following:
 
 \begin{example} \label{exfanonotsrc} Let $(X,D):=(\Bbb P_n,D_k=H_1+...+H_k)$,  $H_j$ hyperplanes in general position, $2\leq k\leq n$: it is Fano, but `purely logarithmic' (i.e: the coefficients of the components of $D=D_k$ are all equal to $1$). 

Its slope rational quotient  is the linear projection $\pi:\Bbb P_n\to \Bbb P_{k-1}$ centered at the intersection of the $H_j's$. The orbifold base of this projection is $(\Bbb P_{k-1}, D'_k)$, with $D'_k=\pi(D_k)$, which has trivial Log-cotangent bundle.

We thus have: $\nu^+(X,D)=0$, $h^0(X,\Omega^{k}_X(Log(D)))=1$, and: 

$h^0(X,\Omega^q_X(Log(D)))=0$, for all $0<q\neq k$. \end{example}

\begin{proof} We need to show that if $f:X\dasharrow Z$ is any dominant rational map with $dim(Z)>0$, then its orbifold base has a non pseudo-effective canonical bundle, on any `neat' birational model. Let $g:(X',D')\to (X,D)$ be a birational morphism, with $(X',D')$ smooth such that $D'$ is the strict transform of $D$ in $X'$, together with a fibration $f:X'\to Z$ on a smooth projective variety $Z$. We assume $f:(X',D')\to Z$ to be `neat' and its orbifold base $(Z,D_Z)$ to be smooth. Because we assumed $(X,D)$ to be klt, we can write: $g^*(K_X+D)=K_{X'}+D'+\Delta'-E^{\bullet}$, where $\Delta',E^{\bullet}$ are supported on the exceptional locus of $g$ and without common component, and where $(X',(D'+\Delta'))$ is again klt. Let $H$ be ample on $Z$. Since $A:=-(K_X+D)$ is ample, we can write $g^*(A)=\vartheta .E'+A'+\varepsilon.f^*(H)$, where $\vartheta>0,\varepsilon>0$ are chosen to be sufficiently small, and $E'$ is the reduced support of the exceptional divisor of $g$.We can thus write: $0=g^*(K_X+D+A)=K_{X'}+D'+(\Delta'+\vartheta .E')+A'+\varepsilon.f^*(H)-E^{\bullet}$.
Since $A',H$ are ample and $(\Delta'+\vartheta .E')$ is supported on $E'$ and $(X',(\Delta'+\vartheta .E'))$ is klt for $0<\vartheta$ sufficiently small, we can chose $\Bbb Q$-divisors linearly equivalent to $A'$ and $H$ in such a way that: 

1. $D":=D'+E'+A'+\varepsilon.f^*(H)$ has snc support and is such that: $(X',D")$ is lc. 

2. $A'$ is $f$-horizontal (for example: its support is irreducible)

3. $D_Z+H+f(D'^{vert})$ has snc support, where $(\bullet)^{hor}$ and $(\bullet)^{vert}$ stand for the $f$-horizontal and $f$-vertical parts of a divisor $(\bullet)$ on $X'$.

We then have: $K_{X'}+D"=E"+E^{\bullet}$, where $E"$ is supported on $E'$ and $(X',D")$ is lc. Moreover: $g:(X',D")\to (X,D)$ is an orbifold morphism (because we have equipped all the components of $E'$ with multiplicities $+\infty$, which was the reason to replace $(\Delta'+\vartheta .E')$ by $E'$). Notice that $g: (X',D")\to (X,D)$ is no longer an orbifold birational equivalence, because of the addition of $A'+\varepsilon.f^*(H)$.

Let now $(Z,D'_Z)$ be the orbifold base of $f:(X',D")\to Z$, which is also `neat'. We have, for its orbifold base $(Z,D'_Z)$: $D'_Z=D_Z+ \varepsilon.H$, by our generic choice of the $\Bbb Q$-divisor $H$ on $Z$.

Now, by \cite{CP13}, Theorem 2.11: $K_{X'/Z}+(D")^{vert}-D(f,0)$ is pseudo-effective. Here $D(f,0)$ is an effective $f$-vertical divisor for the definition of which we refer to loc.cit. In particular, $K_{X'/Z}+D"-f^*(D_Z+\varepsilon.H)=(E"+E^{\bullet})-f^*(K_Z+D_Z+\varepsilon.H):=P$ is pseudo-effective. Let $C\subset X$ be a complete intersection of ample divisors avoiding $g(E')$, and $C'$ its inverse image in $X'$. Then: $$g_*(C').(K_Z+D_Z)=(E"+E^{\bullet}).C'-P.C'-\varepsilon.g_*(C').H\leq -\varepsilon.g_*(C').H<0,$$ since $E".C'=E^{\bullet}.C'.C'=0\leq P.C'$. This implies that $-(K_Z+D_Z)$ is not pseudo-effective, since $[g_*(C')]\in Mov(Z)$. Contradiction.\end{proof}

When $-(K_X+D)$ is nef instead of ample, one still gets, in general:

\begin{theorem}\label{t-knef} Let $(X,D)$ be a smooth projective orbifold such that $-(K_X+D)$ is nef, and let $(Z,D_Z)$ be the orbifold base of any neat orbifold birational model of any rational dominant fibration $f:X\dasharrow Z$. Then:

$\kappa(Z,K_Z+D_Z)=\nu(Z,K_Z+D_Z)=0$ if $(K_Z+D_Z)$ is pseudo-effective. 
 \end{theorem}

 From this, one immediately gets:

 \smallskip

\begin{corollary}\label{cfnu} Let $(X,D)$ be a smooth projective orbifold with $-(K_X+D)$ nef. Then:

 1. If $(X',D')\to (R,D_R)$ is its `slope rational quotient', either $(X,D)$ is $sRC$, or $\kappa(R,K_R+D_R)=\nu(R,K_R+D_R)=\nu^+(X,D)=0$.
 
 2. $\nu^+(X,D)\in \{-\infty,0\}$. 
 
 3. If $(X,D)$ is Fano, then $X$ is rationally connected.
\end{corollary}

\begin{corollary}\label{cfnu+} Let $(X,D)$ be a smooth projective orbifold with $-(K_X+D)$ nef. The following conditions are equivalent:

 1. $(X,D)$ is sRC.
 
 2.  $h^0(X,[Sym^m(\Omega^p)](X,D)))=0,\forall p>0,m>0$.
 
\end{corollary}

\begin{remark}Theorem \ref{t-knef} extends to the smooth orbifold case the results of \cite{QZ}. Using Theorem 2.11 of \cite{CP13}  instead of Log-subajunction techniques simplifies considerably the proof. Corollary \ref{cfnu+} extends to the orbifold case the result of \cite{Cao}, implied by \cite{QZ}, with a different proof.
 \end{remark}

\begin{proof} (or Corollary \ref{cfnu+}): The first property obviously implies the second. Assume now the second property to hold, but not the first. Then by Theorem \ref{t-knef}, $h^0(R,m.(K_R+D_R))\neq 0$ for some large divisible $m>0$. But a nonzero section of $m.(K_R+D_R)$ lifts through $r$ to a nonzero section of $[Sym^m(\Omega^p)](X,D))$, with $p:=dim(R)>0$.
\end{proof}

\begin{proof} (of Theorem \ref{t-knef}) Write, as above: $-N':=g^*(K_X+D)=K_{X'}+D'+(\Delta')+A'-E^{\bullet}$, with $N'$ nef by hypothesis. Thus $N'+\beta.A'$ is ample on $X'$ for any $\beta>0$ rational, arbitrarily small, and $A'$ a polarisation on $X'$. We choose a $\Bbb Q$-divisor $B'$ linearly equivalent to $N'+\beta.A'$ and $f$-horizontal, such that $D":=D'+E'+B'$ is snc. Thus $K_{X'}+D"=K_{X'}+D'+E'+B'\equiv E'-\Delta'+E^{\bullet}+\beta.A'$. We then get, for the orbifold base $(Z,D_Z)$ of $f:(X',D")\to Z,$ from \cite{CP13}, 2.11 again: $K_{X'/Z}+D"-f^*(D_Z):=P_{\beta}$ is pseudo-effective. Note that $D_Z$ does not depend on $\beta>0$ since $B'$ is choosen to be $f$-horizontal. Thus $P_{\beta}\equiv E"+E^{\bullet}-f^*(K_Z+D_Z)+\beta.A'$, with $E":=E'-\Delta'$ effective and $g$-exceptional. Letting $\beta\to 0^+$, we get: $f^*(K_Z+D_Z)=E"+E^{\bullet}-P$, where $P$ is a pseudo-effective class, as limit of the $P_{\beta}$. We thus deduce: $\nu(X',f^*(K_Z+D_Z))=\nu(K_Z+D_Z)\leq \nu(X', E"+E^{\bullet})\leq 0$, and thus: $\nu(Z,K_Z+D_Z)=0$ if $K_Z+D_Z$ is pseudo-effective, in which case the equality $\kappa=\nu$ is due to \cite{K}. This shows the Theorem. \end{proof}

\begin{remark} At the end of the proof of Theorem \ref{t-knef}, we obtain a slightly more precise numerical information: choose the curves $C,C'$ as in the proof of Theorem \ref{tfsrc}. If $K_Z+D_Z$ is pseudo-effective, we have: $0\leq f^*(K_Z+D_Z).C'= (E"-P).C'=-P.C'\leq 0$.
 \end{remark}

\begin{example} Let us give an example of a $3$-dimensional smooth projective $(X,D)$ with $-(K_X+D)$ nef and big, $D$ reduced, such that its `rational quotient' $f:(X,D)\to (Z,D_Z)$ is a fibre bundle with fibres $\Bbb P^1$, and $\pi_1(Z,D_Z)\cong \Bbb Z_3$. Just start with $(X_0=\Bbb P^3, D_0)$, where $D_0$ is the cone over an elliptic curve $E$: $-(K_{X_0}+D_0)$ is thus ample. Blow-up the singular point of $D_0$, to obtain $X$, let $D'$ be the exceptional divisor, and let $F$ be the strict transform of $D_0$. Take then $(X, D:=F+D')$: we have: $K_X+(F+D')=g^*(K_{X_0}+D_0)$, so that the anticanonical bundle of $(X,D)$ is nef and big. Moreover, the natural linear projection $f:X\to \Bbb P^2=Z$ sends $F$ to $E$, with $f^{-1}(E)=F$, while $D'$ is $f$-horizontal. Thus $f$ is an orbifold morphism from $(X,D)$ to $(Z,D_Z)=(\Bbb P^2,E)$, its orbifold base. We obviously have: $K_Z+D_Z=\cO_Z$, and $Z=\Bbb P^2$ has a $K3$ cyclic cover of degree $3$ ramifying exactly over $E$, showing the claim on the orbifold fundamental group, from which follows that $h^0(Z, \Omega^1(Log E))=\{0\}$.
\end{example}

\begin{question} Let $(X,D)$ be projective smooth with $-(K_X+D)$ nef. Is it true that:

1. There exists $(X',D'), f,g, (Z,D_Z)$ as above such that $K_Z+D_Z$ is numerically trivial, and trivial if $D$ is reduced?

2. Assume also that $D$ is reduced, and $h^0(X, \Omega^q_X(Log(D)))=0,$ $ \forall q>0$. Is then $(X,D)$ $sRC$? (This is true if the answer to question 1 is affirmative).
\end{question}

\begin{remark} The proofs of the results in this section apply, more generally, to any smooth projective orbifold pair $(X,D)$ which is the log-resolution of a singular log-canonical projective orbifold pair $(X_0,D_0)$ with $-(K_{X_0}+D_0)$ is nef. More specifically, in this case,  the conclusions of Theorem \ref{t-knef} and corollaries \ref{cfnu} and \ref{cfnu+} still hold, that is: either $(X,D)$ is sRC, or $\nu^+(X,D)=\kappa(R,D_R)=0$, the first case occuring if $h^0(X,[Sym^m(\wedge^p](\Omega^1(X,D))=0,\forall m>,p>0$ . The conclusion of Theorem \ref{tfsrc} also holds: if $(X_0,D_0)$ is Fano and klt, then $(X,D)$ is sRC. We only gave above the statements when $(X,D)=(X_0,D_0)$.
\end{remark}


\section{Orbifold rational curves: conjectures}\label{sorc'}

In this last section, we introduce the notions of orbifold rational curves, orbifold uniruledness, and rational connectedness, extracted from \cite{Ca09} and \cite{Ca10}, to which we refer for more details and justifications for the notions introduced here. We then state the conjectures stating the equivalence between `slope rational connectedness' and orbifold rational connectedness, again as in loc. cit.. A much weaker property will then be shown in the next section \ref{sorc}.

In the present section, $(X,D)$ will be a smooth, connected, and complex-projective orbifold pair, with $D=\sum_{j\in J} c_j.D_j$, the `coefficients' $c_j:=(1-\frac{1}{m_j})\in ]0,1]$ being rational numbers, with the `multiplicities' $m_j:=\frac{a_j}{b_j}=(1-c_j)^{-1}\in \Bbb Q$, with $0<a_j\leq b_j$ being coprime integers. We also consider another smooth projective connected orbifold pair $(C,D_C)$, in which $C$ is a curve (so that $D_C:=\sum_{k\in K}c'_k.{p_k}$, the $p_k's$ being distinct points of $C$, and $c'_k\in \Bbb Q\cap ]0,1], \forall k\in K$). Write again: $c'_k:=(1-\frac{1}{m'_k})$ for the corresponding multiplicities (on $C$).

Recall the definition of an orbifold morphism in our context (see also Definition \ref{deforbmorph}):

\begin{definition}\label{deforc} (\cite{Ca10}, Definition 9) Let $f:C\to X$ be a morphism. We say that $f:(C,D_C)\to (X,D)$ is an orbifold morphism if:

1. $f$ is birational from $C$ to $f(C)$.

2. $f(C)$ is not contained in $Supp(D)$.

3. For each $a\in C$ and each $j\in J$ such that $f(a)\in D_j$, one has: $t_{a,j}.m_{a}(D_C)\geq m_j$, where $t_{a,j}$ be the order of contact of $f(C)$ and $D_j$ at $f(a)$, that is: $f^*(D_j)=t_{a,j}.\{a\}+....$, and $m'_a=m_{a}(D_C)$ is the multiplicity (not the coefficient!) of $a$ in $D_C$. \end{definition}

An immediate computation (or an application of Proposition \ref{pom}) shows that $f:(C,D_C)\to (X,D)$ is an orbifold morphism if and only if : $df:TC\to f^*(TX)$ induces an injection of sheaves $g^*(df):g^*(T(C,D_C))\to g^*(f^*(T(X,D)))$, for any finite map $g:C'\to C$ such that $C'$ dominates a Kawamata cover of $C$ adapted to $D_C$, and such that $f\circ g:C'\to X$ factorises through some Kawamata cover $\pi:Y\to X$ adapted to $D$. 

This former definition is thus justified by this property. In the texts \cite{Ca09}, \cite{Ca10}, a similar justification was given in terms of (slightly different) variants of the orbifold tangent bundle of $(X,D)$ (and for morphisms from higher-dimensional $C's$, not necessarily curves). 

Notice also that, for given $X,D,C,f$ as above, there is smallest orbifold divisor $D_C$ for which $f:(C,D_C)\to (X,D)$ is an orbifold morphism. For each $a\in C$, it attributes to the point $a$ the coefficient $c_{(f,a)}=(1-\frac{1}{m_{(f,a)}})$, with $m_{(f,a)}:=inf\{t_{a,j}.m_j\}$, where $j\in J$ is such that $f(a)\in D_j$. Unless explicitly said, we shall, in the sequel, always consider as $D_C$ this minimal orbifold structure making $f$ an orbifold morphism.

\begin{definition}(\cite{Ca10}, Definition 9 and Definition 10) Let  $f:(C,D_C)\to (X,D)$ be an orbifold morphism, as above. Then $(C,D_C)$  is said to be $D$-rational curve if $deg(K_C+D_C)<0$. This clearly implies that $C\cong \Bbb P^1$.
\end{definition}

\begin{example} Assume that $(X,D)$ is a smooth projective orbifold with $D$ reduced (ie: all coefficients of the components of $D$ are $1$, or equivalently, their multiplicities are $+\infty$). A $D$-rational curve $f:C=\Bbb P^1\to X$ is thus a rational curve of $X$ either contained in $X-D$, or such that a single point of $C$ is sent to $D$. For example, an ordinary double point of $f(C)$ lying on $D$ is excluded.
\end{example}

\begin{definition}(\cite{Ca10}, Definition ) We say that $(X,D)$ is uniruled (resp. Rationally connected) if there exists an irreducible orbifold rational curve going through any generic point (resp. any generic pair of points) of $X$. We say that $(X,D)$ is `weakly uniruled' if, through the generic point of $X$, there exists an irreducible rational curve $C$ with $(K_X+D).C<0$.
\end{definition}

Of course, there are many stronger variants of this notion. See \cite{Ca07} for some of them.

We can now state the main conjecture concerning orbifold rational curves (see \cite{Ca10}, for similar conjectures, but related to $\kappa$, rather than $\nu$):

\begin{conjecture}\label{corat} Let $(X,D)$ be a smooth projective \footnote{Or Compact K\"ahler, or in the class C.} orbifold pair. 

1. $(X,D)$ is uniruled if and only if $K_X+D$ is not pseudo-effective.

1'. $(X,D)$ is uniruled if and only if it is weakly uniruled \footnote{This conjecture is equivalent to Question 9.1 below.}. 

2. $(X,D)$ is rationally connected if and only if it is `slope rationally connected'.
\end{conjecture}

\begin{remark}\label{rorat} One has the following easy implications, for a smooth projective $(X,D)$:

1. Uniruled $\Longrightarrow$ weakly uniruled $\Longrightarrow$ $K_X+D$ not pseudo-effective.

2. Rationally connected $\Longrightarrow$ slope Rationally connected.
\end{remark} 

The reverse implications are known only in very special cases when $n>1$:

3. When $(X,D)$ is Fano, with $D$ reduced (ie: the coefficients of the $D_j's$ are all equal to $1$). Then $(X,D)$ is uniruled (\cite{K-MK}). 
This means that $X$ is covered by rational curves $C$ meeting $D$ in at most one point (on the normalisation of $C$). The example of $\Bbb P^2$ with a reduced divisor consisting of $2$ lines shows that this is optimal in the sense that, as seen above, this (l.c but not klt) orbifold is Fano, but not $sRC$, and so cannot be rationally connected.

4. For $X=\Bbb P^n$, with $D$ consisting of $(n+2)$ hyperplanes with integer multiplicities $m_j,j=0,...,n+1$, such that $(\Bbb P^n,D)$ is Fano, it is shown that $(\Bbb P^n,D)$ is uniruled (\cite{Ca10}, Theorem 8). The arguments given there have no general character, however. But a counting dimension argument shows that this should `heuristically' hold. We refer to \cite{Ca10}, \S 7 for more similar examples and a more detailed discussion.


\section{Orbifold Rational curves:  a weak conditional version}\label{sorc}

Let $(X,D)$ be slope-rationally connected in the sense that there exists a movable classe $\alpha$ such that $\mu^G_{\pi^*(\alpha),min}(\pi^*(T(X,D))>0$ for some (or any) adapted Kawamata-cover $\pi:Y\to (X,D)$. 

The aim of this section is to improve Theorem \ref{torc} to show that the class $\alpha$ can be choosen to be `Geometrically Rational big' in the sense below, if $(X,D)$ is klt and if the following question has a positive answer:

\begin{question} \label{qour1} \emph{Let $(X,D)$ be a connected and smooth complex-projective orbidold pair such that $K_X+D$ is not pseudo-effective. Does there exist on $X$ an algebraic family of rational curves $(C_t)_{t\in T}$, parametrised by an irreducible projective variety $T$, whose generic member is irreducible, and such that $-(K_X+D).C_t>0$}? \end{question}

When $(X,D)$ is klt, the known results of the LMMP permit to reduce this question to the following special case:

\begin{question}\label{qour2} \emph{Let $(X,D)$ be a connected and smooth klt complex-projective orbidold pair such that, after a finite number of divisorial contractions and Log-flips $\varphi:(X,D)\to (X',D')$, with $D'=\varphi_*(D)$, $(X',D')$ is still klt, with $X'$ $\Bbb Q$-factorial, and $K_{X'}+D'$ is Fano with Picard number $1$. Does there exist on $X$ an algebraic family of rational curves $(C_t)_{t\in T}$, parametrised by an irreducible projective variety $T$, whose generic member is irreducible, and such that $-(K_X+D).C_t>0$?} \end{question}

\

The problem is to lift a suitable covering family of rational curves on $(X',D')$ to $X$ preserving the negativity of the intersection number with $K+D$. Note that, in general, there is no covering family of rational curves on $X'$ avoiding the non-canonical singularities of $(X',D')$. It might, however, be true that the canonical singularities of this pair can be avoided.

More might possibly be said in order to characterise slope-Rational connectedness by rational curves (see Question 2 at the end of this section), in analogy with the case when $D=0$.

The sections \S\ref{grc} and \S\ref{bcrit} below are used below only in order to get the `bigness' statement in Theorem \ref{Torc}.


\subsection{Geometrically Rational classes}\label{grc}

\begin{definition}\label{gr} Let $X$ be smooth, complex projective and connected, and $\alpha\in Mov(X)$. We say that $\alpha$ is `Geometrically Rational' if it belongs to the closed cone $RMov(X)$ generated by classes of the form $[C]$, for an irreducible rational curve $C$ on $X$ belonging to an $X$-covering algebraic family of rational curves parametrised by an irreducible projective variety. We say that $\alpha$ is `Geometrically Rational big' if it belongs to the interior $[RMov^0(X)]$ of this cone.
\end{definition}

\begin{remark} $RMov(X)$ is a strict closed subcone (of nonempty interior, see below) of $Mov(X)$. In general, the interior $RMov^0(X)$ of $RMov(X)$ is contained, but not equal to $[RMov(X)\cap Mov^0(X)]$, even if $X$ is a rational surface.  An exemple is $X$, given by $\Bbb P^2$ blown-up in $16$ points. If $C$ is the strict transform of a generic quintic through these $16$ points, then $C$ is easily seen to be ample, but $K_X.C=+1$, which shows that $[C]\notin RMov(X)$, since $-K_X.\alpha\geq 0$ for any $\alpha\in Mov(X)$. The boundary of $RMov(X)$ thus meets $Mov^0(X)$.
\end{remark}

We denote by $K^{<0}(X)$ (resp. $K^{\leq 0}(X)$ the cone of classes in $N_1(X)$ which have negative (resp. non-positive) intersection with $K_X$. 

\begin{proposition}\label{bratclass} Let $X$ be a connected complex projective manifold. Then:

1. $RMov(X)$ is non-empty if and only if $X$ is uniruled.

2. $Mov(X)\cap K^{<0}(X)$ is non-empty if and only if $X$ is uniruled.

3. $RMov(X)\cap Mov^0(X)$ is nonempty if and only if $X$ is rationally connected.

4. $RMov(X)$ has nonempty interior in $N_1(X)$ if and only if $X$ is rationally connected.
\end{proposition}

\begin{question} Is it true that:

1. $Rmov(X)=Mov(X)\cap K_X^{\leq 0}$ if $X$ is uniruled?

2. $Rmov^0(X)=Mov^0(X)\cap K_X^{<0}$ if $X$ is rationally connected?\end{question}

\begin{proof} (of Proposition \ref{bratclass}): Claim 1 is essentially the definition of uniruledness. Claim 2 is \cite{MiMo}.

Claim 3. Assume first that $\alpha\in RMov(X)\cap Mov^0(X)$. Let then $C$ be some irreducible rational curve on $X$ such that $[C]\in [RMov(X)\cap Mov^0(X)]$, and belonging to an algebraic family of $X$-covering rational curves $C_t,t\in T$. Assume $X$ were not rationally connected. Let then $r:X\dasharrow R$ be its `rational quotient' (which is an almost holomorphic fibration). We then have: $r_*([C])=0$. If $H\subset R$ is an effective non-zero diviseur, then $H_X:=f^{-1}(H)$ does not meet the generic $C_t$, and $H_X.[C]=0$, contradicting the fact that $[C]\in Mov^0(X)$. Thus: $X$ is rationally connected if $RMov(X)\cap Mov^0(X)\neq \emptyset$.

In the other direction, assume that $X$ is rationally connected. Let $H$ be an ample divisor on $X$. There thus exists an integer $d>0$ such that any two points of $X$ can be joined by an irreducible rational curve $C$ of degree at most $d$ (by the `comb deformation technique' of \cite{KMM}, see \cite{D}, Theorem 4.27, p. 105). Let then $[C_s],s=1,...,N$ be the classes of irreducible rational curves of degree at most $d$ belonging to irreducible algebraic families of rational curves connecting two generic two points of $X$, and let $\alpha:=\sum_{s=1}^{s=N}[C_s]$. Then $\alpha.D>0$, for any irreducible effective non-zero divisor $D\subset X$, since $[C_s].D\geq 0, \forall s$, and $[C_s].D>0$ for at least one $s$ (a priori depending on $D$), by choosing one of the two points outside of $D$, and $C_s$ irreducible\footnote{It is possible to choose a single class $[C_s]$ by deforming the curve $C_1+\dots+C_N$.}.Thus $\alpha\in RMov^0(X)$, by the Corollary \ref{ratbig} below.

Claim 4. Assume that $X$ is not rationally connected. Claim 3 above then shows that $Rmov(X)\cap Mov^0(X)=\emptyset$. A fortiori, $Rmov(X)\subset Mov(X)$ has no interior point in $Mov^0(X)$.

Conversely, assume that $X$ is rationally connected, and let $\alpha\in Rmov(X)\cap Mov^0(X)$ be as above. It follows again from the `comb deformation technique' of \cite{KMM} (see \cite{D} for example), that $\alpha+\varepsilon [\Gamma]\in RMov(X)$ for any irreducible rational curve $\Gamma\subset X$ if $\varepsilon>0$ is sufficiently small. But the $K_X$-negative part $K^{<0}(X)$ of the closed cone $NE(X)$ of effective curves on $X$ is generated by classes of Mori-extremal rational curves on $X$, by the `cone theorem' (see [D] for example). Since $K^{<0}(X)$ has nonempty interior in $NE(X)$, any class of the form $\alpha-\varepsilon [\Gamma]$ as above belongs to the interior of $RMov(X)$ in $NE(X)$ if $\Gamma$ is a Mori-extremal rational curve on $X$.
\end{proof}


\subsection{Bigness of movable classes.}\label{bcrit}

Let $X$ be smooth, connected and complex projective. Let $(C_v)_{v\in V}$ be an algebraic family of curves parametrised by a complex projective irreducible space $V$. We assume that $C_v$ is irreducible, for $v\in V$ generic, and that the family is $X$-covering. The class $[C:=C_v]\in Mov(X)$ is thus independent of $v$. The closed cone of $H_2(X,\Bbb R)$ generated by such classes is $Mov(X)$. Recall that a  class $\alpha\in Mov(X)$ is `big' if it lies in $Mov^0(X)$, the interior of the cone $Mov(X)$. This notion is obviously not preserved by blow-ups. We have the obvious:

\begin{lemma}\label{lpseff} $\alpha\in Mov(X)$ is big if and only if, equivalently:

1. $\alpha.P>0$ for any pseudo-effective divisor.

2. $\alpha-\varepsilon.A^{n-1}\in Mov(X)$ for some $\varepsilon>0$, if $A$ is some ample divisor on $X$. 
\end{lemma}

Let $(C_v)_{v\in V}$ be as above, and let $q:\bP(TX)\to X$ be the projectified (by lines) tangent bundle of $X$. Then each generic $C_v$ has a natural tangential lifting $\wh{C_v}\subset \bP(TX)$. We denote by $[\wh{C}]\in N_1(\bP(TX))$ the corresponding class. then $[\wh{C}]\in Mov(\bP(TX))$ if, through the generic point $x\in X$ and the tangent direction $\tau\in TX_x$, there exists some $C_v$ going through $x$ with tangent direction $\tau$. 

\begin{theorem}\label{tbig} Assume that $[\wh{C}]\in Mov(\bP(TX))$. Assume that there exists a non-zero pseudo-effective divisor $P\in N^1(X)$ such that $[C].P=0$. Then $P$ is $\bQ$-effective.
\end{theorem} 

The proof will be given in \S \ref{proofbig} below.

\begin{corollary}\label{cbig} The class $[C_v]$ is big if $[\wh{C_v}]\in Mov(\bP(TX)),$ and either:

1. $\wh{C_v}\in \bP(TX)$ is `strictly movable' in the sense that any point $z\in \bP(TX)$ is contained in some irreducible $\wh{C_v}$, or:

2. $[C_v].D>0$, for any irreducible divisor $D\subset X$.
\end{corollary}

From this we get immediately:

\begin{corollary}\label{ratbig} Let $[C_s]\in Mov(X), s=1,...,N$, be such that $[\wh{C_s}]\in Mov(\bP(TX)),\forall s$. If $\alpha:=\sum_{s=1}^{s=N}t_s.[C_s],$ $0\leq t_s\in \Bbb R, \forall s$ be such that $\alpha.D>0$ for any effective non-zero divisor $D\subset X$. Then $\alpha\in Mov^0(X)$.
\end{corollary}





\subsection{Slope positivity relative to Rational classes.}

\

\

We can now improve Theorem \ref{torc} as follows, in the klt case, but assuming a positive answer to Question 1 above:

\begin{theorem}\label{Torc} Let $(X,D)$ be a smooth connected klt complex projective orbifold pair which is slope Rationally connected. Assume that the question 1.2 at the beginning of the present section has a positive answer. 

There then exists a `Geometrically Rational big' class $\alpha$ such that $\mu_{\alpha,min}(\pi^*(T(X,D)))>0$.
\end{theorem}

\begin{proof} The proof is modeled after that of Theorem \ref{torc}, and so also works by induction on $n:=dim(X)$. When $n=1$, the statement is clear, since every big class class is then Geometrically Rational. We now assume that the assertion holds true for every $n'<n$. We consider two cases (possibly replacing $(X,D)$ by some of its orbifold birational models:

1. There exists $\alpha \in Mov^0(X)$ and $\{0\}\subsetneq \cF\subsetneq \pi^*(T(X,D))$, saturated, such that $\mu_{\alpha}(\cF)>0$.

2. For every $\alpha \in Mov^0(X)$ and any saturated $\{0\}\subsetneq \cF\subsetneq \pi^*(T(X,D))$, we have: $\mu_{\alpha}(\cF)\leq 0$.

Case 1: We then get (after replacing $(X,D)$ by a suitable orbifold birational model), and choosing $rank(\cF)$ minimum, a neat fibration which is an orbifold morphism to its orbifold base $f:(X,D)\to (Z,D_Z)$. And $(Z,D_Z)$ is still $sRC$, as well as the smooth orbifold fibres $(X_z,D_z)$ of $f$. We have: $0<d:=dim(Z)<n$, and thus get, as in the proof of Theorem \ref{torc}, classes $\alpha\in Mov^0(X/Z)$ and $\beta\in Mov^0(Z)$ such that $\mu_{\beta,min}(p^*(T(Z,D_Z)))>0$, $f_*(\alpha)=0$, and $\mu_{\alpha,min}(\pi^*(T(X_z,D_z)))>0$. We may, by induction on the dimension, assume that $\beta\in RMov^0(Z)$, and $\alpha_z\in RMov(X_z)$, where $\alpha_z$ is the restriction to the `general' smooth fibre $X_z$ of $f$.

The conclusion in this case 1 then follows from (the proof of Theorem \ref{torc} and) the following:

\begin{lemma} Let $f:X\to Z$ be a fibration between connected complex projective manifolds, together with $\alpha\in Mov(X/Z)$ and $\beta\in RMov^0(Z)$ such that $f_*(\alpha)=0$ and $\alpha_z\in RMov^0(X_z)$. Then:

1. There exists $\beta'\in RMov(X)$ such that $f_*(\beta')=\beta$, and:

2. $\gamma:=\varepsilon.\rho+k.\alpha+\beta'\in RMov^0(X)$ for any $k>0$ sufficiently large and $\varepsilon>0$ sufficiently small, if $\rho\in RMov^0(X)$, which is nonempty by Proposition \ref{bratclass}, since $X$ is rationally connected.
\end{lemma}

\begin{proof} Claim 1: We may reduce to the case where $\beta=[C_t]$ is the class of a $Z$-covering algebraic family of irreducible rational curves $[C_t], t\in T$. Because the fibres $f$ are rationally connected, it follows from \cite{GHS} that for $t$ generic, there exists a section $C'_t$ of $f_t:X_t:=f^{-1}(C_t)\to C_t$. Let $\alpha_t$ be the restriction of $\alpha$ to $X_t$, which makes sense, since $f_*(\alpha)=0$. From \cite{KMM}, we deduce that $\beta'_t:=k_t.\alpha_t+[C'_t]$ is in $RMov^(X_t)$, for $k_t>0$ sufficiently large, and is thus the class of an $X_t$-covering family of rational curves of $X_t$, with $(f_t)_*(\beta'_t)=[C_t]$. From the countability at infinity of the Chow-Barlet space of curves of $X$, we deduce the existence of a $k>0$ such that $\beta':=k.\alpha+\beta\in RMov(X)$ is such that $f_*(\beta')=\beta$, which is Claim 1. 

Claim 2: Since $\rho\in RMov^0(X)$ and $\beta'\in RMov(X)$, $\varepsilon.\rho+\beta'\in RMov^0(X)$ for any $\varepsilon >0$ sufficiently small. On the other hand, we get from the proof of Theorem \ref{torc} that $\mu_{\pi^*(k.\alpha+\beta'),min}(\pi^*(T(X,D)))>0$ for $k>0$ sufficiently large. It now follows that this remains true for $\varepsilon.\rho+k.\alpha+\beta'$, by Lemma \ref{movbig}.
\end{proof}

Case 2: We seem to need to use \cite{BCHM}, here. Let $\psi:X\dasharrow X_0$ be a sequence of divisorial contractions and log-flips, with $(X_0,D_0)=(X_0,\psi_*(D))$, such that one has a Log-Fano-contraction $\varphi:(X_0,D_0)\to Z$ with $(n-d):=dim(Z)<n$, and of relative Picard number $1$. By taking a suitable orbifold birational model of $(X,D)$, we shall assume that $\psi$ is regular. Choose $\alpha_0:=(-(K_{X_0}+D_0))^{d-1}.H^{n-d}$, where $H=\varphi^*(H_Z)$, with $H$ ample on $Z$. This is a movable curve class on $X$ with $\varphi_*(\alpha_0)=0$. Let $\psi^*(\alpha_0):=\alpha\in Mov(X)$ be its inverse image. We thus have: $f_*(\alpha)=0$, if $f=\varphi\circ \psi:X\to Z$. Let $\cF:=\pi^*(Kerd(df))\cap \pi^*(T(X,D))$, if $\pi:Y\to (X,D)$ is a Kawamata-cover adapted to $(X,D)$. Then $\mu_{\alpha}(\cF)>0$, since $f_*(\alpha)=0$, and $-\alpha.(K_{X_z}+D_z)>0$, if $X_z$ is a generic fibre of $f$. We are thus in the first case, unless $\cF=\pi^*(T(X,D))$. Since we are in case 2, $dim(Z)=0$, that is: $(X_0,D_0)$ is Fano of Picard rank $1$, $\pi^*(T(X,D))$ is semi-stable with respect to $\pi^*(\alpha)$, and $\mu_{\pi^*(\alpha),min}(\pi^*(T(X,D)))>0$. Question 1.2 having a positive answer, Moreover, $(X_0,D_0)$ being Log-Fano of Picard rank $1$, $X_0$ is covered by an algebraic family of rational curves $C'_t$ of class (proportional to) $\alpha_0$ such that $-C'_t.(K_{X_0}+D_0)>0$. Question 1.2 having by assumption a positive answer, there exists a class in $\alpha'\in RMov(X)$ with $-(K_X+D).\alpha'>0$. The properties shown above for $\alpha$ still hold for $\alpha'$ also, and in particular: $\pi^*(T(X,D))$ is semi-stable with respect to $\pi^*(\alpha')$, with:  $\mu_{\pi^*(\alpha'),min}(\pi^*(T(X,D)))>0$.
Since $X$ is rationally connected, there exists $\rho\in RMov^0(X)$, and $\varepsilon.\rho+\alpha'\in RMov^0(X)$, enjoying these same properties, by the same argument as in case 1, which also concludes the proof in this case.\end{proof}

\begin{question} 1. Can the big rational class $\alpha$ constructed above be choosen of the form $[C]$, for an irreducible rational curve $C\subset X$ with arbitrary ample normal bundle, and going through any given finite set of points?

2. Can the rational curve $C$ above be so chosen that $\pi^*(T(X,D))_{\vert C'}$ is ample, if $C'=\pi^{-1}(C)$? This might be the right definition of a `free' orbifold $D$-rational curve.\end{question}

\begin{remark} The question 2 might possibly depend on a version of the Grauert-M\"ullich restriction theorem for the curves of the form $C'$ above and for the vector bundle $\pi^*(T(X,D))$ on them.
\end{remark}

\subsection{Proof of the bigness criterion.} \label{proofbig}

We prove here the following result, used above:

\begin{theorem}\label{tbig} Assume that $[\wh{C}]\in Mov(\bP(TX))$. Assume that there exists a non-zero pseudo-effective divisor $P\in N^1(X)$ such that $[C].P=0$. Then $P$ is $\bQ$-effective.
\end{theorem}

\begin{proof} Let $p: Y\to X$ be a surjective map with connected fibers between two smooth 
compact manifolds $Y$ and $X$ of dimension $n+1$ and $n$, respectively. Given a generic point $y\in Y$, we denote by 
$C_y$ the fiber
$p^{-1}\big(p(y)\big)$ of $p$ passing through $y$; if more than one map $p$ is involved,  indices are used in order to distinguish the corresponding fibers.
\medskip

\begin{proposition}\label{currents}(Communicated by M. P\u aun)
Let $Y$ be a $n+1$-dimensional (smooth) compact complex manifold, and let $T$ be a closed positive ${\rm (1,1)}$ current on $Y$. Let surjective maps $p_j: Y\to X_j$ be given, where $X_j$ is an $n$-dimensional compact manifold for $j=1,\dots n+1$ having the following properties.
\begin{enumerate}

\item[(i)] There exists a proper analytic set $S\subset Y$ such that for each $y\in Y\setminus S$ the vector space generated by the tangent space of the curves $C_{j, y}$ at $y$ for $j=1,\dots n+1$ is $T_{Y, y}$. 

\item[(ii)] The restriction of the current $T$ to each generic fiber of $p_j$ is equal to zero, for each $j=1,\dots n+1$.
\end{enumerate}
Then we have $\chi_{Y\setminus S}T= 0$.
\end{proposition}

\begin{proof}
Let $S\subset Y$ be an analytic subset of $Y$, such that the restriction of each $p_j$ to the complement $Y_0:= Y\setminus S$ is a smooth, proper fibration. We show next that we have $\chi_{Y\setminus S}T= 0$.

Let $y\in Y_0$ be an arbitrary point, and let $z_1,\dots z_{n+1}$ be a set of coordinates at $y$ 
such that for each $j=1,\dots n+1$ the subspace 

\begin{equation}\label{2}
\bC\frac{\partial}{\partial z_j}
\end{equation}
coincides with the tangent space of $C_{j, y}$ at $y$. The choice of such a coordinate system
is possible, due to the hypothesis ${\rm (i)}$ above. 

Locally near $y$ the current $T$ can be written as
\begin{equation}\label{3}
T|_{\Omega}= \sum_{j,k}T_{j\ol k}dz_j\wedge d\ol {z_k}
\end{equation}
where $T_{j\ol k}$ are distributions of order zero on $\Omega$. Let $p: Y\to X$ be one of the maps above; we recall the following formula of Fubini type (cf. \cite{Siu74})
\begin{equation}\label{4}
\int_{\Omega}T\wedge p^*\eta= \int_{x\in X}\eta\int_{\Omega\cap p^{-1}(x)}T,
\end{equation}
where the restriction $\displaystyle T|_{p^{-1}(x)}$ is well defined for almost all $x\in X$, so that the right hand side member in \eqref{3} is meaningful. In \eqref{5} we denote by $\eta$ a smooth form of type $(n, n)$ defined (at least) in a open set 
including $p(\Omega)$.

By the implicit function theorem, there exist $(\eta_j)_{j=1,\dots n+1}$ a set of smooth $(n, n)$ forms defined in a small open set centered at $p_j(y)$ in $X_j$ such that we have 
\begin{equation}\label{5}
p_j^*(\eta_j)= \xi_1\wedge \dots \xi_{j-1}\wedge\xi_{j+1}\wedge\dots \xi_{n+1}
\end{equation} 
where we use the notation $\xi_j:= \sqrt{-1}dz_{j}\wedge d\ol {z_{j}}$.

The Fubini formula \eqref{3} combined with the hypothesis ${\rm (ii)}$ show that 
\begin{equation}\label{6}
\chi_{\Omega}T\wedge \xi_1\wedge \dots \xi_{j-1}\wedge\xi_{j+1}\wedge\dots \xi_{n+1}= 0
\end{equation}
for each $j= 1,\dots, n+1$. In other words, the diagonal distributions $T_{j\ol j}$ are identically zero, and so it is the restriction of $T$ to $\Omega$ (this is a consequence of the fact that $T$ is positive). 
\medskip

\noindent The current $T$ has thus no mass on $Y\setminus S$.
\end{proof}

\begin{remark}{\rm
1. The hypothesis ${\rm (ii)}$ above means that $0=\rho\cdot C_{j, y},\forall j$, if $\rho:= \{T\}$ is the cohomology class of $T$ in $H^{1,1}(X, \bR)$, since the cohomology class of $\displaystyle T|_{C_{j, y}}$ is the restriction of $\rho$ to $C_{j, y}$.

2. Proposition \ref{currents} implies that $\rho$ is \emph{effective}, ie: it contains an effective $\bR$-divisor, since 
$T= \chi_{Y\setminus S}T+ \chi_{S}T$,
and so: $T= \chi_{S}T$. The claim follows from the `support  theorem' (see \cite{Dem}).
}\end{remark}\end{proof}


\section{Some questions about the fundamental group}

Let $(X,D)$ be an smooth connected orbifold pair with $X$ complex projective\footnote{Or compact K\"ahler, or class $C$.} Assume that $D$ is `integral' (ie: all multiplicities of the componnents of $D$ are integral or $+\infty$).The fundamental group $\pi_1(X,D)$ is then defined. It was conjectured in \cite{Ca 07} that if $(X,D)$ is `special', then $\pi_1(X,D)$ is almost abelian. Because $sRC$ orbifold pairs are `special', we obtain:

\begin{conjecture} Let $(X,D)$ be as above. 

1. Then $\pi_1(X,D)$ is almost abelian, finite if $(X,D)$ is klt, and trivial if $(X,D)$ is either `purely logarithmic' (ie: all coefficients of the $D_j$ are $1$), or if $h^0(X,\Omega^1(X, Log(D_1))=0.$ Here $D(+\infty)$ is the union of the components of $D$ having multiplicity $+\infty$.

2. Same for $\pi_1^{alg}(X,D)$, the algebraic orbifold fundamental group. \end{conjecture} 

\begin{remark} 1. If $(X,D)$ is $sRC$, $X$ is rationally connected, hence simply-connected.

2. If we assume $(X,D)$ to be rationally connected, the conjecture holds true.

3. The methods ($L^2$-cohomology theory) used when $D=0$ do not seem to apply (immediately) in general.

4. The version for the algebraic fundamental group seems to be much more accessible.

\end{remark}


\section{Motivation: the decomposition of the Core map.}\label{motiv}

In \cite{Ca07}, Th\'eor\`eme 10.3, we showed the decomposition $c=(J\circ r)^n$ of the `core' map $c:(X,D)\to (C,D_C)$ of a smooth projective orbifold pair $(X,D)$. This decomposition was conditional in the `$C_{n,m}^{orb}$'-conjecture introduced in \cite{Ca04}, \S 4.1, p. 564. The $C_{n,m}^{orb}$ conjecture was used in order to define the map $r:(X,D)\to (R^*,D_{R^*})$, its `$\kappa$-rational quotient', for any smooth orbifold $(X,D)$, while $J:(X,D)\dasharrow (J,D_J)$ was a neat model of its `Moishezon-Iitaka fibration' when $\kappa(X,D)\geq 0$. The `Slope-Rational Quotient' $\rho:(X,D)\to (R,D_R)$ defined above permits to give (unconditionally) a variant of the `$\kappa$-rational quotient'. Conjecturally, these two maps actually coincide. We give some details below.

\begin{definition}\label{dk+} (\cite{Ca07}, D\'efinition 5.23, Remarque 5.24)\footnote{The definition (and notation) given there is slightly different, but should lead to the same invariants.}Let $(X,D)$ be a smooth (complex projective\footnote{Compact complex would suffice for the definitions,here}, connected) orbifold pair. Define:$$\kappa^+(X,D):=max\{\kappa(X',L')\vert m>0, L'\subset \otimes^m(\pi^*(\Omega^1(X,D)),rk(L')=1\},$$ and: $\kappa_+(X,D):=max\{\kappa(Z,D_Z)\vert f:(X',D')\to (Z,D_Z)\}$, where $(X',D')$ is birationally orbifold equivalent to $(X,D)$, and $f$ is a `neat' orbifold model of $(f,D)$.

To simplify notation, we write: $f:(X,D)\dasharrow (Z,D_Z)$ for a neat orbifold model of such a fibration.
\end{definition}

We obviously have: $\kappa(X,D)\leq \kappa_+(X,D)\leq \kappa^+(X,D)$.

In \cite{Ca07}, Corollaire 6.14, we showed, assuming $C_{n,m}^{orb}$, the existence of a unique fibration $r: (X,D)\dasharrow (Z,D_Z)$ such that its general orbifold fibres $(X_r,D_r)$ have $\kappa^+(X_r,D_r)=-\infty$, while its orbifold base $(R^*,D_{R^*})$ had $\kappa(R^*,D_{R^*})\geq 0$.

We now replace $\kappa$ by the numerical dimension $\nu$, which usually turns conjectures in theorems.

We showed in \cite{CP15}, that equality holds when we replace $\kappa(X',L')$ by the numerical dimension $\nu(X',L')$ if $K_X+D$ is pseudoeffective: $\nu(X,D):=\nu(X,K_X+D)=\nu^+(X,D)$, the latter being defined as the maximum of $\nu(X',L')$ for the same $L'$ as above.

Since: $\nu^+(X,D):=-\infty$ if and only if: $(X,D)$ is slope-rationally connected, or equivalently, if:
$h^0(X',\otimes^m(\pi^*(\Omega^1(X,D))\otimes A)=0$$,\forall k\geq k(A)$, by our main result here, the `slope-rational quotient'  $\rho:(X,D)\dasharrow (R_,D_R)$ defined above unconditionally should coincide with $r$. The problem one now faces is that $K_R+D_R$ is pseudoeffective, instead of having $\kappa(R^*,D_{R^*})\geq 0$, as one had with the orbifold base of $r$. One cannot however define any `Moishezon-Iitaka-fibration' for $(R,D_R)$ without assuming that $\kappa(R,D_R)\geq 0$, if $K_R+D_R$ is only known to be pseudoeffective.



\end{document}